%% file: main.tex
\documentclass[11pt]{article} % current default for manuscript submission
\usepackage{graphicx, amsmath,amsthm,amssymb,url}

\usepackage{mathtools}
\usepackage[multidot]{grffile}
\usepackage[normalem]{ulem} % just for strikeout

\usepackage{fullpage,etoolbox}
\usepackage{color}
\usepackage[square]{natbib}
\usepackage{hyperref}       % hyperlinks
\hypersetup{
    colorlinks=true,
    linkcolor=blue,
    citecolor = blue,
    urlcolor = blue
}

\newtheorem{theorem}{Theorem}
\newtheorem{lemma}[theorem]{Lemma}

\newtheorem{remark}{Remark}
\newtheorem{definition}{Definition}
\newtheorem{assumption}{Assumption}

\usepackage[algo2e,ruled]{algorithm2e}
\usepackage{makecell,booktabs,ctable}
\begin{document}
%%%%%%%%%%%%%%%%


% user defined commands

\newcommand{\guannan}[1]{\textcolor{cyan}{(Guannan: #1)}}
\newcommand{\adam}[1]{\textcolor{red}{(Adam: #1)}}
\newcommand{\lina}[1]{\textcolor{red}{(Lina: #1)}}
\newcommand{\guannanhighlight}[1]{\textcolor{blue}{#1}}
\newcommand{\revision}[1]{{\color{blue}#1}}

\newcommand{\diag}{\textup{diag}}
\newcommand{\td}{\textup{TD}}
\newcommand{\E}{\mathbb{E}}
\newcommand{\R}{\mathbb{R}}

\newcommand{\algoname}{\textsc{SAC}}
\newcommand{\algonamefull}{Scalable Actor Critic}

\newcommand{\khop}{{ \kappa}}
\newcommand{\rhok}{{ \rho^{\khop+1}}}
\newcommand{\fk}{{ f(\khop)}}
\newcommand{\nik}{{ N_i^{\khop}}}
\newcommand{\njk}{{ N_j^{\khop}}}
\newcommand{\nminusik}{{ N_{-i}^{\khop}}}
\newcommand{\nminusjk}{{ N_{-j}^{\khop}}}

\newcommand{\final}[1]{{\color{red}#1}}


% \TITLE{Bundling Information Goods of Decreasing Value}
\title{Scalable Reinforcement Learning for Multi-Agent \\ Networked Systems}

\author{Guannan Qu\footnote{Department of Electrical and Computer Engineering, Carnegie Mellon University. Email: \url{gqu@andrew.cmu.edu}}\qquad Adam Wierman\footnote{Department of Computing and Mathematical Sciences, California Institute of Technology. Email: \url{adamw@caltech.edu}} \qquad Na Li\footnote{School of Engineering and Applied Sciences, Harvard University. Email: \url{nali@seas.harvard.edu}}}
\date{}
 % end of the block
\maketitle

\begin{abstract}
    We study reinforcement learning (RL) in a setting with a network of agents whose states and actions interact in a local manner where the objective is to find localized policies such that the (discounted) global reward is maximized. A fundamental challenge in this setting is that the state-action space size scales exponentially in the number of agents, rendering the problem intractable for large networks. In this paper, we propose a \algonamefull\ (\algoname) framework that exploits the network structure and finds a localized policy that is an $O(\rhok)$-approximation of a stationary point of the objective for some $\rho\in(0,1)$, with complexity that scales with the local state-action space size of the largest $\khop$-hop neighborhood of the network. We illustrate our model and approach using examples from wireless communication, epidemics and traffic.

\end{abstract}

%%%%%%%%%%%%%%%%%%%%%%%%%%%%%%%%%%%%%%%%%%%%%%%%%%%%%%%%%%%%%%%%%%%%%%

% Samples of sectioning (and labeling) in OPRE
% NOTE: (1) \section and \subsection do NOT end with a period
%       (2) \subsubsection and lower need end punctuation
%       (3) capitalization is as shown (title style).
%
%\section{Introduction.}\label{intro} %%1.
%\subsection{Duality and the Classical EOQ Problem.}\label{class-EOQ} %% 1.1.
%\subsection{Outline.}\label{outline1} %% 1.2.
%\subsubsection{Cyclic Schedules for the General Deterministic SMDP.}
%  \label{cyclic-schedules} %% 1.2.1
%\section{Problem Description.}\label{problemdescription} %% 2.

% Text of your paper here

\section{Introduction}
The modeling and optimization of networked systems such as wireless communication networks and traffic networks is a long-standing challenge.  Typically analytic models must make numerous assumptions to obtain tractable models as a result of the complexity of the systems, which include many unknown, or unmodeled dynamics. Given the success of Reinforcement Learning (RL) in a wide array of domains such as game play \citep{silver2016mastering,mnih2015human}, robotics \citep{duan2016benchmarking}, and autonomous driving \citep{li2019reinforcement}, %computer vision \citep{caicedo2015active}, %\adam{I combined go and game play...we should have another category}, 
it has emerged as a promising tool for tackling the complexity of networked systems. However, when seeking to use RL in the context of the control and optimization of large-scale networked systems, scalability quickly becomes an issue.  The goal of this paper is to develop \emph{scalable} multi-agent RL for networked systems. 

Motivated by real-world networked systems like wireless communication, epidemics, and traffic, we consider an RL model of $n$ agents with \emph{local interaction structure}. Specifically, each agent $i$ has local state $s_i$, local action $a_i$ and the agents are associated with an underlying dependence graph $\mathcal{G}$ and interact locally, i.e, the distribution of $s_i(t+1)$ only depends on the current states of the local neighborhood of $i$ as well as the local $a_i(t)$. Further, each agent is associated with stage reward $r_i$ that is a function of $s_i, a_i$, and the global stage reward is the average of $r_i$. In this setting, the design goal is to find a decision policy that maximizes the (discounted) global reward. This setting captures a wide range of applications, e.g. epidemics \citep{epi_mei2017dynamics}, social networks \citep{application_chakrabarti2008epidemic,application_llas2003nonequilibrium}, wireless communication networks \citep{zocca2019temporal,application_communication}, queueing networks \citep{complexity_papadimitriou1999complexity}, smart transportation \citep{zhang2016control}, and smart building systems \citep{application_wu2016optimal,zhang2017decentralized}. %and multi-agent game play \citep{application_gameplay}.  

%Given the broad success of RL, it is exciting to consider the application of RL to networked systems.  % you have already said this in the first paragraph.
A fundamental difficulty when applying RL to such networked systems is that, even if individual state and action spaces are small, the entire state profile $(s_1,\ldots,s_n)$ and the action profile $(a_1,\ldots,a_n)$ can take values from a set of size exponentially large in $n$. This ``curse of dimensionality'' renders the problem unscalable. For example, most RL algorithms such as temporal difference (TD) learning or $Q$-learning require storage of a $Q$-function \citep{bertsekas1996neuro} whose size is the same as the state-action space, which is exponentially large in $n$. 
%Even if one can use various function approximation schemes \citep{tsitsiklis1997analysis} to approximate the value function ($Q$-function), it is unclear how to choose the function approximators and how good the approximations perform. 
Such scalability issues have indeed been 
%The challenges posed by the curse of dimensionality are not abstract, they have been 
observed in previous research on variants of the problem we study, e.g. in multi-agent RL \citep{marl_littman1994markov,marl_bu2008comprehensive} and factored Markov Decision Process (MDP) \citep{kearns1999efficient,factor_guestrin2003efficient}. %In these problems, either the state space or the action space are product spaces and are exponentially large. 
A variety of approaches have been proposed to manage this issue, e.g. the idea of ``independent learners'' in \citet{tan1993multi,marl_claus1998dynamics}; or 
function approximation schemes \citep{tsitsiklis1997analysis}. % , but it is unclear how to choose the function approximators that have provably near optimal performance. 
  However, such approaches lack rigorous optimality guarantees. 
In fact, it has been suggested that such MDPs with exponentially large state spaces may be fundamentally intractable, e.g., see \citet{complexity_blondel2000survey}.
%\lina{My general feeling for this paragraph and the earlier paragraph is that the message are repetitive and thus the flow is not very clear, kind of going back and forth. Also some messages overlap with related work. I prefer shorting the second paragraph by just integrating the new message into the first paragraph. Also, since both ``unscalable'' and '`intractable'' are used, it causes a little bit confusion.}
%\guannan{I merged the above two paragraphs, and cut some the details about multi-agent RL and only keep the main point. The details will be discussed in the literature review part. }
% We will give a more detailed review in Sec 1.1.

%In fact, it is even difficult to just specify the transition probability matrix of the problem. %In fact, the curse of dimensionality can also be observed in the context of multi-agent reinforcement learning, 
	%, let alone solve the problem efficiently. 
	%	Indeed, most dynamic programming algorithms like value iteration, policy iteration \cite{Bertsekas2007dp} or reinforcement learning algorithms like $Q$-learning \cite{sutton1998introduction} take time at least linear in state space size to solve the problem which is intractable for large $n$. 
	
In addition to the scalability issue, another challenge is that, even if an optimal policy that maps a global state $(s_1,\ldots,s_n)$ profile to a global action $(a_1,\ldots,a_n)$ can be found, it is usually impractical to implement such a policy for real-world networked systems because of the limited information and communication among agents. For example, in large scale networks, each agent $i$ may only be able to implement \emph{localized policies}, where its action $a_i$ only depends on its own state $s_i$. Designing such localized policies with global network performance guarantees can also be challenging, e.g., see \citet{rotkowitz2005characterization}.
%\lina{``Localized'' was only introduced here. may remove ``localized'' in the first paragraph because ``localized'' was left in the air for a long time. ``control applications for large scale networked systems'' is good and clear enough. also one sentence saying why ``localized'' policy is challenging?}
%Therefore, we restrict to the class of local polices, where agent $i$ takes action $a_i$ based on its own local state $s_i$.  %\textcolor{red}{needs to make the first challenge more concise; similar to the second challenge} 

	%The above challenges motivate us to focus on Multi-agent MDPs with \emph{local dependence structures} and \emph{local policies}. Such structures can be found in many real networks,  %biological networks \cite{application_o2013spreading}.

The challenges described above highlight the difficulty of applying RL to control large scale networked systems; however, the network itself provides some structure, particularly the local interaction structure, that can potentially be exploited.  The question that motivates this paper is: \emph{Can the network structure be utilized to develop scalable RL algorithms that provably find a (near-)optimal localized policy?} 

\textbf{Contributions.}  In this work we propose a framework that exploits properties of the network structure to develop RL to learn \emph{localized} policies for large-scale networked systems in a \emph{scalable} manner. Specifically, our main result (Theorem \ref{thm:main}) shows that our algorithm, \algonamefull\ (\algoname), finds a localized policy that is a $O(\rhok)$-approximation of a stationary point of the objective function, with complexity that scales with the local state-action space size of the largest $\khop$-hop neighborhood. To the best of our knowledge, our results are the first to provide such provable guarantees for scalable RL of localized policies in multi-agent networked settings.

 The key technique underlying our results is we prove that, under the local interaction structure, the $Q$-function satisfies an \emph{exponential decay property} (Definition~1), where the $Q$-function's dependence on far away nodes shrink exponentially in their graph distance with rate $\rho\leq \gamma$, where $\gamma$ is the discounting factor. This leads to a tractable approximation of the $Q$-function.  In particular, despite the $Q$-function itself being intractable to compute due to the large state-action space size, we introduce a \emph{truncated $Q$-function} which only depends on a small spatial horizon, (see Lemma \ref{lem:truncated_pg}) that can be computed efficiently and can be used in an actor critic framework which yields an $O(\rho^\kappa)$-approximation. This technique is novel and is a contribution in its own right.  It can be used broadly to develop RL for networked settings beyond the specific actor critic algorithm we propose in this paper. 

To illustrate our model and our results, we provide stylized examples of applications in three areas: multi-access wireless communication, epidemics, and traffic signal control in Section~\ref{subsec:examples}. We conduct numerical experiments to demonstrate the performance of the approach using both synthetic examples and an application to wireless communication in Section~\ref{sec:numerical}.

%\adam{Potentially mention the specific examples here}\guannan{I wrote a short paragraph... not sure it is enough. }

%The key technique that underlies the proof of our main result connects the exponential decay of the $Q$-function with the approximation error of a truncated version of the $Q$-function (see Definition~\ref{def:exp_decaying} and Lemma~\ref{lem:truncated_pg}). 

\textbf{Related Literature.} Our problem falls into the category of the ``succinctly described'' MDPs in \citet[Section 5.2]{complexity_blondel2000survey}, where the state/action space is a product space formed by the individual state/action space of multiple agents. 
%but the transition has some structure and can be succinctly described. 
As the state/action space is exponentially large, such problems are not scalable in general, even when the problem has structure \citep{complexity_blondel2000survey,complexity_whittle1988restless,complexity_papadimitriou1999complexity}. 
%, even under some structural assumptions. %Due to these challenges, \textcolor{blue}{there is few work that is directly similar to our problem???} 
Despite this, there is a large literature on RL/MDPs in multi-agent settings, which we discuss below. %We provide a brief review below, with a focus on how the curse of dimensionality is handled in each case; highlighting the contrast with the model and results in this paper.
%with a few tractable exceptions like the multi-armed bandit MDP problem \citep{complexity_gittins2011multi}. 
%Compared to these work, our model is more structured. 
%For example, currently our dependence graph is restricted to be a one directional tree, and we place some assumptions on the parameters of the models (cf. Section~\ref{sec:fast_decaying}). 
%In the context of \cite{complexity_blondel2000survey}, our model can be viewed as another type of succinctly described MDPs that can be tractable under some conditions. 

% \begin{table}[t]
%     \centering
%     \begin{tabular}{c|cccc}
%     \textbf{Problem}    & \textbf{State} & \textbf{Action} & \textbf{Coupling}& \textbf{Representative Literature}  \\
%      \hline
%      Multi-agent RL    & global & local & yes & \citet{zhang2019multi} \\
%      Factored MDP & local & global & local coupling & \citet{factor_guestrin2003efficient} \\
%      Weakly Coupled MDP & local & local & reward only & \citet{weakly_mdp_meuleau1998solving} \\
%      \textbf{Networked Systems RL} & local & local & local coupling & \textbf{Our work}
%     \end{tabular}
%     \caption{{Comparison to related settings studied previously in the RL literature. }}
%     \label{tab:comparison}
% \end{table}
\textit{Multi-agent RL} dates back to the early work of \citet{marl_littman1994markov,marl_claus1998dynamics,marl_littman2001value,marl_nashq} (see \citet{marl_bu2008comprehensive} for a review) and has been actively studied, e.g. \cite{zhang2018fully,kar2013cal,macua2015distributed,mathkar2017distributed,wai2018multi}, see a more recent review in \citet{zhang2019multi}. 
Multi-agent RL encompasses a broad range of settings including competitive agents and Markov games. The case most relevant to ours is the cooperative multi-agent RL where typically, the agents can take their own actions but they share a common global state and maximize a global reward \citep{marl_bu2008comprehensive}. 
This is in contrast to the model we study, in which each agent has its own state and acts upon its own state. 
Despite the existence of a global state, multi-agent RL still faces scalability issues since the joint-action space is exponentially large. 
Methods have been proposed to deal with this, including independent learners %, where each agent employs single-agent RL methods, treating other agents as part of the environment 
\citep{tan1993multi,marl_claus1998dynamics,matignon2012independent}, where each agent employs a single-agent RL method. While successful in some cases, the independent learner approach can suffer from instability \citep{matignon2012independent}. 
Alternatively, one can use function approximation schemes to approximate the large $Q$-table, e.g. linear function approximation \citep{zhang2018fully} or neural networks \citep{lowe2017multi}. Such methods can reduce computation complexity significantly, but it is unclear whether the performance loss caused by the function approximation is small. 
In contrast, our technique not only reduces computation but also guarantees small performance loss.

\textit{Factored MDPs} are problems where every agent has its own state and the state transition factorizes in a way similar to our model \citep{kearns1999efficient,factor_guestrin2003efficient,osband2014near}. However, they differ from the model we consider in that each agent does not have its own action.  Instead, there is a global action affecting every agent.  Despite the difference, Factored MDPs still suffer from scalability issues. Similar approaches as in the case of multi-agent RL are used, e.g., \citet{factor_guestrin2003efficient} proposes a class of ``factored'' linear function approximators; however, it is unclear whether the loss caused by the approximation is small. 
%In contrast, our truncation technique guarantees small performance loss while also reducing the computational demands.   %the theoretic guarantee in \cite{factor_guestrin2003efficient} is based on a certain projection error, which is not guaranteed to be small, whereas our optimality gap is guaranteed to be small, though we make stronger assumptions compared to \cite{factor_guestrin2003efficient}.

%(Maybe remove) \textbf{Partially-Observable MDPs and Stochastic Games.} By restricting $a_i$ to depend on only the local state $s_i$, the problem for each agent can be viewed as a Partially-Observable MDP (POMDP) \cite[Chapter 5]{bertsekas2005dp_vol1} if every other agent's policy is fixed. By allowing every agent to explore their optimal strategy, our problem shares many similarities to stochastic games \cite{stochasticgame_shapley,stochasticgame_jeffeee}. However, we would like to clarify that we only focus on the cooperative setting where agents cooperatively trying to minimize the total rewards. We also only consider the local policies where $a_i(t)$ is only based on \textit{current} $s_i(t)$. The uncooperative setting and history-based local policies are left for future interest. 

\textit{Other Related Work.} Our work is also related to weakly coupled MDPs, where every agent has its own state and action but their transition is decoupled \citep{weakly_mdp_meuleau1998solving}.
%While similar to our model, our model differs in that the transition probability is coupled among the agents.  
%Another related problem is the \emph{distributed reinforcement learning problem}, \adam{to be filled in} 
Additionally, our model shares some similarity with Glauber dynamics in physics \citep{physics_andrey2015dynamic,mezard2009information}, though our focus is very different from these works. As we consider the class of localized policies, another related line of work is Partially Observable MDP
(POMDP) \citep{nair2005networked,oliehoek2016concise,bertsekas2005dp_vol1}, though the formulations and results we have are very different from those works. %Our work also shares many similarities to stochastic games \citep{stochasticgame_shapley,stochasticgame_jeffeee} %\adam{messed up reference here} 
%as we allow every agent to explore their optimal strategy. But, we would like to emphasize that we focus on the cooperative setting where agents cooperatively trying to minimize the global rewards.% We also only consider the local policies where $a_i(t)$ is only based on
%current $s_i(t)$. The uncooperative setting and history-based local policies are left for future interest. 

Finally, this work is related to our earlier work \citet{qu2019exploiting}, which assumes the full knowledge of the MDP model (not RL) and imposes strong assumptions on the graph. In contrast, our work here does not need knowledge of the MDP and significantly relaxes the assumptions. %\guannan{I also cite our own CDC work. }

\section{Preliminaries}
In this section, we introduce our model, provide a few illustrative examples, and provide important background in RL that underlies our analysis. Throughout this paper, $\Vert \cdot\Vert$ denotes Euclidean norm and $\Vert \cdot\Vert_\infty$ denotes infinity norm. Notation $t$ and $T$ are reserved as iteration counters for the inner loop of the algorithm to be introduced later, $m$ and $M$ for the outer loop, and $\khop$ is used for counting the hops of neighbors. Notation $O(\cdot)$ hides constants and $\tilde{O}(\cdot)$ hides $\log$ factors with respect to iteration variables $T,M$ and variable $\khop$. The total variation distance for two distributions $\pi,\pi'$ over a finite set $\mathcal{S}$ is defined as $\textup{TV}(\pi,\pi') = \sup_{E\subset \mathcal{S}} |\pi( E) - \pi'( E) |$. 
%(Not sure if this paragraph is necessary as most notations (like those for graph) are introduced in Section 2.1.

\subsection{Model}\label{subsec:model}
{We consider a network of $n$ agents that are associated with an underlying undirected graph $\mathcal{G} = (\mathcal{N},\mathcal{E})$}, where  $\mathcal{N}=\{1,\ldots,n\}$ is the set of agents and $\mathcal{E}\subset \mathcal{N}\times\mathcal{N}$ is the set of edges.  Each agent $i$ is associated with state $s_i\in\mathcal{S}_i$, $a_i\in\mathcal{A}_i$ where $\mathcal{S}_i$ and $\mathcal{A}_i$ are finite sets. The global state is denoted as $s = (s_1,\ldots,s_n)\in \mathcal{S}:=\mathcal{S}_1\times\cdots\times \mathcal{S}_n$ and similarly the global action $a=(a_1,\ldots,a_n)\in\mathcal{A}:=\mathcal{A}_1\times\cdots\times\mathcal{A}_n$. At time $t$, given current state $s(t)$ and action $a(t)$, the next individual state $s_i(t+1)$ is independently generated and is only dependent on neighbors:
\begin{align}
  P(s(t+1)|s(t),a(t)) = \prod_{i=1}^n P(s_i(t+1)|s_{N_i}(t),a_i(t)),  \label{eq:transition_factor}
\end{align}
where notation $N_i$ means the neighborhood of $i$ (including $i$ itself) and notation $s_{N_i}$  means the states of the agents in $N_i$. In addition, for integer $\khop\geq 0$, we use $\nik$ to denote the $\khop$-hop neighborhood of $i$, i.e. the nodes whose graph distance to $i$ has length less than or equal to $\khop$. We also let $\fk  = \sup_{i} |\nik|$. 

%\lina{$\khop$ starts with $0$ or $1$?}

%Next, a policy is a map from the state space $\mathcal{S}$ to the action space $\mathcal{A}$. 
Each agent is associated with a class of localized policies $\zeta_i^{\theta_i}$ parameterized by $\theta_i$. 
The localized policy $\zeta_i^{\theta_i}(a_i|s_i)$ is a distribution on the local action $a_i$ conditioned on the local state $s_i$, and each agent, conditioned on observing $s_i(t)$, takes an action $a_i(t)$ independently drawn from $\zeta_i^{\theta_i}(\cdot|s_i(t))$. %We will denote the local policy node $i$ as $\zeta_i$, and a local policy profile is the tuple of local policies at all nodes, 
We use $\theta = (\theta_1,\ldots,\theta_n)$ to denote the tuple of the localized policies $\zeta_i^{\theta_i}$, and also use $\zeta^\theta(a|s) = \prod_{i=1}^n\zeta_i^{\theta_i}(a_i|s_i)$ to denote the joint policy, which is a product distribution of the localized policies as each agent acts independently. % $\zeta_i^{\theta_i}$ and write the $\zeta^\theta = (\zeta_1^{\theta_1},\ldots,\zeta_n^{\theta_n})$ where 

Further, each agent is associated with a stage reward function $r_i(s_i,a_i)$ that depends on the local state and action, and the global stage reward is $r(s,a) = \frac{1}{n}\sum_{i=1}^n r_i(s_i,a_i)$. %We will also interpret $r_i$ as a vector. We assume that all rewards $r_i$ are bounded above by $\bar{r}$. 
The objective is to find localized policy tuple $\theta$ such that the discounted global stage reward is maximized, starting from some initial state distribution $\pi_0$,
\begin{align}
     \max_\theta J(\theta): = \E_{s \sim \pi_0} \E_{a(t) \sim \zeta^{\theta}(\cdot|s(t))}\bigg[ \sum_{t=0}^\infty \gamma^t r(s(t),a(t) )  \bigg|s(0) = s\bigg]. \label{eq:discounted_cost}
\end{align}

\begin{remark}\label{rem:general_dependence} In the state transition \eqref{eq:transition_factor} of our model, the distribution of each node's next state is allowed to depend on its neighbors' states $s_{N_i}(t)$, but only on its own action $a_i(t)$ as opposed to its neighbors' actions $a_{N_i}(t)$. This restriction is imposed only for simplicity of exposition. With a simple change of notation, our model, algorithm and analysis can be extended to the more general dependence on $a_{N_i}(t)$. Similarly, each agent's reward function $r_i(s_i,a_i)$ can be generalized to depend on its neighbors' state-action pairs, i.e. $r_i(s_{N_i},a_{N_i})$, and each agent's localized policy can also be generalized to depend on its neighbors' states, i.e. $\zeta_i^{\theta_i}(a_i|s_{N_i})$. \end{remark} 

\begin{remark}
In the paper, the interaction graph $\mathcal{G}$ is undirected, but our model and results can be easily generalized to the directed graph setting without essential changes in the algorithm and the analysis. In detail, to generalize to the directed graph case, the only change needed is to redefine the ``neighbors''. Specifically, $N_i$ needs to be redefined as the ``in-neighborhood'' of $i$ (including $i$ itself), i.e. $i$ itself and the set of nodes that have a directed link pointing towards $i$. In addition, $\nik$ needs to be redefined as the $\khop$-hop ``in-neighborhood'' of $i$, i.e. the nodes whose shortest directed link to $i$ has a length less than or equal to $\khop$ (including $i$ itself). 
%We also let $\fk  = \sup_{i} |\nik|$. We note that through out this paper, by ($\khop$-hop) neighborhood, we always mean the ($\khop$-hop) in-neighbors as opposed to ``out''-neighbors in the sense defined above.  
\end{remark}

\subsection{Examples}\label{subsec:examples}

In this section, we provide three networked system examples in wireless communication, epidemics, and traffic that feature the local dependence structure we study in this paper. For ease of exposition, we present simple versions of these examples, keeping the essence of the model and highlighting the dependence structure while ignoring some application-specific details.

\textbf{Wireless Communication.} We consider a wireless network with multiple access points \citep{zocca2019temporal}, where there is a set of users  $\mathcal{N} = \{1, 2, \cdots, n\},$ and a set of network access points $Y = \{y_1, y_2, \cdots, y_m\}$. Each user $i$ only has access to a subset $Y_i \subseteq Y$ of the access points. We define the interaction graph as the conflict graph, in which two users $i$ and $j$ are neighbors if and only if they share an access point, i.e. the neighbors of user $i$ is $N_i = \{j\in \mathcal{N}: Y_i\cap Y_j\neq \emptyset \}$. Each user $i$ maintains a queue of packets defined as follows. At time step $t$, with probability $p_i$, user $i$ receives a new packet with an initial deadline $d_i$. Then, user $i$ can choose to send the earliest packet in its queue to one access point in its available set $Y_i$, or not send anything at all. If an action of sending to $y_k\in Y_i$ is taken, and if no other users send to the same access point at this time, then the earliest packet in user $i$'s queue is transmitted with success probability $q_k$ which depends on the access point $y_k$; however, if another user also chooses to send to $y_k$, then there is a conflict and no transmission occurs. If the packet is successfully transmitted, it will be removed from user $i$'s queue and user $i$ will get a reward of $1$. After this, the system moves to the next time step, with all deadlines of the remaining packets decreasing by $1$ and packets with deadline $0$ being discarded.
In this example, the local state $s_i$ of user $i$ is a characterization of its queue of packets, and is represented by a $d_i$ binary tuple $s_i = (e_1, e_2, \cdots, e_{d_i})\in\mathcal{S}_i = \{0,1\}^{d_i}$, where for each $\ell \in \{1,\ldots, d_i\}$, $e_\ell \in \{0, 1\}$ indicates whether user $i$ has a packet with remaining deadline $\ell$. 
The action space is $\mathcal{A}_i = \{\textrm{null}\} \cup Y_i  $, where $\textrm{null}$ represents the action of not sending. 
%When the user has an empty queue, then all actions will be understood as the $\textrm{null}$ action. 
%When user $i$'s queue is non-empty and it takes action $a_i = y_k \in Y_i$, i.e. sending the packet to access point $y_k$, then the packet is transmitted with success probability $q_k$ that depends on the access point $y_k$, \emph{conditioned on} no other users select this same access point; however, if another user chooses to send a packet to the same access point (i.e. a collision), neither packet is sent. 
The detailed transition is provided in Table~\ref{tab:wireless_transition}, where $s_i(t+1) $ only depends on $s_{N_i}(t), a_{N_i}(t)$, which fits into the local interaction structure we consider. The local reward is given by $r_i(s_{N_i}(t),a_{N_i}(t)) = 1$ in the case of the last row of Table~\ref{tab:wireless_transition}, and $r_i(s_{N_i}(t),a_{N_i}(t)) = 0$ in all other cases. The local state space $\mathcal{S}_i$, the local action space $\mathcal{A}_i$, the local transition probabilities in Table~\ref{tab:wireless_transition} and the local reward function $r_i(\cdot)$ form a networked MDP model described in Section~\ref{subsec:model}. 
The above model serves as a basis for more complex multi-access wireless communication models studied in the literature, including those with multiple channels \citep{block2016spatial}, (imperfect) carrier sensing \citep{kim2011achievable}. Existing analytical approaches typically require knowledge of modeling details and parameters such as the packet arrival rate \citep{tassiulas1990stability,yun2012optimal}, while our RL-based approach does not require such knowledge and learns to improve performance in a model-free manner. 
%\guannan{This model serves as basis for more complex models, like those with multiple channels [][]. But these paper make various simplying assumptions to xxx. Our RL framework doesn't assume this ... }
% (\eqref{eq:transition_factor} and Remark~\ref{rem:general_dependence}). %Further, the expected global reward can be interpreted as the throughput of the entire network. 

\begin{table}[]
    \centering
    \begin{tabular}{ccc|c}
      $s_i(t)$   & $a_i(t)$ & $s_{N_i/\{i\}}(t), a_{N_i/\{i\}}(t)$  & $s_i(t+1)$ \\ \specialrule{.2em}{.1em}{.1em}
      any & $\mathrm{null}$ & any & Left shift $s_i(t)$ and append $ \mathrm{Bernoulli}(p_i) $.\\\hline
     all zero & not null  &any & Left shift $s_i(t)$ and append $ \mathrm{Bernoulli}(p_i) $.
      \\\hline
      not all zero & $y_k \in Y_i$ & \makecell{$\exists j\in N_i/\{i\}$ s.t.\\
       $s_j(t)$ not all zero, $a_j(t) = y_k$} & Left shift $s_i(t)$ and append $ \mathrm{Bernoulli}(p_i) $. \\\hline
       \multicolumn{3}{c|}{all other cases (denote $a_i(t) = y_k$) } & \makecell{Flip left most ``$1$'' in $s_i(t)$ to ``$0$'' w.p. $q_k$,\\then left shift $s_i(t)$ and append Bernoulli($p_i$).}
    \end{tabular}
    \caption{State transition for the wireless communication example. ``Left shift'' means for a binary tuple, discarding the left most bit. Bernoulli$(p_i)$ means a random variable sampled i.i.d. from the Bernoulli distribution that has probability $p_i$ to be $1$, and probability $1-p_i$ to be $0$. %\guannan{Let me know what you think about this table... I find it is complex to write down the probabilities, and I feel the table is a clear way that shows the transition.} \adam{looks good to me..."then last bit set as..." can be rephrased though maybe "append"} \guannan{great idea!} 
    }
    \label{tab:wireless_transition}
\end{table}

\textbf{Epidemic Network. } 
%A wide range of spreading dynamics over networks like epidemic networks, opinion dynamics are also captured by our model. 
We consider an SIS (Susceptible-Infected-Susceptible) epidemic network model \citep{epi_mei2017dynamics,ahn2014random,ruhi2016analysis}, where there is a undirected graph of nodes $\mathcal{G} = (\mathcal{N},\mathcal{E})$, and each node has a binary state space $\mathcal{S}_i = \{\textrm{susceptible}, \textrm{infected}\}$, as well as a finite action space $\mathcal{A}_i$ with action $a_i\in\mathcal{A}_i$ representing epidemic control measures like different levels of vaccination \citep{preciado2013optimal}. The evolution of the states follows a local interaction structure: the probability of a node turning from susceptible to infected depends on the whether its neighboring nodes are infected or not as well as its control action in place \citep{ruhi2016improved,preciado2013optimal}; the probability of a node turning from infected to susceptible depends on the recovering rate. More precisely, $s_i(t+1)$ only depends on $s_{N_i}(t)$ and $a_i(t)$, and the state transition is provided by,
\begin{align}
    P(s_i(t+1) = \mathrm{susceptible} | s_{N_i}(t), a_i(t) ) = \left\{\begin{array}{ll} [1- \beta_i(a_i(t))]^{|\{j\in N_i/\{i\}: s_j(t) = 1 \} | }, & \text{ if } s_i(t) = \mathrm{susceptible}, \\
\delta_i, & \text{ if } s_i(t) = \mathrm{infected}, \end{array}\right. \label{eq:epidemic_transition}
\end{align}
where $\delta_i\in(0,1)$ is a given recovering rate parameter, and $\beta_i:\mathcal{A}_i\rightarrow (0,1)$ is a given transmission rate function and it depends on the control action $a_i(t)$, and $|\{j\in N_i/\{i\}: s_j(t) = 1 \} |$ is the number of neighboring nodes excluding $i$ itself that are infected. The local reward of each node is given by,
\begin{align}
    r_i(s_i,a_i) =  \mathbf{1}(s_i = \text{susceptible}) - c_i(a_i),\label{eq:epidemic_reward}
\end{align}
which consists of two parts: a positive reward of $1$ if the node is free from infection, subtracting a given cost function $c_i(a_i)$ on the epidemic control measure $a_i$.
In this setup, the expected global reward is a weighted balance between the overall infection level and the epidemic control cost. The above defined local state space $\mathcal{S}_i$, local action space $\mathcal{A}_i$, local transition probabilities in \eqref{eq:epidemic_transition} and local rewards in \eqref{eq:epidemic_reward} form a networked MDP model in Section~\ref{subsec:model}. 
We comment that the above SIS model is a basis for more complex models, e.g. those with ``exposed'' and ``recovered'' states \citep{kuznetsov1994bifurcation,britton2010stochastic} or more complex control interventions \citep{morris2020optimal}. These more complex models can also be captured using the framework above.  We reiterate that the approach in our paper does not require knowledge of model specifications and learns in a model-free manner. 

%Each node is also associated with control actions, representing epidemic control measures like quarantine, which will affect the probability the probability of transmission and curing; and such epidemic control actions will come at a certain cost. 

\textbf{Traffic Network. } Lastly, we consider a traffic signal control problem adapted from \citet{varaiya2013max}. In this setting, each node $i $ represents a road link and the interaction graph represents the physical connection of the road links. Given road link $i$, the local state $s_i = (x_{i,j})_{i\rightarrow j}$ is a tuple of variables with $j$ ranging from neighboring links $i$ can turn to, and $x_{i,j}$ is the number of vehicles on link $i$ that intend to turn to link $j$, and can only take values in $[S] = \{0,1,\ldots,S\}$. Correspondingly, the local state space is $\mathcal{S}_i = [S]^{N_i/\{i\}}$. 
Similarly, the local action $a_i = (y_{i,j})_{i\rightarrow j} $ is the binary traffic signal tuple with $y_{i,j}$ controlling the on-off of turn movement $(i\rightarrow j)$, and the local action space is $\mathcal{A}_i = \{0,1\}^{\{N_i\}/\{i\}}$. At each time, a random amount of vehicles on the queue $x_{i,j}$ will flow into link $j$ when the traffic signal $y_{i,j}$ is on. Meanwhile, link $i$ will receive vehicles from other incoming links, a random fraction of which are then assigned to each of the queues in  $(x_{i,j})_{i\rightarrow j} $. 
Mathematically,
\begin{align}
    x_{i,j}(t+1) = \Big[ x_{i,j}(t) - \min(C_{i,j}(t) y_{i,j}(t),x_{i,j}(t)) + \sum_{k\rightarrow i} \min( C_{k,i}(t) y_{k,i}(t), x_{k,i}(t) ) R_{i,j}(t)\Big]_{0}^{S}, \label{eq:example_traffic}
\end{align}
where $[x]_0^S$ means $\max(\min(x,S),0)$, $C_{i,j}(t)$ (and similarly $C_{k,i}(t)$) is an i.i.d. random variable indicating the random amount of vehicles leaving $i$ for $j$, and $R_{i,j}(t)$ is an i.i.d. random variable that controls the split of the inflow to link $i$ to the queue $x_{i,j}(t)$. See
\citet{varaiya2013max} for the complete details. Given a fixed distribution on the random variables $C_{i,j}(t), R_{i,j}(t)$, \eqref{eq:example_traffic} provides a complete characterization of the distribution of $s_i(t+1)$ conditioned on $s_{N_i}(t)$ and $a_{N_i}(t)$, in which the local state at each link $s_i(t+1)$ only depends on its neighbors' current states and current actions, which fits into our local interaction structure (equation \eqref{eq:transition_factor} and Remark~\ref{rem:general_dependence}). The local reward is a characterization of the congestion level at link $i$, and one version of the reward is the negative queue length
\begin{align}
r_i(s_i,a_i) = -\sum_{j\in N_i/\{i\}} x_{i,j}. \label{eq:example_traffic_reward}
\end{align}
The local state space $\mathcal{S}_i$, the local action space $\mathcal{A}_i$, the local transition probabilities \eqref{eq:example_traffic} and the local rewards \eqref{eq:example_traffic_reward} form a networked MDP discussed in Section~\ref{subsec:model}. 
%We note that the model in \citet{varaiya2013max} does not consider capacity constraints on the queue lengths, while in practice the queue lengths are are upper bounded by their capacity \citep{xiao2014pressure}, which fits into our RL-based model. 
We comment that policies for traffic signal control, like the max pressure policy in \citet{varaiya2013max}, typically require knowledge of the statistics of the random variables $C_{ij}(t)$ and $R_{ij}(t)$, while our approach learns from data in a model-free manner. 
%\guannan{I am least confident with this last example. Not sure if I have explained clearly. Also, do I need to mention somewhere that the sates space are infinite (all positive integers), but we can do truncation?} \adam{I think mentioning truncation is good...the rest was clear to me}
%\lina{what your wrote is clear to me but I think the criticism will be from $R_{ij}$. but I don't think we should worry it because it is from another paper and it is only an example.}

\subsection{Background in RL}

To provide background for the analysis in this paper, we review a few key concepts in RL. First, fixing a localized policy tuple $\theta = (\theta_1,\ldots,\theta_n)$, an important notion is the $Q$-function, which is defined for policy $\theta$ as a ``table'' of values for each state-action pair $(s,a)\in\mathcal{S}\times\mathcal{A}$ and it is the expected infinite horizon discounted reward under policy $\theta$ conditioned on the initial state and action being $(s,a)$:
{\small\begin{align}
Q^\theta(s,a) &= \E_{a(t) \sim \zeta^{\theta}(\cdot|s(t))}\bigg[ \sum_{t=0}^\infty \gamma^t r(s(t),a(t) )  \bigg|s(0) = s, a(0) = a\bigg] \nonumber \\
&=\frac{1}{n} \sum_{i=1}^n  \E_{a(t) \sim \zeta^{\theta}(\cdot|s(t))} \bigg[ \sum_{t=0}^\infty \gamma^t r_i(s_i(t),a_i(t) ) \bigg|s(0) = s,a(0) = a \bigg] 
:=\frac{1}{n}\sum_{i=1}^n Q_i^\theta(s,a).  \label{eq:full_q}
\end{align}}In the last step, we have defined $Q_i^\theta(s,a)$ which is the $Q$ function for the individual reward $r_i$. Both $Q^\theta$ and $Q_i^\theta$ are exponentially large tables and, therefore, are intractable to compute and store. 

Additionally, another important concept we use is the policy gradient theorem, which provides a characterization of the gradient of the objective $J(\theta)$ and is the basis of many algorithmic results in RL. The policy gradient theorem shows that the gradient of $J(\theta)$ depends on $Q^\theta$ and, therefore, is intractable to compute using the form in Lemma~\ref{lem:policy_grad}. 

\begin{lemma}[\cite{sutton2000policy}]\label{lem:policy_grad} Let $\pi^\theta $ be a distribution on the state space given by $\pi^\theta(s) = (1-\gamma) \sum_{t=0}^\infty \gamma^t \pi_t^\theta(s) $, where $\pi_t^\theta$ is the distribution of $s(t)$ under a fixed policy $\theta$ when $s(0)$ is drawn from $\pi_0$. Then, we have,
\begin{align} 
    \nabla J(\theta) = \frac{1}{1-\gamma} \E_{s\sim \pi^\theta, a\sim\zeta^\theta(\cdot|s)} \Big[ Q^{\theta}(s,a) \nabla \log \zeta^\theta (a|s)\Big]. \label{eq:approx_policy_grad}
\end{align}
 %here the term $\sum_{s'} P(s(t) = s|s(0) = s' ) \pi_0(s')$ can be interpreted as the probability of $s(t)$ equaling $s$ when the initial state $s(0)$ is drawn from distribution $\pi_0$. 
\end{lemma}

\section{Algorithm Design and Results}\label{sec:main_results}

In this paper we propose an algorithm, \algonamefull\ (\algoname), which provably finds an $O(\rhok)$-stationary point of the objective $J(\theta)$ (i.e. a $\theta$ s.t. $\Vert \nabla J(\theta)\Vert^2\leq\varepsilon$) for some $\rho\leq\gamma$,
with complexity scaling in the size of the local state-action space of the largest $\khop$-hop neighborhood.  We state our main result formally in Theorem \ref{thm:main} after introducing the details of \algoname\ and the key idea underlying its design.

%\begin{theorem}[Informal] \label{thm:informal} \algoname\ learns a policy $\theta$ that is a $O(\rhok )$ stationary point of of the cost $J(\theta)$ for some $0<\rho\leq \gamma$ \adam{"local optimal" hasn't been defined}, with complexity scaling in the size of the state-action space size of the largest $\khop$-hop neighborhood in the network. \adam{this statement needs work}
%\end{theorem}\guannan{I edited a bit. It is hard to very precise for the informal version as so many things are not defined yet. }

%\algoname\ is based on the observation that the $Q$-function satisfies an exponential decay property, which allows efficient approximation of the policy gradient via truncation.  So, before presenting the details of the algorithm we first formalize this key idea.

%In Section~\ref{subsec:key_idea}, we will describe our key idea, which will then be used to design our algorithm in Section~\ref{subsec:algo}. Finally, the formal performance guarantee will be provided in Section~\ref{subsec:convergence}.

\subsection{Key Idea: Exponential Decay of $Q$-function Leads to Efficient Approximation}\label{subsec:key_idea}
Recall that the policy gradient in Lemma~\ref{lem:policy_grad} is intractable to compute due to the dimension of the $Q$-function. Our key idea is that exponential decay of the $Q$ function allows efficient approximation of the $Q$-function via truncation. 
To illustrate this, we start with the definition of the exponential decay property. 
Recall that $\nik$ is the set of $\khop$-hop neighborhood of node $i$ and define $\nminusik = \mathcal{N}/\nik$, i.e. the set of agents that are outside of $i$'th $\khop$-hop neighborhood. We write state $s$ as $(s_{\nik}, s_{\nminusik})$, i.e. the states of agents that are in the $\khop$-hop neighborhood of $i$ and outside of the $\khop$-hop neighborhood respectively. Similarly, we write $a$ as $(a_{\nik}, a_{\nminusik})$. The exponential decay property is then defined as follows.

\begin{definition}\label{def:exp_decaying}
The $(c,\rho)$-exponential decay property holds if, for any localized policy $\theta$, for any $i\in\mathcal{N}$, $s_{\nik}\in \mathcal{S}_{\nik}$, $s_{\nminusik}, s_{\nminusik}'\in \mathcal{S}_{\nminusik} $, $a_{\nik} \in \mathcal{A}_{\nik}$, $a_{\nminusik}, a_{\nminusik}' \in \mathcal{A}_{\nminusik}$, $Q_i^\theta$ satisfies,
$$|Q_i^\theta(s_{\nik},s_{\nminusik}, a_{\nik}, a_{\nminusik}) - Q_i^\theta(s_{\nik},s_{\nminusik}',a_{\nik}, a_{\nminusik}')| \leq c\rhok . $$
\end{definition}

% {\color{red}

% % \begin{lemma}\label{lem:exp_decay_mixing}
% %   Suppose $r_i$ is upper bounded by $\bar{r}$ for all $i$, and assume there exists $c'>0$ and $\mu\in(0,1)$ s.t. under any policy $\theta$, the Markov chain is ergodic and starting from any initial state, $\textup{TV}(\pi_{t,i},\pi_{\infty,i}) \leq c' \mu^t$ where $\pi_{t,i}$ is the distribution of $(s_i(t),a_i(t))$ and  $\pi_{\infty,i}$ is the distribution for $(s_i,a_i)$ in stationarity. Then, the $(\frac{2c'\bar{r}}{1-\gamma \mu},\gamma\mu)$-exponential decay property holds. 
% % \end{lemma}

% %\lina{there is formating for ``T th proof"} \guannan{I tentatively put the above Lemma into main text as part (b) of Lemma 2. If we finally decide to do so, we should change the proof accordingly. }
%  }

It may not be immediately clear when the exponential decay property holds.  Lemma~\ref{lem:exp_decaying} (a) below highlights that the exponential decay property holds generally with $\rho = \gamma$, without any assumption on the transition probabilities except for the factorization structure \eqref{eq:transition_factor} and the localized policy structure. 
Further, many MDPs in practice have ergodicity and fast mixing properties, and Lemma~\ref{lem:exp_decaying} (b) shows that when such fast mixing property holds, the $(c,\rho)$-exponential decay property holds for some $\rho<\gamma$ depending on the mixing rate. The proof of Lemma~\ref{lem:exp_decaying} is postponed to Appendix \ref{subsec:appendix:exponential_decaying}. The condition on mixing rate in Lemma~\ref{lem:exp_decaying} (b) is similar to those used in the literature on the finite time analysis of RL methods, e.g. \citet{zou2019finite}. In fact, our condition is weaker than the common mixing rate condition in that we only require the distribution of the local state-action pair $(s_i(t),a_i(t))$ to mix, instead of the full state-action pair $(s(t),a(t))$. We leave it as future work to study such ``local'' mixing behavior and its relation to the local transition probabilities \eqref{eq:transition_factor}.

\begin{lemma}\label{lem:exp_decaying} 
Assume $\forall i$, $r_i$ is upper bounded by $\bar{r}$. Then the following holds.
\begin{itemize}
    \item [(a)] The $(\frac{\bar{r}}{1-\gamma}, \gamma)$-exponential decay property holds.
    \item [(b)]If there exists $c'>0$ and $\mu\in(0,1)$ s.t. under any policy $\theta$, the Markov chain is ergodic and starting from any initial state, $\textup{TV}(\pi_{t,i},\pi_{\infty,i}) \leq c' \mu^t, \forall t,$ where $\pi_{t,i}$ is the distribution of $(s_i(t),a_i(t))$ and  $\pi_{\infty,i}$ is the distribution for $(s_i,a_i)$ in stationarity, and recall $\textup{TV}(\cdot,\cdot)$ is the total variation distance. Then, the $(\frac{2c'\bar{r}}{1-\gamma \mu},\gamma\mu)$-exponential decay property holds.  

\end{itemize}
\end{lemma}
%\lina{need to provide definition of $TV$ somewhere, either in the first paragraph of preliminaries where you mentioned the notations or here in the theorem add a footnote explaining. also I have some concern that reviewer will question the usefulness of statement (b). Maybe add a a couple of sentence saying that the statement (b) is to show that the decaying constant can be better than $\gamma$? Or it might be better to just remove it and replace it with some emperical data report? } 

%\guannan{I am not sure either whether it helps to add (b). My incentive for doing so as this would make the paper more different from the conference version. }

Our definition of exponential decay is similar in spirit to the ``correlation decay'', or ``spatial decay'' that has been studied in the literature \citep{gamarnik2013correlation,gamarnik2014correlation,bamieh2002distributed}, though these works consider very different settings. For example, \cite{gamarnik2013correlation} and \cite{gamarnik2014correlation} study optimization in a graphical model setting (no concept of state and/or time), and show that the effect of cost functions far away on the optimal solution at a particular node shrinks exponentially in their graph distance, under certain weak interaction assumptions. Compared to these works where the optimization problem is static, we focus on an MDP setting which has states that evolve on a time axis. Further, our exponential decay is in terms of the $Q$ functions, as opposed to the optimal solution. That being said, we believe there are deep connections between our results and that in \citet{gamarnik2013correlation,gamarnik2014correlation}, and we leave the investigation of it as future work.  

The power of the exponential decay property is that such properties usually lead to scalable and distributed algorithm design, as in \citet{gamarnik2013correlation}. In our context, the exponential decay property guarantees that the dependence of $Q_i^\theta$ on other agents shrinks quickly as the distance between them grows. This motivates us to consider the following class of truncated $Q$-functions, 
\begin{align}
    \hat{Q}_i^\theta (s_{\nik}, a_{\nik})  = \sum_{s_{\nminusik}\in \mathcal{S}_{\nminusik} , a_{\nminusik}\in \mathcal{A}_{\nminusik}} w_i(s_{\nminusik}, a_{\nminusik};s_{\nik}, a_{\nik} ) Q_i^\theta (s_{\nik},s_{\nminusik}, a_{\nik},a_{\nminusik}),\label{eq:truncated_q}
\end{align}
where $w_i(s_{\nminusik}, a_{\nminusik};s_{\nik}, a_{\nik} )$ are \textit{any} non-negative weights satisfying 
\begin{align}
    \sum_{s_{\nminusik}\in \mathcal{S}_{\nminusik} , a_{\nminusik}\in \mathcal{A}_{\nminusik}} w_i(s_{\nminusik}, a_{\nminusik};s_{\nik}, a_{\nik} ) = 1, \qquad \forall (s_{\nik}, a_{\nik}) \in \mathcal{S}_{N_{i}^k}\times \mathcal{A}_{N_{i}^k}. \label{eq:truncated_q_weights}
\end{align}With the definition of the truncated $Q$-function, our key insight is the following Lemma~\ref{lem:truncated_pg}, which says when the exponential decay property holds, the truncated $Q$-function \eqref{eq:truncated_q} approximates the full $Q$-function with high accuracy and can be used to approximate the policy gradient. 
%an accurate approximation of the full $Q$-function, and can be used to accurately approximate the policy gradient. 
The proof of Lemma~\ref{lem:truncated_pg} is postponed to {Appendix~\ref{sec:appendix_truncated_pg}.}

\begin{lemma}\label{lem:truncated_pg} 
Under the $(c,\rho)$-exponential decay property, the following holds:
\begin{itemize}
  \item [(a)]   Any truncated $Q$-function in the form of \eqref{eq:truncated_q} satisfies,
    \begin{align*} \sup_{(s,a)\in \mathcal{S}\times\mathcal{A}} |\hat{Q}_i^\theta (s_{\nik}, a_{\nik}) - Q_i^\theta(s,a)|\leq c\rhok.  \end{align*} 
    \item [(b)] Given $i$, define the following truncated policy gradient,
\begin{align}
    \hat{h}_i(\theta) = \frac{1}{1-\gamma} \E_{s\sim\pi^\theta,a\sim \zeta^\theta(\cdot|s)} \Big[\frac{1}{n}\nabla_{\theta_i} \log\zeta_i^{\theta_i} (a_i|s_i)\sum_{j\in {\nik}} \hat{Q}_j^{\theta} (s_{\njk},a_{\njk})\Big] , \label{eq:truncated_pg}
\end{align}
where $\hat{Q}_j^\theta$ can be any truncated $Q$-function in the form of \eqref{eq:truncated_q}. Then, if $\Vert \nabla_{\theta_i} \log \zeta_i^{\theta_i} (a_i|s_i)\Vert\leq L_i, \forall a_i, s_i$, we have
$\Vert \hat{h}_i(\theta) - \nabla_{\theta_i} J(\theta)\Vert \leq  \frac{cL_i }{1-\gamma}\rhok $.
\end{itemize}
 
\end{lemma}

%\begin{proof}
%Should I put the proof here?
%\end{proof}

The power of this lemma is that the truncated $Q$ function has a much smaller dimension than the true $Q$ function, and is thus scalable to compute and store.  However, despite the reduction in dimension, the error resulting from the approximation is small. In the next section, we use this idea to design a scalable algorithm. 

% \begin{lemma} For any $i$, any $s_{\nik}$, $s_{\nminusik}$ and $s_{\nminusik}'$ we have
% $$|V_i^\theta(s_{\nik},s_{\nminusik}) - V_i^\theta(s_{\nik},s_{\nminusik}')| \leq \frac{2\bar{r}}{1-\gamma}\gamma^{k+1} $$
% \end{lemma}

% \begin{proof}
% Let $\pi_i(t)$ be the distribution of $s_i(t)$ conditioned on $s(0) = (s_{\nik},s_{\nminusik})$, and let $\pi_i'(t)$ be the distribution of $s_i(t)$ conditioned on $s(0)=(s_{\nik},s_{\nminusik}')$. Then, we must have $\pi_i(t) = \pi_i'(t)$ for all $t\leq k$, because $\pi_i(t)$ only depends on $s_{N_i^t}(0)$, i.e. the initial states of $i$'th $t$-hop neighborhood.
% \begin{align*}
%   |V_i^\theta(s_{\nik},s_{\nminusik}) - V_i^\theta(s_{\nik},s_{\nminusik}')| &\leq   \sum_{t=0}^\infty \bigg| \E \big[\gamma^t r_i(s_i(t))  \big|s_{1:n}(0) = s_{1:n} \big] - \E \big[\gamma^t r_i(s_i(t))  \big|s_{1:n}(0) = s_{1:n} \big] \bigg|\\
%   & = \sum_{t=0}^\infty \bigg| \gamma^t \langle r_i, \pi_i(t)\rangle  - \gamma^t \langle r_i,\pi_i'(t)\rangle \bigg|\\
%   &= \sum_{t=k+1}^\infty \bigg| \gamma^t \langle r_i, \pi_i(t)-\pi_i'(t)\rangle \bigg|\\
%   &=\sum_{t=k+1}^\infty  \gamma^t  \Vert r_i\Vert_\infty \Vert\pi_i(t)-\pi_i'(t)\Vert_1 \\
%  & \leq  \frac{2\bar{r}}{1-\gamma}\gamma^{k+1}
% \end{align*}
% \end{proof}

\subsection{Algorithm Design: \algonamefull\ (\algoname)}\label{subsec:algo}

%In the previous section we have shown that the truncated $Q$-function leads to accurate approximation of the policy gradient, despite the fact that the dimension of the truncated $Q$-function \eqref{eq:truncated_q} is much smaller than that of the full $Q$ function \eqref{eq:full_q}. \lina{this sentence is redundant because it follows immediately from the previous paragraph. }
The good properties of the truncated $Q$-function open many possibilities for algorithm design. For instance, one can first obtain the truncated $Q$-function in some way (which could be much easier than directly computing the full $Q$-function) and then do a policy gradient step using the Lemma~\ref{lem:truncated_pg}. In this subsection, we propose one particular approach using the actor critic framework.  Our approach, \algonamefull\ (\algoname), uses temporal difference (TD) learning to obtain the truncated $Q$-function and then uses policy gradient for policy improvement. The pseudocode of the proposed algorithm is given in Algorithm~\ref{algorithm:key_algo}.%, and we describe the approach in detail in the following.

\textbf{Overall structure.} The overall structure of \algoname\ is a for-loop from line~\ref{algo:outer_start} to line~\ref{algo:outer_end}. Inside the outer loop, there is an inner loop (line \ref{algo:innerloop_start} through line \ref{algo:innerloop_end}) that uses temporal difference learning to get the truncated $Q$-function, which is followed by a policy gradient step that does policy improvement. 

\textit{The Critic: TD-inner loop.} Line \ref{algo:innerloop_start} through line \ref{algo:innerloop_end} is the policy evaluation inner loop that obtains the truncated $Q$ function, where line \ref{algo:td_1} and \ref{algo:td_2} are the temporal difference update. We note that steps \ref{algo:td_1} and \ref{algo:td_2} use the same update equation as TD learning, except that it ``pretends'' $(s_{\nik}, a_{\nik})$  is the true state-action pair while the true state-action pair should be $(s,a)$. As will be shown in the theoretic analysis, such a TD update implicitly gives an estimate of a truncated $Q$ function.

\textit{The Actor: policy gradient.} Line~\ref{algo:actor_start} through line \ref{algo:actor_end} define the actor actions.  Here, each agent calculates an estimate of the truncated gradient based on~\eqref{eq:truncated_pg}, and then conducts a gradient step.

\begin{algorithm2e}[t]\caption{\algoname: \algonamefull}\label{algorithm:key_algo}
\DontPrintSemicolon
\LinesNumbered
%\SetAlgoLined
\KwIn{$\theta_i(0)$; parameter $\khop$; $T$, length of each episode; step size parameters $h,t_0,\eta$.}
	    
	    \For{$m=0,1,2,\ldots$}{\label{algo:outer_start}
	         %Short hand notation: $\zeta_i^m \gets \zeta_i^{\theta_i(m)}$\;
	         Sample initial state $s(0) \sim \pi_0$, each agent $i$ takes action $a_i(0) \sim \zeta_i^{\theta_i(m)}(\cdot|s_i(0))$, receives reward $r_i(0) = r_i (s_i(0),a_i(0))$.\;
	         Initialize $\hat{Q}_i^{0}\in \R^{\mathcal{S}_{\nik}\times \mathcal{A}_{\nik}}$ to be the all zero vector.\; % the $\hat{Q}_i^{T}$ in the previous iteration. \;
	        \For{$t=1$ \KwTo $T$}{ \label{algo:innerloop_start}
	         Get state $s_i(t)$, take action $a_i(t) \sim \zeta_i^{\theta_i(m)}(\cdot|s_i(t))$, get reward $r_i(t) = r_i(s_i(t),a_i(t))$.\;       
	          
	         Update the truncated $Q$ function with step size $\alpha_{t-1} = \frac{h}{t-1+t_0}$, \;
	         {
	  $\hat{Q}_i^{t}(s_{\nik}(t-1),a_{\nik}(t-1)) = (1 - \alpha_{t-1}) \hat{Q}_i^{t-1}(s_{\nik}(t-1),a_{\nik}(t-1)) + \alpha_{t-1} (r_i(t-1) +  \gamma \hat{Q}_i^{t-1}(s_{\nik}(t),a_{\nik}(t)) ),$\; \label{algo:td_1}
 $\hat{Q}_i^{t}(s_{\nik},a_{\nik}) =  \hat{Q}_i^{t-1}(s_{\nik},a_{\nik}) \text{ for } (s_{\nik},a_{\nik})\neq (s_{\nik}(t-1),a_{\nik}(t-1)) .$\; \label{algo:td_2}
}

	        %\State Collect $V_j(\zeta_{\tilde\Delta_j^k})$ for all children $j\in c_i$.
	             %\State $V_i(\tilde\Delta^k_i) \gets $ \par \hskip\algorithmicindent  $\underset{\zeta_i\in \Gamma_i }{\max} \hat{R}_i^k(\zeta_i,\zeta_{\tilde\Delta^k_i}) + \underset{j\in c_i}{\sum} V_j(\zeta_i,\zeta_{\tilde\Delta^{k-1}_i}) $
	             %\State $\zeta^*_i(\zeta_{\tilde\Delta^k_i}) \gets $ \par \hskip\algorithmicindent   $ \underset{\zeta_i\in \Gamma_i }{\arg\max} \hat{R}_i^k(\zeta_i,\zeta_{\tilde\Delta^k_i}) + \underset{j\in c_i}{\sum} V_j(\zeta_i,\zeta_{\tilde\Delta^{k-1}_i}) $
	        } \label{algo:innerloop_end}
	        Each agent $i$ calculates approximated gradient, \; \label{algo:actor_start}
	        $\hat{g}_i(m) =  \sum_{t=0}^{T} \gamma^t \frac{1}{n}\sum_{j\in \nik} \hat{Q}_j^T (s_{\njk}(t),a_{\njk}(t)) \nabla_{\theta_i} \log \zeta_i^{\theta_i(m)} (a_i(t)|s_i(t)) . $\;
	         Each agent $i$ conducts gradient step $\theta_i(m+1) = \theta_i(m) + \eta_m \hat{g}_i(m)$  with $\eta_m = \frac{\eta}{\sqrt{m+1}}$.\label{algo:actor_pg} \; \label{algo:actor_end}
	    }\label{algo:outer_end}
\end{algorithm2e}
\textbf{Communication.} To implement our training algorithm, each agent needs to communicate with other agents in its $\khop$-hop neighborhood in line 7 and line 11; after training is done, each agent implements its localized policy that does not need communication. This communication requirement is weaker than the ``centralized training with decentralized execution'' paradigm in the multi-agent RL literature \citep{lowe2017multi}, where in the training phase, global communication is used. We also comment that when $\khop = 0$, our algorithm does not need communication and is effectively the same as the independent learner approach in the literature \citep{tan1993multi,lowe2017multi}, as each agent simply runs a single-agent actor critic method based on its local state and local action. When $\khop > 1$, our algorithm requires communication with agents beyond the direct $1$-hop neighbors, which may be unrealistic for some applications. An interesting future direction is to reduce the communication requirements, e.g. potentially using consensus schemes like in \citet{zhang2018fully}, and also techniques that only communicate quantized bits as opposed to real numbers \citep{magnusson2020maintaining}.

\textbf{Discussion.}
Our algorithm serves as an initial concrete demonstration of how to make use of the truncated $Q$-functions to develop a scalable RL method for networked systems. There are many extensions and other approaches that could be pursued, either within the actor critic framework or beyond. One immediate extension is to do a warm start, i.e., initialize $\hat{Q}_i^0$ as the final estimate $\hat{Q}_i^T$ in the previous outer-loop. Additionally, one can use the TD-$\lambda$ variant of TD learning, incorporate variance reduction schemes like the advantage function (Advantage Actor Critic), or incorporate function approximation. Further, beyond the actor critic framework, another direction is to develop $Q$-learning/SARSA type algorithms based on the truncated $Q$-functions. An appealing aspect of $Q$-learning/SARSA algorithms is that they may exhibit better convergence properties, but the challenge is that, unlike the actor critic framework, it is not straightforward to enforce the policy to be local in $Q$-learning/SARSA algorithms. These are interesting topics for future work.
%{\color{blue} Other techniques is to include }
%\lina{a couple of sentences about Q learning and other methods which does not do function approximation?}

\begin{remark}[Model-based vs model-free]
		We note that by using an actor critic framework, the proposed approach is model-free, meaning it does not explicitly estimate the transition probabilities and the reward function. This is in contrast to model-based RL, which explicitly estimates the transition probabilities (or parameters that determine the transition probabilities, like the $\delta_i$, $\beta_i(\cdot)$ parameter in the epidemic example in Section~\ref{subsec:examples}). On one hand, it is known that model-based RL can be more sample efficient than model-free RL in certain circumstances \citep{tu2019gap}. Additionally, for specific applications, model-based control design may come with properties like robustness \citep{varaiya2013max}. On the other hand, model-free RL offers more flexibility since it does not impose assumptions on the model class. The comparison and tradeoff between model-based and model-free approaches is an open research question and is beyond the scope of this paper. We refer the reader to \citet{tu2019gap,qu2020combining} for more details. 
\end{remark}

\subsection{Approximation Bound}\label{subsec:convergence}
In this section, we state and discuss the formal approximation guarantee for \algoname.  Before stating the theorem, we first state the assumptions we use. The first assumption is standard in the RL literature and bounds the reward and state/action space size.  

\begin{assumption}[Bounded reward and state/action space size] \label{assump:bounded_reward}
	The reward is upper bounded as $0\leq r_i(s_i,a_i) \leq \bar{r}, \forall i, s_i, a_i$. The individual state and action space size are upper bounded as $|\mathcal{S}_i|\leq S, |\mathcal{A}_i|\leq A, \forall i$.
\end{assumption}

%Next, we state our assumption about the exponential decay property.
\begin{assumption}[Exponential decay]\label{assump:exp_decaying} The $(c,\rho)$-exponential decay property holds for some $\rho\leq\gamma$. 
\end{assumption}

\noindent Note that under Assumption~\ref{assump:bounded_reward}, Assumption~\ref{assump:exp_decaying} automatically holds with $\rho = \gamma$, cf. Lemma~\ref{lem:exp_decaying} (a). However, we state the exponential decay property as an assumption to account for the more general case that $\rho$ could be strictly less than $\gamma$, cf. Lemma~\ref{lem:exp_decaying} (b).

Our third assumption can be interpreted as an ergodicity condition which ensures that the state-action pairs are sufficiently visited. 
%\begin{assumption}[Bounded local  $\khop$-hop neighborhood size] for any $i$, $|\nik|\leq \fk $. %and $\fk  = poly(k)$. (e.g. line, 2-d lattice).
%\end{assumption}

\begin{assumption}[Sufficient local exploration]\label{assump:local_explor} There exists positive integer $\tau$ and $\sigma\in (0,1)$ s.t. under any fixed policy $\theta$ and any initial state-action $(s,a)\in\mathcal{S}\times\mathcal{A}$, $\forall i \in\mathcal{N}, \forall (s_{\nik}', a_{\nik}')\in \mathcal{S}_{\nik}\times \mathcal{A}_{\nik}$, we have
		$P( (s_{\nik}(\tau), a_{\nik}(\tau) )=  (s_{\nik}', a_{\nik}') |(s(1),a(1))=(s,a)) \geq \sigma $.
\end{assumption}

\noindent Assumption \ref{assump:local_explor} requires that every state action pair in the $\khop$-hop neighborhood must be visited with some positive probability after some time. This type of assumption is common for finite time convergence results in RL. For example, in \citet{srikant2019finite,li2020sample}, it is assumed that every state-action pair is visited with positive probability in the stationary distribution and the state-action distribution converges to the stationary distribution with some rate.  This implies our assumption which is weaker in the sense that we only require local state-action pair $(s_{\nik},a_{\nik})$ to be visited as opposed to the full state-action pair $(s,a)$. Having said that, we note that by making Assumption \ref{assump:local_explor}, we do not consider the exploration-exploitation tradeoff, which is a challenging issue even in single-agent RL. One potential way to relax Assumption~\ref{assump:local_explor} is to use Upper Confidence Bound (UCB) bonuses to encourage exploration, which has been proposed in single-agent RL \citep{jin2018q}. We leave the study of the exploration-exploitation tradeoff in the multi-agent networked setting as future work.  
%\guannan{Not sure if the local exploration assumption is well justified here... }

Finally, we assume boundedness and Lipschitz continuity of the gradients, which is standard in the RL literature.

\begin{assumption}[Bounded and Lipschitz continuous gradient] \label{assump:gradient_bounded} For any $i$, $a_i$, $s_i$ and $\theta_i$, we assume $\Vert \nabla_{\theta_i} \log \zeta_i^{\theta_i} (a_i|s_i)\Vert\leq L_i$. As a result, $\Vert\nabla_\theta \log \zeta^{\theta}(a|s)\Vert\leq L = \sqrt{\sum_{i=1}^n L_i^2}$. Further, assume $\nabla J(\theta)$ is $L'$-Lipschitz continuous in $\theta$.
\end{assumption}

With these assumptions in hand, we are ready to state our convergence result.

\begin{theorem} \label{thm:main}
Under Assumption \ref{assump:bounded_reward}, \ref{assump:exp_decaying}, \ref{assump:local_explor} and \ref{assump:gradient_bounded}, for any $\delta\in(0,1)$, $M\geq 3$, suppose the critic step size $\alpha_t = \frac{h}{t+t_0}$ satisfies $h\geq \frac{1}{\sigma}\max(2,\frac{1}{1-\sqrt{\gamma}})$, $t_0\geq\max(2h, 4\sigma h, \tau)$; and the actor step size satisfies $\eta_m = \frac{\eta}{\sqrt{m+1}}$ with $\eta\leq \frac{1}{4L'}$. Further, if the inner loop length $T$ is large enough s.t. $ T+1\geq \log_\gamma \frac{c(1-\gamma)}{\bar{r}}+ (\khop+1)\log_\gamma\rho  $ and
{\small	\begin{align}
	    \frac{C_a(\frac{\delta}{2 nM},T)}{\sqrt{T+t_0}} + \frac{C_a'}{T+t_0}  \leq \frac{2c\rhok }{(1-\gamma)^2}, \label{eq:main_theorem_T_bound}
	\end{align}}where $C_a(\delta,T)=   \frac{6\bar{\epsilon}}{1-\sqrt{\gamma}}  \sqrt{  \frac{\tau h}{\sigma}  [ \log(\frac{2\tau T^2 }{\delta}) + \fk \log SA]}$ and $C_a'= \frac{2}{1-\sqrt{\gamma}} \max( \frac{16\bar\epsilon  h\tau}{\sigma}, \frac{2 \bar{r}}{1-\gamma}{(\tau+t_0)})$,
%	\guannan{I seriously doubt if putting lower bound on $T$ is more readable.} \adam{agreed}
with $\bar{\epsilon} =4 \frac{\bar{r}}{1-\gamma} + 2\bar{r} $ and we recall that $\fk=\max_i|\nik|$ is the size of the largest $\khop$-neighborhood.
Then, with probability at least $1-\delta$,
{\small\begin{align}
 \frac{\sum_{m=0}^{M-1} \eta_m\Vert\nabla J(\theta(m))\Vert^2}{\sum_{m=0}^{M-1} \eta_m } \leq \frac{\frac{2\bar{r}}{\eta(1-\gamma)}+ \frac{8\bar{r}^2 L^2}{(1-\gamma)^4} \sqrt{\log M \log\frac{4}{\delta} }  + \frac{96 \bar{r}^2 L' L^2}{(1-\gamma)^4}\eta \log M}{ \sqrt{M+1}} + \frac{12 L^2 c\bar{r}}{(1-\gamma)^5} \rhok  . \label{eq:main_theorem_bound}
	\end{align}}    
\end{theorem}

The proof of Theorem~\ref{thm:main} is deferred to Section~\ref{sec:proof}. To interpret the result, note that the first term in \eqref{eq:main_theorem_bound} converges to $0$ in the order of $\tilde O(\frac{1}{\sqrt{M}})$ and the second term, which we denote as $\varepsilon_\khop$, is the bias caused by the truncation of the $Q$-function and it scales in the order of $O(\rhok)$. As such, our method \algoname\ will eventually find an $O(\rhok)$-approximation of a stationary point of the objective function $J(\theta)$, which could be very close to a true stationary point even for small $\khop$ as $\varepsilon_\khop$ decays exponentially in $\khop$. 

In terms of complexity, \eqref{eq:main_theorem_bound} gives that, to reach a $O(\varepsilon_\khop)$-approximate stationary point, the number of outer-loop iterations required is $M \geq \tilde{\Omega}(\frac{1}{\varepsilon_\khop^2} poly(\bar{r},L,L', \frac{1}{(1-\gamma)}))$, which scales polynomially with the parameters of the problem. We emphasize that it does not scale exponentially with $n$. Further, since the left hand side of \eqref{eq:main_theorem_T_bound} decays to $0$ as $T$ increases in the order of $\tilde{O}(\frac{1}{\sqrt{T}})$ and the right hand side of \eqref{eq:main_theorem_T_bound} is in the same order as $O(\varepsilon_\khop)$, the inner-loop length required is $T\geq \tilde{\Omega}(\frac{1}{\varepsilon_k^2}poly(\tau,\frac{1}{\sigma},\frac{1}{1-\gamma},\bar{r}, \fk ))$. 
This iteration complexity for the inner loop can potentially be further reduced if we do a warm start for the inner-loop, as the $Q$-estimate from the previous outer-loop should be already a good estimate for the current outer-loop. We leave the finite time analysis of the warm start variant as future work.

In the complexity bound, a key parameter is $\sigma$, which we recall is defined in Assumption~\ref{assump:local_explor} and it roughly means the probability that a state-action pair in a $\khop$-hop neighborhood is visited. Suppose we interpret $\sigma$ to scale with $\sigma\sim \frac{1}{(|S||A|)^{f(\khop)}}$, where we recall $\fk$ is the size of the largest $\khop$-hop neighborhood around any node, and $(|S||A|)^{f(\khop)}$ is the largest state-action space size of $\khop$-hop neighborhoods of any node. Then, the iteration complexity scales with $\frac{1}{\sigma} \sim  (|S||A|)^{f(\khop)}$, whereas the steady state error depends on $\rhok$. Therefore, $\khop$ is a parameter that balances between complexity and performance -- the larger $\khop$ is, the higher the complexity but the smaller the steady state error. Exactly how the complexity grows depends on $f(\kappa)$, the size of $\khop$-hop neighborhoods, which in turn depends on the topology of the interaction graph. On one hand, for a sparse graph where $f(\khop)$ is a constant much smaller than the number of nodes $n$, the state-action space size of $\khop$-hop neighborhoods is much smaller than the global state-action space size, in which case our algorithm can avoid the exponential scaling in $n$ and is scalable to implement. In the case where the graph is very dense or even complete, we have $\fk= \Omega(n)$ for any $\khop>0$ and our algorithm still suffers from the curse of dimensionality as the complexity scales with $(|S||A|)^{\Omega(n)}$. However, in the case of dense or complete graphs, the local interaction structure becomes degenerate as it takes exponentially many parameters to even specify the local transition probabilities in \eqref{eq:transition_factor}. How to handle the dense or complete graph case remains an interesting future direction, and we believe more structural assumptions are needed to break the curse of dimensionality in that case. 
 
%Parameters $\tau$ and $\frac{1}{\sigma}$ are from Assumption~\ref{assump:local_explor} and they scale with the local state-action space size of the largest $\khop$-hop neighborhood. 
%Therefore, the inner-loop length $T$ required scales with the size of the local state-action space of the largest $\khop$-neighborhood, which is much smaller than the full state-action space size when the graph is sparse. 

% \revision{
% Exactly how the complexity is depends on $f(\kappa)$, the size of $\khop$ neighborhoods which depends on the topology of the interaction graph, and the sparser the graph, the smaller the complexity. 
% }

Another thing to note in Theorem~\ref{thm:main} is that our guarantee is an upper bound on the running average of the squared norm of the gradient and is essentially a local convergence guarantee. This kind of local convergence is typical for actor critic methods, see e.g. \cite{rl_konda2000actor,zhang2018fully}. Recently, there have been works studying the optimization landscape for policy optimization, showing that in single-agent settings and for certain parameterizations of the policy, the objective of the policy optimization problem $J(\theta)$ may satisfy the gradient dominance property, indicating any stationary point will be a global optimum \citep{bhandari2019global,agarwal2019theory}. One interesting future direction is to show whether similar properties hold in our multi-agent networked setting with the local interaction structure, and whether the local convergence guarantee can imply a global optimality guarantee.

When put in the context of the broader literature in multi-agent RL, our contribution can be interpreted as follows. In multi-agent RL, the typical way to handle the curse of dimensionality is to use function approximation of the $Q$ function. For example, \citet{zhang2018fully} uses linear function to approximate the $Q$-function. Alternatively, it has also been popular to use neural networks to do the approximation \citep{lowe2017multi}. However, in these approximation methods, the resulting steady state error in the actor critic framework depends on the approximation error, and it is generally unclear how to choose the function approximator that is both computationally tractable and also accurate in representing the true $Q$ function. In the context of these works, our truncated $Q$-functions can be viewed as a specific way of function approximation that exploits the local interaction structure in our problem, and our method is not only computationally tractable but also has a small approximation error. 
\section{Proof of Main Result}\label{sec:proof}
In this section, we provide the proof of the main result Theorem~\ref{thm:main} with some auxiliary derivations postponed to the appendix. As our algorithm is an actor critic algorithm, the proof is divided into two parts: firstly, we provide an analysis of the critic, i.e. TD learning that estimates the truncated $Q$-function; secondly, we analyze the actor and finish the proof of Theorem~\ref{thm:main}.

\textbf{Analysis of the critic. }The first part of the analysis concerns the critic, and we show that the critic inner loop converges to an estimate with steady-state error exponentially small in $\khop$. Specifically, recall that within iteration $m$ the inner loop update is
\begin{subequations}\label{appendix:critic:td_original}
\begin{align}
    \hat{Q}_i^{t}(s_{\nik}(t-1),a_{\nik}(t-1)) &= (1 - \alpha_{t-1}) \hat{Q}_i^{t-1}(s_{\nik}(t-1),a_{\nik}(t-1))\nonumber \\
    &\quad + \alpha_{t-1} (r_i(s_i(t-1),a_i(t-1)) +  \gamma \hat{Q}_i^{t-1}(s_{\nik}(t),a_{\nik}(t)) ), \label{appendix:critic:td_original_1}\\
     \hat{Q}_i^{t}(s_{\nik},a_{\nik}) &=  \hat{Q}_i^{t-1}(s_{\nik},a_{\nik}) \text{ for } (s_{\nik},a_{\nik})\neq (s_{\nik}(t-1),a_{\nik}(t-1)) , \label{appendix:critic:td_original_2}
\end{align}
\end{subequations}
where $\hat{Q}_i^{0}\in \R^{\mathcal{S}_{\nik}\times \mathcal{A}_{\nik}}$ is initialized to be all zero, and $\alpha_t = \frac{h}{t+t_0}$ is the step size. We note that when implementing \eqref{appendix:critic:td_original} within outer loop iteration $m$, trajectory $(s(t),a(t))$ is generated by the agents taking a fixed policy $\theta(m)$. Let $Q_i^{\theta(m)} \in \R^{\mathcal{S}\times\mathcal{A}}$ be the true $Q$-function for reward $r_i$ under this fixed policy $\theta(m)$ as defined in \eqref{eq:full_q}. Given the above notation, we prove the following theorem on the critic, which bounds the error between the approximation $\hat{Q}_i^T$ generated by \eqref{appendix:critic:td_original} and the true $Q_i^{\theta(m)}$.

\begin{theorem}\label{appendix:critic:thm}
Assume Assumption \ref{assump:bounded_reward}, \ref{assump:exp_decaying}, \ref{assump:local_explor} are true and suppose $t_0, h$ satisfies, $h\geq \frac{1}{\sigma}\max(2,\frac{1}{1-\sqrt{\gamma}})$ and $t_0\geq\max(2h, 4\sigma h, \tau)$.
Then, inside outer loop iteration $m$, for each $i\in\mathcal{N}$, with probability at least $1-\delta$, we have the following error bound,
$$ \sup_{(s,a)\in\mathcal{S}\times \mathcal{A}} \big| Q_i^{\theta(m)}(s,a) - \hat{Q}_i^T(s_{\nik},a_{\nik}) \big| \leq \frac{C_a}{\sqrt{T+t_0}} + \frac{C_a'}{T+t_0} + \frac{2c\rhok}{(1-\gamma)^2}, $$
where 
$$C_a=   \frac{6\bar{\epsilon}}{1-\sqrt{\gamma}}  \sqrt{  \frac{\tau h}{\sigma}  [ \log(\frac{2\tau T^2 }{\delta}) + \fk\log SA]} ,\quad C_a'= \frac{2}{1-\sqrt{\gamma}} \max( \frac{16\bar\epsilon  h\tau}{\sigma}, \frac{2 \bar{r}}{1-\gamma}{(\tau+t_0)}),$$
with $\bar{\epsilon} =4 \frac{\bar{r}}{1-\gamma} + 2\bar{r} $.
\end{theorem}
The proof of Theorem~\ref{appendix:critic:thm} is postponed to Section~\ref{subsec:analysis_critic}. The result in Theorem~\ref{appendix:critic:thm} is an upper bound of the infinity-norm error between the truncated $Q$-function $\hat{Q}_i^T$ obtained by TD learning and the true $Q$-function. 
This error bound can be further decomposed into two parts - a transient part that converges to zero in the order of $\tilde{O}(\frac{1}{\sqrt{T}} + \frac{1}{T})$, and a steady state error that is exponentially small in $\kappa$. In proving Theorem~\ref{appendix:critic:thm}, we use two key techniques. The first is using the exponential decay property (cf. Definition~\ref{def:exp_decaying}, Lemma \ref{lem:exp_decaying}) to show that in the ``steady state'', the error of the truncated $Q$-function is bounded by $O(\rhok)$. This is possible due to our results in Lemma~\ref{lem:truncated_pg} which shows the class of truncated $Q$-functions are good approximations of the full $Q$-function. Our second proof technique is that we develop novel finite time analysis tools for TD learning to obtain the finite time error bound. Our proof uses a novel recursive decomposition of the error. Compared to existing work on finite time analysis on TD learning in \cite{srikant2019finite}, our result does not need the ergodicity assumption (the weaker Assumption~\ref{assump:local_explor} instead), and obtains error bounds in the infinity norm as opposed to the (weighted) Euclidean norm in \citet{srikant2019finite}. This finite time proof technique could be of independent interest. 

\textbf{Analysis of the actor.} We view the actor step as a biased stochastic gradient step, with the bias characterized by our result in the critic (Theorem~\ref{appendix:critic:thm}). Under this viewpoint, we finish the analysis of the actor and the proof of the main result in Section~\ref{subsec:analysis_actor}.

\subsection{Analysis of the Critic: Proof of Theorem~\ref{appendix:critic:thm}}\label{subsec:analysis_critic}
Since Theorem~\ref{appendix:critic:thm} is entirely about a particular outer-loop iteration $m$, inside which the policy is fixed to be $\theta(m)$, to simplify notation we drop the dependence on $m$ and $\theta(m)$ throughout this section. Particularly, we refer to $Q_i^{\theta(m)}$ as $Q_i^*$, which is the true $Q$-function for reward $r_i$ under policy $\theta(m)$ (cf. \eqref{eq:full_q}). 
We also introduce short-hand notation $z =(s,a) \in \mathcal{Z} = \mathcal{S}\times \mathcal{A}$ to represent a particular state action pair $ (s,a) \in \mathcal{S}\times \mathcal{A}$. 
Similarly, we define $z_i = (s_i,a_i)\in \mathcal{Z}_i = \mathcal{S}_i\times \mathcal{A}_i$, and $z_{\nik} = (s_{\nik},a_{\nik})  \in\mathcal{Z}_{\nik} = \mathcal{S}_{\nik}\times \mathcal{A}_{\nik}$. 
A vector $v\in\R^{\mathcal{Z}}$ means a vector of dimension $|\mathcal{Z}|$ that is indexed by $z\in\mathcal{Z}$, with its $z$'th entry denoted by $v(z)$. 
For example, the full $Q$-function $Q_i^*$ will be treated as a vector in $\R^{\mathcal{Z}}$ with its $z$'th entry denoted by $Q_i^*(z)$. 
Similarly, a vector $v\in\R^{\mathcal{Z}_{\nik}}$ means a vector of dimension $|\mathcal{Z}_{\nik}|$ indexed by $z_{\nik}\in\mathcal{Z}_{\nik}$, and its $z_{\nik}$'th entry will be denoted by $v(z_{\nik})$. 
For example, the truncated $Q$-functions $\hat{Q}_i^t$ will be treated as vectors in $\R^{\mathcal{Z}_{\nik}}$. Following a similar convention, a matrix $A\in\R^{\mathcal{\mathcal{Z}\times\mathcal{Z}}}$ will be a $|\mathcal{Z}|$-by-$|\mathcal{Z}|$ matrix indexed by $(z,z') \in\mathcal{Z}\times\mathcal{Z}$, with its $(z,z')$'th entry denoted by $A(z,z')$. 

Theorem~\ref{appendix:critic:thm} essentially says that the critic iterate $\hat{Q}_i^t$ in \eqref{appendix:critic:td_original} will become a good estimate of $Q_i^*$ as $t$ increases. We note that the full $Q$-function $Q_i^*$ must satisfy the Bellman equation \citep{bertsekas1996neuro},
\begin{align}
    Q_i^* = \td(Q_i^*):= r_i + \gamma P Q_i^*, \label{appendix:critic:eq:bellman}
\end{align}
where $\td:\R^{\mathcal{Z}}\rightarrow\R^{\mathcal{Z}}$ is the standard Bellman operator for reward $r_i$ and $P\in \R^{\mathcal{Z}\times\mathcal{Z}}$ is the transition probability matrix from $z(t)$ to $z(t+1)$ under policy $\theta(m)$. Note in \eqref{appendix:critic:eq:bellman}, without causing any confusion, $r_i$ is interpreted as a vector in $\R^{\mathcal{Z}}$ although $r_i$ only depends on $z_i$. 

%\lina{I think it will help readers to understand equations in this section better if you explain how to treat $Q$ as a vector. And how to treat $P$ accordingly. it might be better to explain $\mathbb{R}^{\mathcal{S}\times\mathcal{A}}$ as well.} \guannan{I added a few paragraphs... not sure if it is enough. I reorganize and do not talk about $\R^{\mathcal{S}\times\mathcal{A}}$ at all to avoid confusion. } \lina{I think the current description is good.}

Our proof is divided into 3 steps. In Step 1, we rewrite \eqref{appendix:critic:td_original_1} and \eqref{appendix:critic:td_original_2} in a linear update form (cf. \eqref{appendix:eq:critic:linear_update}), study its behavior (Lemma~\ref{lem:critic:barab}), and decompose the error into a recursive form (Lemma~\ref{appendix:critic:lem:error_recursive}). In Step 2, we bound the noise sequences in the error decomposition (Lemma~\ref{appendix:critic:lem:martingale_bound} and Lemma~\ref{appendix:critic:lem:phi_bound}). Finally, in Step 3, we use the recursive error decomposition and the bound on the noise sequences to prove Theorem~\ref{appendix:critic:thm}.

\textbf{Step 1: error decomposition. } Define $\mathbf{e}_{z_{\nik}} $ to be the indicator vector in $\mathbb{R}^{\mathcal{Z}_{\nik}}$, i.e. the $z_{\nik}$'th entry of $\mathbf{e}_{z_{\nik}} $ is $1$ and other entries are zero. Then, the critic update equation \eqref{appendix:critic:td_original} can be written as, 
\begin{align}
\hat{Q}_i^{t} = \hat{Q}_i^{t-1} + \alpha_{t-1}\cdot\mathbf{e}_{z_{\nik}(t-1)}\cdot[r_i(z_i(t-1))+ \gamma\hat{Q}_i^{t-1}(z_{\nik}(t)) - \hat{Q}_i^{t-1}(z_{\nik}(t-1)) ], \label{eq:q_update_vec_0}
\end{align}
with $\hat{Q}_i^0$ being the all zero vector in $\R^{\mathcal{Z}_{\nik}}$.  Notice that $\hat{Q}_i^{t-1}(z_{\nik}) = \mathbf{e}_{z_{\nik}}^\top \hat{Q}_i^{t-1}$, we can make the following definition (where we omit the dependence on $i$ in notation $A_{z,z'}$, $b_z$),
%This section aims to analyze \eqref{eq:q_update_vec_0} within a innerloop, where the policy is fixed to be some $\theta$. %For this reason, throughout his section we 
\begin{align*}
    A_{z,z'} =\mathbf{e}_{z_{\nik}} [ \gamma   \mathbf{e}_{z_{\nik}'}^\top - \mathbf{e}_{z_{\nik}}^\top ] \in \mathbb{R}^{\mathcal{Z}_{\nik}\times \mathcal{Z}_{\nik} },\quad
    b_{z} = \mathbf{e}_{z_{\nik}} r_i(z_i) \in\R^{\mathcal{Z}_{\nik}},
\end{align*}
% \lina{to be more rigorous, should $A, b$ have an index $i$?} \guannan{Right, but I intentially ignore that to make the notation cleaner; further, in this section we only consider a fixed $i$ so there won't be abuse of notation. I change the dependence on $z$ to be in the subscript to avoid confusion with the $z$'th entry of $A$... }\lina{How about add a short phrase that without causing confusion, we omit the index $i$ for $A$ and $b$. The reason is that the section still full of $Q_i$, $Z_i$, $z_{N_{i}}$. When reading (14), it might make readers pause to think whether all agents' $A_{z,z'}$ turn out to be the same? I paused and think about it...}
and rewrite \eqref{eq:q_update_vec_0} in a linear form
\begin{align}
    \hat{Q}_i^t = \hat{Q}_i^{t-1} + \alpha_{t-1}\big[A_{z(t-1),z(t)} \hat{Q}_i^{t-1} + b_{z(t-1) } \big]. \label{appendix:eq:critic:linear_update}
\end{align}
We define the following simplifying notations, $
    A_{t-1} = A_{z(t-1),z(t)}$,  $
    b_{t-1} = b_{z(t-1)}$. Let $\mathcal{F}_t$ be the $\sigma$-algebra generated by $z(0),\ldots,z(t)$. Then, clearly $A_{t-1}$ is $\mathcal{F}_t$-measurable and $b_{t-1}$ is $\mathcal{F}_{t-1}$ measurable. As a result, 
$\hat{Q}_i^t$ is $\mathcal{F}_t$-measurable. 
%Our update equation \eqref{eq:q_update_vec} can be written as,
%\begin{align}
%    \hat{Q}_i^t = \hat{Q}_i^{t-1} + \alpha_{t-1}\big[A_{t-1} \hat{Q}_i^{t-1} + b_{t-1} ] 
%\end{align}
Let $\tau>0$ to be the integer in Assumption~\ref{assump:local_explor}. Let $d_{t-1}\in \R^{\mathcal{Z}}$ be the distribution of $z(t-1)$ conditioned on $\mathcal{F}_{t-\tau}$. Further define, $
    \bar{A}_{t-1} = \E A_{t-1} | \mathcal{F}_{t-\tau}$, $ \bar{b}_{t-1} = \E b_{t-1} | \mathcal{F}_{t-\tau}$.
%i.e. the ``averaged'' $A$ and $b$ under distribution $d_{t-1}$. 
It is clear that $d_{t-1}$, $\bar{A}_{t-1}$, $\bar{b}_{t-1}$ are all $\mathcal{F}_{t-\tau}$ measurable random vectors (matrices).
With these notations, \eqref{appendix:eq:critic:linear_update} can be rewritten as,  
\begin{align}
    \hat{Q}_i^t &=\hat{Q}_i^{t-1} + \alpha_{t-1}\big[A_{t-1} \hat{Q}_i^{t-1} + b_{t-1} ] \nonumber  \\
    &= \hat{Q}_i^{t-1} + \alpha_{t-1}\big[ \bar{A}_{t-1} \hat{Q}_i^{t-1} + \bar{b}_{t-1} ] + \alpha_{t-1}[(A_{t-1} - \bar{A}_{t-1})\hat{Q}_i^{t-1} + b_{t-1} -\bar{b}_{t-1}] \nonumber \\
    &= \hat{Q}_i^{t-1} + \alpha_{t-1}\big[\bar{A}_{t-1} \hat{Q}_i^{t-1} + \bar{b}_{t-1} ] \nonumber \\
    &\quad + \alpha_{t-1}\underbrace{[(A_{t-1} - \bar{A}_{t-1})\hat{Q}_i^{t-\tau} + b_{t-1} - \bar{b}_{t-1}]}_{:=\epsilon_{t-1}} + \alpha_{t-1} \underbrace{(A_{t-1} - \bar{A}_{t-1})(\hat{Q}_i^{t-1} - \hat{Q}_i^{t-\tau}) }_{:=\phi_{t-1}} , \label{appendix:critic:eq:q_epsilon_phi}
\end{align}where in the last step, we have defined sequences $\epsilon_{t-1}, \phi_{t-1} \in\R^{\mathcal{Z}_{\nik}}$, which are noise sequences that we will later control in Step 2. For now, we focus on the term $\bar{A}_{t-1} \hat{Q}_i^{t-1} + \bar{b}_{t-1}$, and show the following Lemma~\ref{lem:critic:barab}. The proof of Lemma~\ref{lem:critic:barab} is similar to the analysis of fixed points for TD learning with linear function approximation, see e.g. \cite{van1995stable}. We postpone the detailed proof of Lemma~\ref{lem:critic:barab} to Appendix~\ref{subsec:proof_barab}. 
\begin{lemma}\label{lem:critic:barab}
$\forall t$, there exists diagonal matrix $D_{t-1}\in \R^{\mathcal{Z}_{\nik}\times \mathcal{Z}_{\nik}}  $ and operator $g_{t-1}:\R^{\mathcal{Z}_{\nik}} \rightarrow \R^{\mathcal{Z}_{\nik}}$ s.t.
\begin{align}
     \bar{A}_{t-1} \hat{Q}_i^{t-1} + \bar{b}_{t-1} = - D_{t-1} \hat{Q}_i^{t-1} + D_{t-1} g_{t-1}(\hat{Q}_i^{t-1}), \label{appendix:critic:eq:aqb_g} 
\end{align}
where $D_{t-1}$ satisfies $D_{t-1} \succeq  \sigma I$ (recall $\sigma >0$ is from Assumption~\ref{assump:local_explor}) with its $z_{\nik}$'th diagonal entry being $\bar{d}_{t-1}(z_{\nik}) = \sum_{z_{\nminusik}\in\mathcal{Z}_{\nminusik}} d_{t-1}(z_{\nik},z_{\nminusik})$, which means the marginalized distribution of $z_{\nik}$ under distribution $d_{t-1}$. Further, $g_{t-1}$ is a $\gamma$-contraction in the infinity norm, with unique fixed point $\hat{Q}_i^{*,{t-1}}$ satisfying $ \sup_{z\in\mathcal{Z}} |  \hat{Q}_i^{*,{t-1}} (z_{\nik}) - Q_i^*(z) | \leq \frac{ c \rhok}{1-\gamma}$.
\end{lemma}
%\lina{there are some notations making me confused about lemma 4} \guannan{I changed to $\hat{Q}_i^{*,t-1}$... not sure if it is better. } \lina{I prefer changing $\bar{d}_{t-1}(z_{\nik})$ to be $\bar{d}_{t-1}(z_{\nik})$ for the marginal distribution.  } 
From Lemma~\ref{lem:critic:barab}, we can see the first two terms in \eqref{appendix:critic:eq:q_epsilon_phi} can be written as $(I - \alpha_{t-1} D_{t-1})\hat{Q}_i^{t-1} + \alpha_{t-1} D_{t-1} g_{t-1}(\hat{Q}_i^{t-1})$, which roughly speaking drives the iterate to $\hat{Q}_i^{*,t-1}$, the fixed point of operator $g_{t-1}$. Though depending on $t$, $\hat{Q}_i^{*,t-1}$ is a good approximation of the true $Q$-function $Q_i^*$ regardless of $t$, as shown in Lemma~\ref{lem:critic:barab}. As such, moving towards $\hat{Q}_i^{*,t-1}$ at each time step should eventually produce a good estimate of the full $Q$-function $Q_i^*$. In the following, we turn the above intuition into an error decomposition. In details, we unroll \eqref{appendix:critic:eq:q_epsilon_phi} and get, %, we plug \eqref{appendix:critic:eq:aqb_g} into \eqref{appendix:critic:eq:q_epsilon_phi} and expand it recursively, getting,
\begin{align}
    \hat{Q}_i^t &= (I - \alpha_{t-1} D_{t-1})\hat{Q}_i^{t-1} + \alpha_{t-1} D_{t-1} g_{t-1}(\hat{Q}_i^{t-1}) + \alpha_{t-1} \epsilon_{t-1} + \alpha_{t-1} \phi_{t-1} \nonumber\\
    &= \underbrace{\prod_{k=\tau}^{t-1} (I - \alpha_k D_{k})}_{\tilde{B}_{\tau-1,t}}\hat{Q}_i^\tau + \sum_{k=\tau}^{t-1}  \underbrace{\alpha_k D_k\prod_{\ell=k+1}^{t-1}(I - \alpha_\ell D_{\ell} )}_{B_{k,t}}   g_{k} (\hat{Q}_i^k) + \sum_{k=\tau}^{t-1}  \alpha_k \underbrace{ \prod_{\ell=k+1}^{t-1}(I - \alpha_\ell D_{\ell} )}_{\tilde{B}_{k,t}}  (\epsilon_k + \phi_k), \label{eq:hatq_recursive}
\end{align}
where we have used the following short-hand notations
$B_{k,t} = \alpha_k D_k\prod_{\ell=k+1}^{t-1}(I - \alpha_\ell D_{\ell} ), \tilde{B}_{k,t} = \prod_{\ell=k+1}^{t-1}(I - \alpha_\ell D_{\ell} )$. We also define
$\beta_{k,t} = \alpha_k \prod_{\ell=k+1}^{t-1} (1 - \alpha_\ell \sigma)$,  
$\tilde{\beta}_{k,t} =  \prod_{\ell=k+1}^{t-1} (1 - \alpha_\ell \sigma).$
Since every diagonal entry of $D_\ell$ is lower bounded by $\sigma$ almost surely (Lemma~\ref{lem:critic:barab}), we have every entry of $B_{k,t}$ is upper bounded by $\beta_{k,t}$  and every entry of $\tilde{B}_{k,t}$ is upper bounded by $\tilde{\beta}_{k,t}$ almost surely. With these short-hand notations, we state the error decomposition below, which is a consequence of \eqref{eq:hatq_recursive} and the property of operator $g_k(\cdot)$ and its fixed point $\hat{Q}_i^{*,k}$ in Lemma~\ref{lem:critic:barab}. The proof of Lemma~\ref{appendix:critic:lem:error_recursive} is postponed to Appendix~\ref{appendix:critic:subsec:error_recursive}.
%\eqref{eq:hatq_recursive} can be rewritten as 
% \begin{align}
%     \hat{Q}_i^t &= \tilde{B}_{\tau-1,t}\hat{Q}_i^\tau + \sum_{k=\tau}^{t-1}  B_{k,t}   {g}_{k} (\hat{Q}_i^k) + 
%  \sum_{k=\tau}^{t-1}  \alpha_k \tilde{B}_{k,t} \epsilon_k  + \sum_{k=\tau}^{t-1}  \alpha_k \tilde{B}_{k,t}  \phi_k. \label{eq:hatq_recursive_compact}
% \end{align}

% The goal of this step is to decompose the error. Let $a_t = \Vert \Phi \hat{Q}_i^t - Q_i^*\Vert_\infty = \sup_{z\in\mathcal{Z}} |\hat{Q}_i^t(z_{\nik}) - Q_i^*(z) |$ be the error at time $t$. 
% %From \eqref{eq:hatq_recursive_compact}, and also utilizing the $\gamma$-contraction of $g_k$ as well as the property of the fixed point of $g_k$ \eqref{appendix:critic:eq:g_fixedpoint}, 
% We show the following Lemma, which decomposes the error in a resursive form. Lemma~\ref{appendix:critic:lem:error_recursive} relies on (a) analyzing matrix $\bar{A}_{t-1}$, which is shown as an approximated version of the Bellman operatior, and show that like the standard Bellman operator, a certain contraction property hold for $A_{t-1}$; (b) then, unrolling ``unrolling'' equation~\eqref{appendix:critic:eq:q_epsilon_phi} back to the history. The detailed proof of Lemma~\ref{appendix:critic:lem:error_recursive} is postponed to Section~\ref{appendix:critic:subsec:error_recursive}.
\begin{lemma}\label{appendix:critic:lem:error_recursive}
Let $\xi_t = \sup_{z\in\mathcal{Z}}| \hat{Q}_i^t(z_{\nik}) - Q_i^*(z) |$. The following recursion holds almost surely,
\begin{align*}
\xi_t \leq \tilde{\beta}_{\tau-1,t}\xi_\tau  + \gamma \sup_{z_{\nik}\in\mathcal{Z}_{{\nik}}} \sum_{k=\tau}^{t-1}  b_{k,t}(z_{\nik})  \xi_k +  \frac{2c\rhok}{1-\gamma} +  \Vert \sum_{k=\tau}^{t-1}  \alpha_k \tilde{B}_{k,t} \epsilon_k\Vert_\infty  + \Vert \sum_{k=\tau}^{t-1}  \alpha_k \tilde{B}_{k,t}  \phi_k \Vert_\infty,
\end{align*}
where $b_{k,t}(z_{\nik}) = \alpha_k \bar{d}_k(z_{\nik}) \prod_{\ell=k+1}^{t-1}(1-\alpha_\ell \bar{d}_\ell(z_{\nik}))$ is the $z_{\nik}$'th diagonal entry of $B_{k,t}$% $\tilde{b}_{k,t}(z_{\nik}) =   \prod_{\ell=k+1}^{t-1}(1-\alpha_\ell \bar{d}_\ell(z_{\nik}))$ is the $z_{\nik}$'th diagonal entry of $\tilde{B}_{k,t}$.
%, in which $\bar{d}_k(z_{\nik})$ is the marginalized probability of $z_{\nik}$ under distribution $d_k$. Matrix $B_{k,t}, \tilde{B}_{k,t} \in \R^{\mathcal{Z}_{\nik}\times \mathcal{Z}_{\nik}}$ are diagonal matrices with entries $b_{k,t}$ and $\tilde{b}_{k,t}$ respectively. 
\end{lemma}
From Lemma~\ref{appendix:critic:lem:error_recursive}, it is clear that to bound the error $\xi_t$, we need to bound  $\Vert \sum_{k=\tau}^{t-1}  \alpha_k \tilde{B}_{k,t} \epsilon_k\Vert_\infty$ and $\Vert \sum_{k=\tau}^{t-1}  \alpha_k \tilde{B}_{k,t}  \phi_k \Vert_\infty$, which is the focus of the next step.

\noindent\textbf{Step 2: bound the $\epsilon_k$ and the $\phi_k$-sequence. } The goal of this step is to bound $\Vert \sum_{k=\tau}^{t-1}  \alpha_k \tilde{B}_{k,t} \epsilon_k\Vert_\infty$ and $\Vert \sum_{k=\tau}^{t-1}  \alpha_k \tilde{B}_{k,t}  \phi_k \Vert_\infty$. We first show some simple properties of $\hat{Q}_i^t,\epsilon_{t}$ and $\phi_t$ in Lemma~\ref{appendix:critic:lem:bounded}. Since every entry of $\alpha_k \tilde{B}_{k,t}$ is upper bounded by $\beta_{k,t}$, we also show some properties of $\beta_{k,t}$, $\tilde{\beta}_{k,t}$ in Lemma~\ref{appendix:critic:lem:stepsize}. The proofs of the two lemmas are postponed to Appendix~\ref{appendix:critic:subsec:auxilliary}.
\begin{lemma}\label{appendix:critic:lem:bounded}
We have almost surely, (a) $\Vert \hat{Q}_i^t\Vert_\infty \leq  \frac{\bar{r}}{1-\gamma}$; (b) $\Vert \epsilon_t\Vert_\infty \leq \bar{\epsilon} =4 \frac{\bar{r}}{1-\gamma} + 2\bar{r} $; (c) $\Vert \phi_{t}\Vert_\infty \leq 2 \bar{\epsilon} \sum_{k=t-\tau+1}^{t-1} \alpha_k $.
\end{lemma}
\begin{lemma}\label{appendix:critic:lem:stepsize}
If $\alpha_t = \frac{h}{t+t_0}$, where $ t_0\geq h>\frac{2}{\sigma}$ and $t_0\geq 4\sigma h$, and $t_0\geq \tau$, then $\beta_{k,t},\tilde{\beta}_{k,t}$ satisfies (a) $ \beta_{k,t}\leq  \frac{h}{k+t_0} \Big( \frac{k+1+t_0}{t+t_0}\Big)^{\sigma h}$, $ \tilde\beta_{k,t}\leq   \Big( \frac{k+1+t_0}{t+t_0}\Big)^{\sigma h}$; (b) $\sum_{k=1}^{t-1}\beta_{k,t}^2 \leq \frac{2h}{\sigma } \frac{1}{(t+t_0)} $; (c) $\sum_{k=\tau}^{t-1}\beta_{k,t}\sum_{\ell = k-\tau+1}^{k-1} \alpha_\ell \leq \frac{8h\tau}{\sigma} \frac{1}{t+t_0} $.
\end{lemma}

%Recall that,
% \begin{align*}
%     \epsilon_{t-1} &=(A_{t-1} - \bar{A}_{t-1})\hat{Q}_i^{t-\tau} + b_{t-1} - \bar{b}_{t-1},\quad
%     \phi_{t-1} = (A_{t-1} - \bar{A}_{t-1})(\hat{Q}_i^{t-1} - \hat{Q}_i^{t-\tau}).
% \end{align*}
We now bound $\Vert \sum_{k=\tau}^{t-1}  \alpha_k \tilde{B}_{k,t} \epsilon_k\Vert_\infty$. Clearly, $\epsilon_{t-1}$ is $\mathcal{F}_t$-measurable, and satisfies 
\begin{align}
    \E [\epsilon_{t-1} | \mathcal{F}_{t-\tau} ]&= \E [(A_{t-1} - \bar{A}_{t-1})\hat{Q}_i^{t-\tau} + b_{t-1} - \bar{b}_{t-1} | \mathcal{F}_{t-\tau}]% \nonumber \\
    %&=\E [(A_{t-1} - \bar{A}_{t-1}) | \mathcal{F}_{t-\tau}]  \hat{Q}_i^{t-\tau} +\E[ b_{t-1} - \bar{b}_{t-1} | \mathcal{F}_{t-\tau}] 
    =0, \label{appendix:critic:eq:martingale}
\end{align}
where the last equality is due to the definition of $\bar{A}_{t-1}$ and $\bar{b}_{t-1}$ and the fact $\hat{Q}_i^{t-\tau}$ is $\mathcal{F}_{t-\tau}$-measurable. 
%we have used 
% \begin{align*}
%     \E [A_{t-1}|\mathcal{F}_{t-\tau}] &=  \E[A(z(t-1),z(t))|\mathcal{F}_{t-\tau}]=  \E \tilde{A} (z(t-1))|\mathcal{F}_{t-\tau}= \bar{A}^{d_{t-1}} = \bar{A}_{t-1},\\
%     \E [b_{t-1}|\mathcal{F}_{t-\tau}] &= \E b(z(t-1)) |\mathcal{F}_{t-\tau} = \bar{b}^{d_{t-1}} = \bar{b}_{t-1} ,
% \end{align*}
% per the definition of $d_{t-1}$. 
Equation~\eqref{appendix:critic:eq:martingale} shows that $\epsilon_{t-1}$ is a ``shifted'' martingale difference sequence (it is not a standard martingale difference sequence which would require $\E \epsilon_{t-1}| \mathcal{F}_{t-1} = 0$). Therefore, $\Vert \sum_{k=\tau}^{t-1}  \alpha_k \tilde{B}_{k,t} \epsilon_k\Vert_\infty$ can be controlled by Azuma-Hoeffding type inequalities, as shown by Lemma~\ref{appendix:critic:lem:martingale_bound}. We comment that $\tilde{B}_{k,t}$ is also random and $\tilde{B}_{k,t}\epsilon_k$ is no longer a martingale difference sequence. As a result, to prove Lemma~\ref{appendix:critic:lem:martingale_bound} requires more than the direct application of the Azuma-Hoeffding bound. For more details, see the full proof of Lemma~\ref{appendix:critic:lem:martingale_bound} in Appendix~\ref{appendix:critic:subsec:epsilon_phi}.

\begin{lemma}\label{appendix:critic:lem:martingale_bound}
We have with probability $1-\delta$,
$$\Big\Vert \sum_{k=\tau}^{t-1} \alpha_k \tilde{B}_{k,t} \epsilon_k \Big\Vert_\infty \leq 6 \bar{\epsilon} \sqrt{  \frac{\tau h}{\sigma(t+t_0)}  [ \log(\frac{2\tau t }{\delta}) + \fk\log SA]} .$$
\end{lemma}

Finally we bound sequence $\Vert \sum_{k=\tau}^{t-1}  \alpha_k \tilde{B}_{k,t}  \phi_k \Vert_\infty$ using the fact $\phi_{t-1} = (A_{t-1} - \bar{A}_{t-1})(\hat{Q}_i^{t-1} - \hat{Q}_i^{t-\tau})$ can be bounded by the movement of $\hat{Q}_i^t$ after $\tau$ steps (i.e. $\Vert \hat{Q}_i^{t-1} - \hat{Q}_i^{t-\tau}\Vert_\infty$), which is small due to the step size selection. The proof of Lemma~\ref{appendix:critic:lem:phi_bound} can also be found in Appendix~\ref{appendix:critic:subsec:epsilon_phi}.

\begin{lemma}\label{appendix:critic:lem:phi_bound} We have almost surely,  
$ \Vert \sum_{k=\tau}^{t-1}  \alpha_k \tilde{B}_{k,t}  \phi_k \Vert_\infty \leq \frac{16 \bar\epsilon h\tau}{\sigma} \frac{1}{t+t_0}:=C_\phi \frac{1}{t+t_0}$.
\end{lemma}

\textbf{Step 3: bound the critic error. } We are now ready to use the error decomposition in Lemma~\ref{appendix:critic:lem:error_recursive} as well as the bounds on the $\epsilon_k$, $\phi_k$-sequences in Lemma~\ref{appendix:critic:lem:martingale_bound} and Lemma~\ref{appendix:critic:lem:phi_bound} to bound the error of the critic. Recall that Theorem~\ref{appendix:critic:thm} states with probability $1-\delta$,
%Suppose $t_0, h$ satisfies, $h\geq \frac{1}{\sigma}\max(2,\frac{1}{1-\sqrt{\gamma}})$, $t_0\geq\max(2h, 4\sigma h, \tau)$.
\begin{align}
    \xi_T\leq \frac{C_a}{\sqrt{T+t_0}} + \frac{C_a'}{T+t_0} + \frac{C_0}{1-\gamma}, \label{appendix:critic:eq:step5_errorbound}
\end{align}
where $C_0 = \frac{2c \rhok}{1-\gamma} $, and
$$C_a=   \frac{6\bar{\epsilon}}{1-\sqrt{\gamma}}  \sqrt{  \frac{\tau h}{\sigma}  [ \log(\frac{2\tau T^2 }{\delta}) + \fk\log SA]} ,\quad  C_a'= \frac{2}{1-\sqrt{\gamma}} \max( \frac{16\bar\epsilon  h\tau}{\sigma}, \frac{2 \bar{r}}{1-\gamma}{(\tau+t_0)}).$$

    To prove \eqref{appendix:critic:eq:step5_errorbound}, we start by applying Lemma~\ref{appendix:critic:lem:martingale_bound} to $t\leq T$ with $\delta$ replaced by $\delta/T$. Then, using a union bound, we get with probability $1-\delta$, for any $t\leq T$,
    $\Big\Vert \sum_{k=\tau}^{t-1} \alpha_k \tilde{B}_{k,t} \epsilon_k \Big\Vert_\infty \leq C_\epsilon \frac{1}{\sqrt{t+t_0}}$
 with $ C_\epsilon = 6 \bar{\epsilon} \sqrt{  \frac{\tau h}{\sigma}  [ \log(\frac{2\tau T^2 }{\delta}) + \fk\log SA]} $. Combining this with Lemma~\ref{appendix:critic:lem:error_recursive} and using Lemma~\ref{appendix:critic:lem:phi_bound}, we get with probability $1-\delta$, for all $\tau\leq t\leq T$,
\begin{align}
    \xi_t \leq \tilde{\beta}_{\tau-1,t}\xi_\tau  + \gamma \sup_{z_{\nik}}\sum_{k=\tau}^{t-1}  b_{k,t}(z_{\nik})  \xi_k + C_\epsilon\frac{1}{\sqrt{t+t_0}} + C_\phi \frac{1}{t+t_0} +C_0. \label{appendix:critic:eq:step_5_errorrecursive}
\end{align}
We now condition on \eqref{appendix:critic:eq:step_5_errorrecursive} is true and use induction to show \eqref{appendix:critic:eq:step5_errorbound}. Eq. \eqref{appendix:critic:eq:step5_errorbound} is true for $t=\tau$, as $\frac{C_a'}{\tau+t_0} \geq \frac{2}{1-\sqrt{\gamma}} \frac{2 \bar{r}}{1-\gamma} > \xi_\tau $, where we have used $|\xi_\tau| \leq \Vert Q_i^*\Vert_\infty + \Vert \hat{Q}_i^\tau\Vert_\infty \leq \frac{2\bar{r}}{1-\gamma}$. Then, assume \eqref{appendix:critic:eq:step5_errorbound}  is true for up to $k\leq t-1$, we have by \eqref{appendix:critic:eq:step_5_errorrecursive},
\begin{align*}
    \xi_t &\leq \tilde{\beta}_{\tau-1,t}\xi_\tau  + \gamma \sup_{z_{\nik}}\sum_{k=\tau}^{t-1}  b_{k,t}(z_{\nik})  [\frac{C_a}{\sqrt{k+t_0}} +\frac{C_a'}{k+t_0}+ \frac{C_0}{1-\gamma}] + C_\epsilon\frac{1}{\sqrt{t+t_0}} + C_\phi \frac{1}{t+t_0} +C_0\\
    &\leq \tilde{\beta}_{\tau-1,t}\xi_\tau  + \gamma C_a \sup_{z_{\nik}}\sum_{k=\tau}^{t-1}  b_{k,t}(z_{\nik})  \frac{1}{\sqrt{k+t_0}} + \gamma C_a' \sup_{z_{\nik}}\sum_{k=\tau}^{t-1}  b_{k,t}(z_{\nik})  \frac{1}{{k+t_0}} \\
    &\qquad + C_\epsilon\frac{1}{\sqrt{t+t_0}} + C_\phi \frac{1}{t+t_0} + \frac{C_0}{1-\gamma}.
\end{align*}
We use the following auxiliary Lemma, whose proof is provided in Section~\ref{appendix:critic:subsec:b_kt_bound}.
\begin{lemma}\label{lem:b_kt_bound}
Recall $\alpha_k = \frac{h}{k+t_0}$, and $b_{k,t}(z_{\nik}) = \alpha_k \bar{d}_k(z_{\nik}) \prod_{\ell=k+1}^{t-1}(1-\alpha_\ell \bar{d}_\ell(z_{\nik})) $, here $\bar{d}_k(z_{\nik})\geq \sigma$. If $\sigma h (1-\sqrt{\gamma}) \geq 1$, $t_0\geq 1$, and $\alpha_0\leq \frac{1}{2}$, then, for any $z_{\nik}$, and any $0<\omega\leq 1$,
$$\sum_{k=\tau}^{t-1}  b_{k,t}(z_{\nik}) \frac{1}{(k+t_0)^\omega} \leq \frac{1}{\sqrt{\gamma}{(t+t_0)}^\omega}.  $$
\end{lemma}
With Lemma~\ref{lem:b_kt_bound}, and using the bound on $\tilde{\beta}_{\tau-1,t}$ in Lemma~\ref{appendix:critic:lem:stepsize} (a), we have
\begin{align*}
    \xi_t& \leq \tilde{\beta}_{\tau-1,t}\xi_\tau  + \sqrt{\gamma} C_a \frac{1}{\sqrt{t+t_0}} + \sqrt{\gamma} C_a' \frac{1}{{t+t_0}} + C_\epsilon\frac{1}{\sqrt{t+t_0}} + C_\phi \frac{1}{t+t_0} + \frac{C_0}{1-\gamma}\\
    &\leq  \underbrace{\sqrt{\gamma} C_a \frac{1}{\sqrt{t+t_0}} + C_\epsilon\frac{1}{\sqrt{t+t_0}}}_{:=F_t}  + \underbrace{\sqrt{\gamma} C_a' \frac{1}{{t+t_0}} + C_\phi \frac{1}{t+t_0}+ \Big(\frac{\tau+t_0}{t+t_0}\Big)^{\sigma h}\xi_\tau}_{:= F_t'}  + \frac{C_0}{1-\gamma} .
\end{align*}
To finish the induction, it suffices to show $F_t\leq \frac{C_a}{\sqrt{t+t_0}}$ and $F_t' \leq \frac{C_a'}{t+t_0}$. To see this, note
\begin{align*}
    F_t\frac{\sqrt{t+t_0}}{C_a} &= \sqrt{\gamma} + \frac{C_\epsilon}{C_a}, \quad
    F_t'\frac{t+t_0}{C_a'} = \sqrt{\gamma} +  \frac{C_\phi}{C_a'} + \frac{\xi_\tau(\tau+t_0)}{C_a'} \frac{(\tau+t_0)^{\sigma h-1}}{(t+t_0)^{\sigma h - 1}}.
\end{align*}
So, we can require $C_a, C_a'$ to be large enough such that
    $\frac{C_\epsilon}{C_a}  \leq  1 - \sqrt{\gamma}$, $
    \frac{C_\phi}{C_a'} \leq \frac{1 - \sqrt{\gamma}}{2}$, and $\frac{\xi_\tau {(\tau+t_0)}}{C_a'}  \leq \frac{1 - \sqrt{\gamma}}{2}$.
%which results in 
%$$C_a\geq \frac{1}{1-\sqrt{\gamma}} C_\epsilon,\quad C_a'\geq \frac{2}{1-\sqrt{\gamma}} \max( C_\phi, \xi_\tau(\tau+t_0))$$
Using $\xi_\tau\leq \frac{2\bar{r}}{1-\gamma}$, one can check our selection of $C_a$ and $C_a'$ satisfies the above three inequalities, and so the induction is finished and the proof of Theorem~\ref{appendix:critic:thm} is concluded. \qed

\subsection{Analysis of the Actor and Proof of Main Result (Theorem~\ref{thm:main})} \label{subsec:analysis_actor}
With the error of the critic bounded in Theorem~\ref{appendix:critic:thm}, the second part of the proof focuses on the actor, i.e. the policy gradient step. Recall that at iteration $m$, the policy gradient step is given by
$ \theta_i(m+1) = \theta_i(m) + \eta_m \hat{g}_i(m)$  
with $\eta_m = \frac{\eta}{\sqrt{m+1}}$ and $\hat{g}_i(m)$ is given by 
\begin{align}
    \hat{g}_i(m) =  \sum_{t=0}^{T} \gamma^t  \frac{1}{n} \sum_{j\in \nik} \hat{Q}_j^{m,T} (s_{\njk}(t),a_{\njk}(t)) \nabla_{\theta_i} \log \zeta_i^{\theta_i(m)} (a_i(t)|s_i(t)) , \label{appendix:actor:hat_g_i}
\end{align}
where $\hat{Q}_i^{m,T}$ is the final estimate of the $Q$-function for $r_i$ at the end of the critic loop in iteration $m$, where we have added an additional superscript $m$ to $\hat{Q}_i^{m,T}$ to indicate its dependence on $m$; $\{s(t), a(t)\}_{t=0}^T$ is the state-action trajectory with $s(0)$ drawn from $\pi_0$ (the initial state distribution defined in the objective function $J(\theta)$, cf. \eqref{eq:discounted_cost}) and the agents taking policy $\theta(m)$. Our goal is to show that $\hat{g}_i(m)$ is approximately the right gradient direction, $\nabla_{\theta_i}J(\theta(m))$, which by Lemma~\ref{lem:policy_grad} can be written as,
\begin{align}
    \nabla_{\theta_i} J(\theta(m)) = \sum_{t=0}^\infty \E_{s\sim\pi_t^{\theta(m)}, a\sim \zeta^{\theta(m)}(\cdot|s)} \left[\gamma^t Q^{\theta(m)}(s,a) \nabla_{\theta_i}\log\zeta^{\theta(m)}(a|s) \right],\label{appendix:actor:eq:grad_i}
\end{align}
where $\pi_t^{\theta(m)}$ is the distribution of $s(t)$ under fixed policy $\theta(m)$ when the initial state is drawn from $\pi_0$; $Q^{\theta(m)}$ is the true $Q$ function for the global reward $r$ under policy $\theta(m)$, cf. \eqref{eq:full_q}. 

To bound the difference between $\hat{g}_i(m)$ and the true gradient $\nabla_{\theta_i}J(\theta(m))$, we define the following additional sequences, 
\begin{align}
g_i (m) &= \sum_{t=0}^{T} \gamma^t \frac{1}{n} \sum_{j\in \nik} Q_j^{\theta(m)} (s(t),a(t)) \nabla_{\theta_i} \log \zeta_i^{\theta_i(m)} (a_i(t)|s_i(t)), \label{appendix:actor:eq:g_i} \\
h_i(m) &=  \sum_{t=0}^T \E_{s\sim\pi^{\theta(m)}_t,a \sim \zeta^{\theta(m)}(\cdot|s)}\left[ \gamma^t \frac{1}{n} \sum_{j\in {\nik}} {Q}_j^{\theta(m)} (s,a)\nabla_{\theta_i} \log\zeta_i^{\theta_i(m)} (a_i|s_i)\right],\label{appendix:actor:eq:h_i}
  %  \hat{g}_i (m)  &= \sum_{t=0}^{T} \gamma^t \frac{1}{n} \sum_{j\in \nik} \hat{Q}_j(s_{\njk}(t),a_{\njk}(t)) \nabla_{\theta_i} \log \zeta^{\theta_i(m)} (a_i(t)|s_i(t))  
\end{align}
where $Q_i^{\theta(m)}$ is the true $Q$ function for $r_i$ under policy $\theta(m)$. We also use notation $h(m)$, $g(m)$, $\hat{g}(m)$ to denote the respective $h_i(m), g_i(m), \hat{g}_i(m)$ stacked into a larger vector. The following result is an immediate consequence of Assumption~\ref{assump:bounded_reward} and Assumption~\ref{assump:gradient_bounded}, whose proof is postponed to Appendix~\ref{appendix:actor:subsec:upperbound}.
\begin{lemma}\label{appendix:actor:lem:upperbound} We have almost surely, $\forall m\leq M$,
   $ \max(\Vert\hat{g}(m)\Vert, \Vert g(m)\Vert, \Vert h(m)\Vert, \Vert \nabla J(\theta(m))\Vert ) \leq \frac{\bar{r} L}{(1-\gamma)^2} $. 
\end{lemma}

With these definitions, we decompose the error between the gradient estimator $\hat{g}(m)$ and the true gradient $\nabla J(\theta(m))$ into the following three terms,
\begin{align}
    \hat{g}(m) = \underbrace{\hat{g}(m) - g(m)}_{e^1(m)} + \underbrace{g(m) - h(m)}_{e^2(m)} + \underbrace{h(m) - \nabla J(\theta(m))}_{e^3(m)} + \nabla J(\theta(m)) \label{appendix:actor:eq:grad_err_decomposition}.
\end{align}
%$$\hat{g}_i(m) - \nabla_{\theta_i} J(\theta(m))= \hat{g}_i(m) - g_i (m) + g_i (m)- h_i(m) + h_i(m) - \nabla_{\theta_i} J(\theta(m))  $$
In the following, we will provide bounds on $e^1(m), e^2(m), e^3(m)$, and then combine these bounds to prove our main result Theorem~\ref{thm:main}. %In Step 1, we bound the first term $\Vert e^1(m)\Vert $ which is a direct consequence of our result in the analysis of the critic, cf. Theorem~\ref{appendix:critic:thm} in Appendix~\ref{sec:appendix_critic}. In Step 2, we study $e^2(m)$, which turns out to be a martingale difference sequence and can be controlled by Azuma-Hoeffding bound. In Step 3, we bound $e^3(m)$, and finally in Step 4, 

\noindent\textbf{Bounds on $ e^1(m) $. } Notice that the difference between $\hat{g}_i(m)$ and $g_i(m)$ is that the critic estimate $\hat{Q}_j^{m,T}$ is replaced with the true $Q$-function $Q_j^{\theta(m)}$. By Theorem~\ref{appendix:critic:thm}, we have $\hat{Q}_j^{m,T}$ will be very close to $Q_j^{\theta(m)}$ with high probability when $T$ is large enough, based on which we can bound $\Vert e^1(m)\Vert$, which is formally provided in Lemma~\ref{appendix:actor:lem:e1_bound}. The proof of Lemma~\ref{appendix:actor:lem:e1_bound} is postponed to Appendix~\ref{appendix:actor:subsec:e1_bound}.
\begin{lemma}\label{appendix:actor:lem:e1_bound}
     When $T$ is large enough s.t. $\frac{C_a(\frac{\delta}{2 nM},T)}{\sqrt{T+t_0}} + \frac{C_a'}{T+t_0}  \leq \frac{2c\rhok}{(1-\gamma)^2}  $, where $$C_a(\delta,T)=   \frac{6\bar{\epsilon}}{1-\sqrt{\gamma}}  \sqrt{  \frac{\tau h}{\sigma}  [ \log(\frac{2\tau T^2 }{\delta}) + \fk\log SA]} ,\quad C_a'= \frac{2}{1-\sqrt{\gamma}} \max( \frac{16\bar\epsilon  h\tau}{\sigma}, \frac{2 \bar{r}}{1-\gamma}{(\tau+t_0)}),$$
with $\bar{\epsilon} =4 \frac{\bar{r}}{1-\gamma} + 2\bar{r} $,
    %\guannan{Come back check the lower bound}
    % \begin{align*}
    %     T+t_0 &\geq \frac{4}{c\rhok} \max(\frac{96\bar{r} h\tau}{\sigma}, 2\bar{r}(\tau+t_0)  )\\
    %      T+t_0 &\geq \max(\frac{\tau h}{\sigma}[ \log\frac{4\tau n M}{\delta} + \fk \log SA] C_F, 4 \frac{C_F\tau h}{\sigma} \log [\frac{2C_F\tau h}{\sigma}]
    % \end{align*}
    % where $C_F= \frac{288 \bar{\epsilon}^2(1-\gamma)^2}{c^2\rho^{2\khop+2}}$ with $\bar{\epsilon} = 4\frac{\bar{r}}{1-\gamma} + 2\bar{r}$, 
    then we have with probability at least $1-\frac{\delta}{2}$, 
    $\sup_{0\leq m\leq M-1} \Vert e^1(m)\Vert \leq \frac{4c L\rhok}{(1-\gamma)^3}$.
\end{lemma}

\noindent\textbf{Bounds on $e^2(m)$.} Let $\mathcal{G}_m$ be the $\sigma$-algebra generated by the trajectories in the first $m$ outer-loop iterations. Then, $\theta(m)$ is $\mathcal{G}_{m-1}$ measurable, and so is $h_i(m)$. Further, by the way that the trajectory $\{(s(t),a(t))\}_{t=0}^T$ is generated, we have $\E [g(m) | \mathcal{G}_{m-1}] = h(m)$. As such, $ \eta_m\langle \nabla J(\theta(m)), e^2(m) \rangle$ is a martingale difference sequence w.r.t. $\mathcal{G}_m$, and we have the following bound in Lemma~\ref{appendix:actor:lem:martingale_bound} which is a direct consequence of Azuma-Hoeffding bound. The proof of Lemma~\ref{appendix:actor:lem:martingale_bound} is postponed to Section~\ref{appendix:actor:subsec:martingale_bound}. 

\begin{lemma}
\label{appendix:actor:lem:martingale_bound}
    	With probability at least $1-\delta/2$, we have
    	$$\Big| \sum_{m=0}^{M-1} \eta_m \langle \nabla J(\theta(m)), e^2(m) \rangle \Big| \leq  \frac{2\bar{r}^2 L^2}{(1-\gamma)^4} \sqrt{2\sum_{m=0}^{M-1} \eta_m^2 \log\frac{4}{\delta} }.$$
\end{lemma}

\noindent\textbf{Bounds on $e^3(m)$. } The term $e^3(m)$ is the error between $h(m)$ and the true gradient $\nabla J(\theta(m))$, where $h(m)$ is similar to the truncated policy gradient in Lemma~\ref{lem:truncated_pg} and therefore should be close to the true gradient by Lemma~\ref{lem:truncated_pg}. We provide Lemma~\ref{appendix:actor:lem:error_h_grad} below to bound $\Vert e^3(m)\Vert$. Due to technical issues, $h(m)$ is not exactly the same as the truncated policy gradient defined in Lemma~\ref{lem:truncated_pg}, but a variant of it instead. Therefore the conclusion of Lemma~\ref{lem:truncated_pg} can not be directly used in the proof of Lemma~\ref{appendix:actor:lem:error_h_grad}. Nevertheless, the proof of Lemma~\ref{appendix:actor:lem:error_h_grad} follows essentially the same arguments as in Lemma~\ref{lem:truncated_pg} and is postponed to Appendix~\ref{appendix:actor:subsec:error_h_grad}. %\guannan{Add more discussions here.}

\begin{lemma} \label{appendix:actor:lem:error_h_grad}
	 When $T+1\geq \frac{\log \left(\frac{c(1-\gamma)}{\bar{r}}\right)+ (\khop+1)\log\rho}{\log \gamma}  $, we have almost surely,
$\Vert e^3(m) \Vert\leq  2 \frac{L c }{(1-\gamma)}\rhok	$.
\end{lemma}
%We next define $g_i$, which is basically $\hat{g}_i$ but with the approximated $Q$-table replaced with the true $Q$-table.

\noindent\textbf{Combining the bounds and proof of Theorem~\ref{thm:main}.} With the above bounds on $e^1(m)$, $e^2(m)$ and $e^3(m)$, we are now ready to prove the main result Theorem~\ref{thm:main}. Since $\nabla J(\theta)$ is $L'$ Lipschitz continuous, we have
    \begin{align}
        J(\theta(m+1)) &\geq J(\theta(m)) + \langle \nabla J(\theta(m)), \theta(m+1) - \theta(m)\rangle - \frac{L'}{2}\Vert \theta(m+1) - \theta(m)\Vert^2 \nonumber\\
        &= J(\theta(m)) + \eta_m  \langle \nabla J(\theta(m)), \hat g(m)\rangle - \frac{L' \eta_m^2}{2} \Vert \hat{g}(m)\Vert^2 .\label{eq:actor:J_recursive}
    \end{align}
    Using the decomposition of $\hat{g}(m)$ in \eqref{appendix:actor:eq:grad_err_decomposition}, we get, 
    $$\Vert \hat{g}(m) \Vert^2 \leq 4 \Vert e^1(m)\Vert^2 + 4\Vert e^2(m)\Vert^2 + 4\Vert e^3(m)\Vert^2 + 4\Vert \nabla J(\theta(m))\Vert^2.$$
    Further, we can bound $\langle \nabla J(\theta(m)),\hat{g}(m)\rangle$,
\begin{align*}
    \langle \nabla J(\theta(m)), \hat g(m)\rangle  &= \Vert \nabla J(\theta(m))\Vert^2 + \langle \nabla J(\theta(m)), e^1(m) + e^2(m) + e^3(m)\rangle \\
    &\geq  \Vert \nabla J(\theta(m))\Vert^2 + \langle \nabla J(\theta(m)), e^2(m) \rangle - \Vert \nabla J(\theta(m))\Vert (\Vert e^1(m)\Vert  + \Vert e^3(m)\Vert).
\end{align*}
Plugging the above bounds on $\Vert\hat{g}(m)\Vert^2$ and $\langle \nabla J(\theta(m)),\hat{g}(m)\rangle$ into \eqref{eq:actor:J_recursive}, we have,  
    \begin{align}
        J(\theta(m+1)) \geq J(\theta(m)) + (\eta_m - 2L' \eta_m^2)\Vert\nabla J(\theta(m))\Vert^2 + \eta_m \varepsilon_{m,0} - \eta_m \varepsilon_{m,1} -\eta_{m}^2 \varepsilon_{m,2}, \label{eq:actor:J_recursive_2}
    \end{align}
where $\varepsilon_{m,0}= \langle \nabla J(\theta(m)), e^2(m) \rangle$,
    $\varepsilon_{m,1} = \Vert \nabla J(\theta(m))\Vert (\Vert e^1(m)\Vert  + \Vert e^3(m)\Vert)$, $
    \varepsilon_{m,2} = 2L'( \Vert e^1(m)\Vert^2 + \Vert e^2(m)\Vert^2 + \Vert e^3(m)\Vert^2)$.
%Here $\varepsilon_{m,1}$ captures bias and scales as $O(\rhok)$; $\varepsilon_{m,2}$ captures variance. $\varepsilon_{m,0}$ is a martingale difference sequence w.r.t. everything happened before episode $m$. 
    Doing a telescope sum for \eqref{eq:actor:J_recursive_2}, we get
    \begin{align}
        J(\theta(M)) & \geq J(\theta(0)) + \sum_{m=0}^{M-1} (\eta_m - 2L' \eta_m^2)\Vert\nabla J(\theta(m))\Vert^2 +  \sum_{m=0}^{M-1} \eta_m \varepsilon_{m,0} -  \sum_{m=0}^{M-1}\eta_m \varepsilon_{m,1} - \sum_{m=0}^{M-1}\eta_{m}^2 \varepsilon_{m,2} \nonumber\\
        &\geq  J(\theta(0)) + \sum_{m=0}^{M-1} \frac{1}{2}\eta_m\Vert\nabla J(\theta(m))\Vert^2 +  \sum_{m=0}^{M-1} \eta_m \varepsilon_{m,0} -  \sum_{m=0}^{M-1}\eta_m \varepsilon_{m,1} - \sum_{m=0}^{M-1}\eta_{m}^2 \varepsilon_{m,2},
    \end{align} where we have used $\eta_m - 2L'\eta_m^2 = \eta_m(1 - 2L'\eta_m) \geq \frac{1}{2} \eta_m$, which is true because $\eta_m\leq \eta \leq \frac{1}{4L'}$. 
    After rearranging, we get
    \begin{align}
        \sum_{m=0}^{M-1} \frac{1}{2}\eta_m\Vert\nabla J(\theta(m))\Vert^2  \leq J(\theta(M))  - J(\theta(0)) -\sum_{m=0}^{M-1} \eta_m \varepsilon_{m,0} +  \sum_{m=0}^{M-1}\eta_m \varepsilon_{m,1} + \sum_{m=0}^{M-1}\eta_{m}^2 \varepsilon_{m,2} . \label{appendix:actor:eq:sum_gradient_square}
    \end{align}

We now apply our bounds on $e^1(m),e^2(m),e^3(m)$. By Lemma~\ref{appendix:actor:lem:martingale_bound}, we have with probability $1-\frac{\delta}{2}$,
\begin{align}
        \Big| \sum_{m=0}^{M-1} \eta_m \varepsilon_{m,0} \Big| \leq \frac{2\bar{r}^2 L^2}{(1-\gamma)^4} \sqrt{2\sum_{m=0}^{M-1} \eta_m^2 \log\frac{4}{\delta} } . \label{appendix:actor:eq:epsilon_0}
\end{align}
By Lemma~\ref{appendix:actor:lem:e1_bound} and Lemma~\ref{appendix:actor:lem:error_h_grad}, we have with probability $1-\frac{\delta}{2}$,
\begin{align}
   \sup_{m\leq M-1} \varepsilon_{m,1} &\leq \frac{\bar{r} L}{(1-\gamma)^2} (\sup_{m\leq M-1} \Vert e^1(m)\Vert + \sup_{m\leq M-1} \Vert e^3(m)\Vert ) \leq  \frac{\bar{r} L}{(1-\gamma)^2} (\frac{4c L\rhok}{(1-\gamma)^3}+ 2 \frac{L c }{(1-\gamma)}\rhok )\nonumber  \\
   &\leq \frac{6 L^2 c\bar{r}}{(1-\gamma)^5} \rhok . \label{appendix:actor:eq:epsilon_1}
\end{align}    
By Lemma~\ref{appendix:actor:lem:upperbound}, we have almost surely, $\max(\Vert e^1(m)\Vert, \Vert e^2(m)\Vert, \Vert e^3(m)\Vert) \leq 2\frac{\bar{r} L}{(1-\gamma)^2}$, and hence,
\begin{align}
  \sup_{m\leq M-1}  \varepsilon_{m,2} &= 2L'( \Vert e^1(m)\Vert^2 + \Vert e^2(m)\Vert^2 + \Vert e^3(m)\Vert^2)  \leq \frac{24 \bar{r}^2 L' L^2}{(1-\gamma)^4}. \label{appendix:actor:eq:epsilon_2}
\end{align}
Using a union bound, we have with probability $1-\delta$, all three events \eqref{appendix:actor:eq:epsilon_0}, \eqref{appendix:actor:eq:epsilon_1} and \eqref{appendix:actor:eq:epsilon_2} hold, which when combined with \eqref{appendix:actor:eq:sum_gradient_square} implies 
\begin{align}
        &\frac{\sum_{m=0}^{M-1} \eta_m\Vert\nabla J(\theta(m))\Vert^2}{\sum_{m=0}^{M-1} \eta_m } \nonumber\\
        &\leq \frac{2(J(\theta(M))  - J(\theta(0))) + 2\bigg|\sum_{m=0}^{M-1} \eta_m \varepsilon_{m,0} \bigg| +  2\sup_{m\leq M-1}\varepsilon_{m,2}\sum_{m=0}^{M-1}
        \eta_{m}^2}{\sum_{m=0}^{M-1} \eta_m } +2 \sup_{m\leq M-1} \varepsilon_{m,1}   \nonumber\\
        &\leq \frac{2(J(\theta(M))  - J(\theta(0)))+ \frac{4\bar{r}^2 L^2}{(1-\gamma)^4} \sqrt{2\sum_{m=0}^{M-1} \eta_m^2 \log\frac{4}{\delta} }  + \frac{48 \bar{r}^2 L' L^2}{(1-\gamma)^4}\sum_{m=0}^{M-1}\eta_{m}^2}{\sum_{m=0}^{M-1}\eta_m}   + \frac{12 L^2 c\bar{r}}{(1-\gamma)^5} \rhok .
    \end{align}    
    Since $\eta_m= \frac{\eta}{\sqrt{m+1}}$, we have, 
$\sum_{m=0}^{M-1} \eta_m > 2\eta (\sqrt{M+1}-1) \geq \eta \sqrt{M+1}$
and
$\sum_{m=0}^{M-1} \eta_m^2< \eta^2 (1 + \log(M))<2 \eta^2\log(M)$ (using $M\geq 3$). Further we use the bound $J(\theta(M)) \leq \frac{\bar{r}}{1-\gamma}$ and $J(\theta(0))\geq 0$ almost surely. Combining these results, we get with probability $1-\delta$,  
     \begin{align}
 \frac{\sum_{m=0}^{M-1} \eta_m\Vert\nabla J(\theta(m))\Vert^2}{\sum_{m=0}^{M-1} \eta_m } \leq \frac{\frac{2\bar{r}}{\eta(1-\gamma)}+ \frac{8\bar{r}^2 L^2}{(1-\gamma)^4} \sqrt{\log M \log\frac{4}{\delta} }  + \frac{96 \bar{r}^2 L' L^2}{(1-\gamma)^4}\eta \log M}{ \sqrt{M+1}} + \frac{12 L^2 c\bar{r}}{(1-\gamma)^5} \rhok .\nonumber
 \end{align}    
 This concludes the proof of Theorem~\ref{thm:main}.
\qed
%We show that, given Theorem~\ref{appendix:critic:thm}, the descent direction we use, 
%$$\hat{g}_i(m) = \sum_{t=0}^{T} \gamma^t \frac{1}{n}\sum_{j\in \nik} \hat{Q}_j^T (s_{\njk}(t),a_{\njk}(t)) \nabla_{\theta_i} \log \zeta_i^{\theta_i(m)} (a_i(t)|s_i(t)) , $$
%is a biased estimator of $\nabla_{\theta_i} J(\theta(m))$, with bias decaying exponentially in $\kappa$. We then conduct standard SGD analysis to provide the convergence guarantee. 

\section{Numerical Studies} \label{sec:numerical}
In this section, we first conduct numerical studies in Section~\ref{sec:numerical_synthetic} using a synthetic example where the optimal solution is known in closed form to verify our theoretic results. Then, in Section~\ref{sec:numerical_communication}, we demonstrate our approach using the wireless communication example introduced in Section~\ref{subsec:examples}.

\subsection{Synthetic Experiments} \label{sec:numerical_synthetic}
We first study a synthetic example where the interaction graph is a line of $n$ nodes $\mathcal{N} = \{1,2,\ldots,n\}$ with the left most node labeled as $1$ and the right most as $n$. Each node has a binary local state space and local action space $\mathcal{S}_i = \mathcal{A}_i = \{0,1\}$. %To illustrate the effect of exponential decay and the interaction of agents, we consider the following transition and rewards. 
The left most node (node $1$) has a reward $1$ whenever $s_1=1$, and all other reward values of node $1$ and the rewards of all other nodes are $0$. Further, $s_1(t+1) = 1$ with probability $1$ when the second node has state $s_2(t) = 1$, and in all other cases $s_1(t+1) = 0$ with probability $1$. For the $i$'th node with $2\leq i \leq n -1$,
\[P(s_i(t+1) = 1|s_{i-1}(t),s_i(t),s_{i+1}(t),a_i(t) ) = \left\{\begin{array}{ll}
    1, & \text{ if }s_{i+1}(t) = 1, a_i(t) = 1 ,  \\
    0.8, & \text{ if } s_{i+1}(t) = 0, a_i(t) = 1, \\
    0, & \text{ all other cases}.
\end{array} \right. \]
For the last node $i=n$, $s_n(t+1) = 1$ w.p. $1$ when $a_n(t) = 1$, and $s_n(t+1) = 0$ w.p. $1$ when $a_n(t) = 0.$ The initial states of all nodes are $s_i(0) = 1$. 

In this example, the rewards and transitions are designed in a way so that the optimal policy can be determined explicitly. Only agent $1$ has a non-zero reward when its state is $s_1 = 1$, and it will stay in state $s_1 = 1$ only when $s_2 = 1$, which happens with high probability when $a_2 = 1$ and $s_3 = 1$, and so on so forth. Therefore, it is clear that the optimal policy is for all nodes to always take action $a_i(t) = 1$ regardless of the local state, in which case the states of all agents will stay as $s_i(t) = 1$, and the resulting optimal global discounted reward is $\frac{1}{n}\frac{1}{1-\gamma}$. 

In our experiments, we set $n = 8$, $\gamma = 0.7$. We use the softmax policy for the localized policies, which is a standard policy parameterization that encompasses all stochastic policies from $\mathcal{S}_i$ to $\mathcal{A}_i$ \citep{sutton1998introduction}. We run the SAC algorithm with $\kappa = 0$ up to $\kappa = 7$. We plot the global discounted reward throughout the training process in Figure~\ref{fig:synthetic_training}, and we also plot the optimality gap in Figure~\ref{fig:synthetic_exp_decay} computed as the difference of the optimal discounted reward ($\frac{1}{n}\frac{1}{1-\gamma}$) and the discounted reward achieved by the algorithm with the respective $\khop$ values. Figure~\ref{fig:synthetic_training} shows that when $\kappa$ increases, the performance of the algorithm also increases, though the improvement becomes small when $\kappa>1$. This is also confirmed by Figure~\ref{fig:synthetic_exp_decay}, which shows the optimality gap is decaying exponentially when $\kappa$ increases. This is consistent with our theoretic result in Theorem~\ref{thm:main}. We note that the exponential decay appears to stop in Figure~\ref{fig:synthetic_exp_decay} when $\khop>4$, and this is due to the increasing complexity in training for large $\khop$ as discussed in Section~\ref{subsec:convergence}, and may also be due to the fact that any further potential improvement may be too small (in the order of $1\times 10^{-3}$) to be noticed because of the noise. 
\begin{figure}
  \begin{minipage}[t]{0.5\textwidth}
    \centering
    \includegraphics[width = \textwidth]{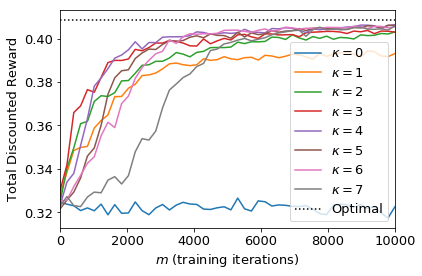}
    \caption{Global discounted reward during the training process in the synthetic experiments.}
    \label{fig:synthetic_training}
  \end{minipage}~
  \begin{minipage}[t]{0.5\textwidth}
    \centering
    \includegraphics[width = \textwidth]{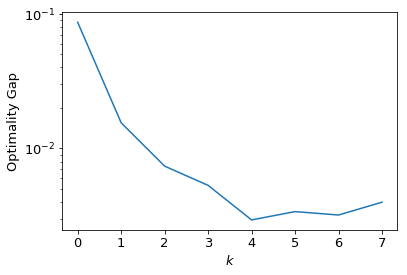}
    \caption{Optimality gap as function of $\kappa$ in the synthetic experiments. }
    \label{fig:synthetic_exp_decay}
  \end{minipage}
\end{figure}

\subsection{Multi-Access Wireless Communication}\label{sec:numerical_communication}
We next study the multi-access wireless communication example discussed in Section~\ref{subsec:examples}. We consider a grid of users in Figure~\ref{fig:communication_grid}, where each user has access points on the corners of the block it is in. 
In the experiments, we set the grid size as $6\times 6$, deadline as $d_i = 2$, and all parameters $p_i$ (packet arrival probability for user $i$) and $q_k$ (success transmission probability for access point $y_k$) are generated uniformly random from $[0,1]$. We set $\gamma = 0.7$ and the initial state to be uniformly random, and run the SAC algorithm with $\khop=0,1,2$ to learn a localized softmax policy, starting from an initial policy where the action is chosen uniformly random. 
We compare the proposed method with a benchmark based on the localized ALOHA protocol \citep{aloha}, where each user has a certain probability of sending the earliest packet and otherwise not sending at all. 
When it sends, it sends the packet to a random access point in its available set, with probability proportional to the success transmission probability of this access point and inverse proportional to the number of users that share this access point.
The results are shown in Figure~\ref{fig:comm6by6k02}. It shows that the proposed algorithm can outperform the ALOHA based benchmark, despite the proposed algorithm does not have access to the transmission probability $q_k$ which the benchmark has access to. It also shows that the SAC with $\khop = 1,2$ outperforms $\khop =0$, which as we mentioned in Section~\ref{subsec:algo} corresponds to the independent learner approach in the literature \citep{tan1993multi,lowe2017multi}. SAC with $\khop = 2$ outperforms $\khop = 1$, but the improvement is small, which is consistent with the results in the synthetic experiments in Section~\ref{sec:numerical_synthetic}. We also study a case with the same $6\times 6$ grid of access points, but the $36$ users are assigned randomly to the square blocks in the grid. We plot the results in Figure~\ref{fig:communication_grid_random}. Similar phenomenons can be observed in Figure~\ref{fig:communication_grid_random} as in Figure~\ref{fig:comm6by6k02}, with SAC outperforms the benchmark and SAC with $\kappa=1,2$ outperforms $\kappa = 0$, the independent learner approach.  

\begin{figure}[t]
    \centering
    \includegraphics[width = .6\textwidth]{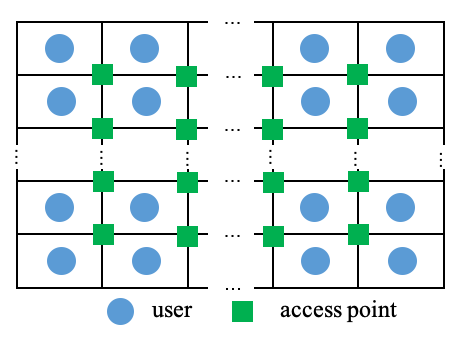}
    \caption{Grid of users and access points.}
    \label{fig:communication_grid}
\end{figure}

\begin{figure}[t]
  \begin{minipage}[t]{0.5\textwidth}
    \centering
    \includegraphics[width = \textwidth]{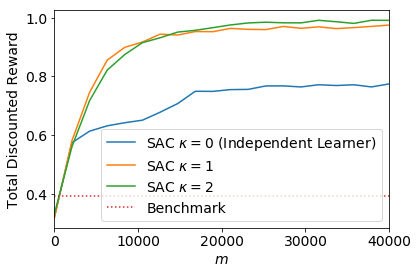}
    \caption{Global discounted reward during training. }
    \label{fig:comm6by6k02}
  \end{minipage}~
  \begin{minipage}[t]{0.5\textwidth}
    \centering
    \includegraphics[width = \textwidth]{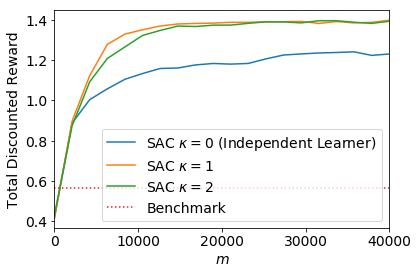}
    \caption{Global discounted reward during training when user locations are assigned randomly.}
    \label{fig:communication_grid_random}
  \end{minipage}~
\end{figure}

\section{Concluding Remarks and Extensions}
This paper proposes a \algoname\ algorithm that provably finds a close-to-stationary point of $J(\theta)$ in time that scales with the local state-action space size of the largest $\khop$-hop neighborhood, which can be much smaller than the full state-action space size when the graph is sparse. This represents the first scalable RL method for localized control of multi-agent networked systems with such a provable guarantee. There are many possible extensions and open questions, which we discuss below.  

\textbf{Average Reward.} One extension is to consider average reward instead of discounted reward, i.e., to consider 
$$ J_{\mathrm{ave}}(\theta) = \lim_{T\rightarrow\infty} \frac{1}{T}\E_{s\sim \pi_0} \E_{a(t)\sim \zeta^\theta(\cdot|s(t))} \Big[\sum_{t=0}^{T-1} r(s(t),a(t)) | s(0) = s \Big]. $$
Under appropriate ergodicity assumptions, the average reward above is equivalent to the reward under the stationary distribution. Average reward is common in applications where the performance is measured in stationarity, e.g. throughput or waiting time in communication and queueing networks. Despite the importance in applications, average reward RL is known to be more challenging even in single-agent settings, see e.g. \cite{tsitsiklis1999average,tsitsiklis2002average}. For example, the $Q$ function needs to be defined in a different way, 
$$Q^\theta(s,a) = \E_{a(t)\sim \zeta^\theta(\cdot|s(t))} \Big[\sum_{t=0}^\infty ( r(s(t),a(t)) -  J_{\mathrm{ave}} (\theta))| s(0) = s , a(0) = a\Big] , $$
and because of the lack of a discounting factor $\gamma$, the associated Bellman operator no longer has $\gamma$ as the natural contraction factor. 
In ongoing work summarized in \cite{qu2020scalable}, we have begun to study an average reward multi-agent RL setting with the same local interaction structure \eqref{eq:transition_factor} as in this paper. While similar exponential decay properties on the $Q$-functions can be defined, in the average reward setting, \cite{qu2020scalable} shows that the exponential decay only holds under a form of bounded interaction strength assumption.  Under this assumption, \cite{qu2020scalable} proposes a variant of the SAC algorithm that achieves similar guarantees as the algorithm in this paper. However, the bounded interaction strength assumption in \cite{qu2020scalable} may be restrictive, and searching for weaker assumptions that guarantee the exponential decay (or weaker forms of decay) is an interesting open future direction. 

\textbf{Time-varying dependence structure. } In this paper, the interaction graph that defines the dependence structure \eqref{eq:transition_factor} is a fixed graph, while in many real world applications the graph is time-varying. In other words, if the interaction graph at time $t$ is $\mathcal{G}_t$, then \eqref{eq:transition_factor} becomes, 
\[P(s(t+1)|s(t),a(t)) = \prod_{i=1}^n P_i(s_i(t+1)| s_{N_i(\mathcal{G}_t)}(t),a_i(t)), \]
where $N_i(\mathcal{G}_t)$ means the set of neighbors of $i$ in graph $\mathcal{G}_t$. 
In another of our extensions of this work \citep{lin2020multiagent}, we study a stochastic graph setting where, at each time step, the graph $\mathcal{G}_t$ is sampled from a fixed graph distribution in which each link is present with probability exponentially decreasing in a predefined distance measure between the two nodes. \citet{lin2020multiagent} shows that a weaker form of exponential decay holds in this setting, termed $\mu$-decay, and a similar SAC algorithm can be used. Beyond \citet{lin2020multiagent}, an open question is how to handle the case where $\mathcal{G}_t$ can be arbitrarily chosen, rather than stochastically sampled.  An issue in this setting is that the transition kernel may not be time-homogeneous anymore, which means that is a challenging open problem.

\textbf{Other future directions.} While \algoname\ is an actor critic based algorithm, the idea underlying \algoname, including the exponential decay property and the truncated $Q$-function \eqref{eq:truncated_q}, is a contribution in its own right, and can potentially lead to other scalable RL methods for networked systems. For example, within the actor critic framework, one idea is to change the critic to variants like TD-$\lambda$, change the actor to include the advantage function, or use the simultaneous update version of the actor critic algorithm (as opposed to the inner-loop structure used in this paper). Beyond the actor critic framework, one can develop policy-iteration type algorithms (e.g. $Q$-learning/SARSA type methods) on our truncated $Q$-functions. 
Further, our proof technique in showing the finite time convergence of TD learning can be of independent interest.
In fact, we have already considered preliminary applications of our technique in the context of Q-learning and TD learning with state aggregation \citep{lin2020multiagent}, and we expect other applications to emerge in the coming years.
Additionally, the setting we consider here does not provide an investigation of the tradeoff between exploration and exploitation.  Adding consideration of this tradeoff is an interesting direction for future work. Finally, other future directions include studying the landscape of our policy optimization problem, and studying the robustness of the trained policies.

% Appendix here
% Options are (1) APPENDIX (with or without general title) or
%             (2) APPENDICES (if it has more than one unrelated sections)
% Outcomment the appropriate case if necessary
%
% \begin{APPENDIX}{<Title of the Appendix>}
% \end{APPENDIX}
%
%   or
%
\appendix
\input{appendix}

% Acknowledgments here
 \section*{Acknowledgement}
 We would like to thank Yiheng Lin of Caltech and Prof. Longbo Huang of Tsinghua University for suggesting the wireless communication example. This work was supported by Resnick Sustainability Institute Fellowship, NSF CAREER 1553407, ONR YIP N00014-19-1-2217, AFOSR YIP FA9550-18-1-0150, the PIMCO Fellowship, NSF AitF-1637598, NSF CNS-1518941, Amazon AI4Science Fellowship, and Caltech Center for Autonomous Systems and Technologies (CAST). 
 %\lina{? is acknowledge from somewhere else?} \guannan{This is from the template... Not sure if we need one. The thing I can think of is to acknowledge Yiheng and Longbo for suggesting the application example.} \adam{Yes, acknowledge Longbo and Yiheng}

% References here (outcomment the appropriate case)

% CASE 1: BiBTeX used to constantly update the references
%   (while the paper is being written).
\bibliographystyle{abbrvnat} % outcomment this and next line in Case 1
\bibliography{networkRLref} % if more than one, comma separated

% CASE 2: BiBTeX used to generate mypaper.bbl (to be further fine tuned)
%\input{mypaper.bbl} % outcomment this line in Case 2

%If you don't use BiBTex, you can manually itemize references as shown below.

%%%%%%%%%%%%%%%%%
\end{document}

%% file: appendix.tex
\section{The Exponential Decay Property}\label{subsec:appendix:exponential_decaying}

%\adam{Guannan: You don't seem to use punctuation in your equations -- we need to add that in.  I did it for most of the body, but probably missed some.  I have not done it for the appendix.}

%\adam{Can we use exponential decay instead of ''exponential decaying''...it reads a little better.}
%\guannan{I have changed everything to decay.}

Our main results depend on the $(c,\rho)$-exponential decay of the $Q$-function (cf. Definition~\ref{def:exp_decaying}), which means that for any $i$, any $s_{\nik}$, $s_{\nminusik}$ and $s_{\nminusik}'$, $a_{\nik}$, $a_{\nminusik}$ and $a_{\nminusik}'$, 
$$|Q_i^\theta(s_{\nik},s_{\nminusik}, a_{\nik}, a_{\nminusik}) - Q_i^\theta(s_{\nik},s_{\nminusik}',a_{\nik}, a_{\nminusik}')| \leq c\rhok .$$

In Section~\ref{subsec:key_idea}, we have pointed out in Lemma~\ref{lem:exp_decaying} that the $(c,\rho)$-exponential decay property always holds with $\rho\leq \gamma$, assuming the rewards $r_i$ are upper bounded. We now provide the proof of Lemma~\ref{lem:exp_decaying}.

\begin{proof}[Proof of Lemma~\ref{lem:exp_decaying}] We first prove part (a). For notational simplicity, denote $s = (s_{\nik},s_{\nminusik})$, $a = (a_{\nik}, a_{\nminusik})$; $s'= (s_{\nik},s_{\nminusik}')$ and $a' = (a_{\nik}, a_{\nminusik}')$. Let $\pi_{t,i}$ be the distribution of $(s_i(t),a_i(t))$ conditioned on $(s(0),a(0)) = (s,a)$ under policy $\theta$, and let $\pi_{t,i}'$ be the distribution of $(s_i(t),a_i(t))$ conditioned on $(s(0),a(0))=(s',a')$ under policy $\theta$. Then, we must have $\pi_{t,i} = \pi_{t,i}'$ for all $t\leq \khop$. The reason is that, due to the local dependence structure \eqref{eq:transition_factor} and the localized policy structure, $\pi_{t,i}$ only depends on $(s_{N_i^t},a_{N_i^t})$ (the initial state-action of agent $i$'th $t$-hop neighborhood) which is the same as $(s_{N_i^t}',s_{N_i^t}')$ when $t\leq \khop$ per the way the initial state $(s,a)$, $(s',a')$ are chosen. With these definitions, we expand the definition of $Q_i^\theta$ in \eqref{eq:full_q},
\begin{align}
   &|Q_i^\theta(s,a) - Q_i^\theta(s',a')| \nonumber \\ &\leq   \sum_{t=0}^\infty \bigg| \E \big[\gamma^t r_i(s_i(t),a_i(t))  \big|(s(0),a(0)) = (s,a) \big] - \E \big[\gamma^t r_i(s_i(t),a_i(t))  \big|(s(0),a(0)) = (s',a') \big] \bigg| \nonumber \\
  & = \sum_{t=0}^\infty \bigg| \gamma^t \E_{(s_i,a_i)\sim \pi_{t,i}} r_i(s_i,a_i)  -\gamma^t \E_{(s_i,a_i)\sim \pi_{t,i}'} r_i(s_i,a_i) \bigg|\nonumber \\
  &= \sum_{t=\khop+1}^\infty \bigg| \gamma^t \E_{(s_i,a_i)\sim \pi_{t,i}} r_i(s_i,a_i)  -\gamma^t \E_{(s_i,a_i)\sim \pi_{t,i}'} r_i(s_i,a_i) \bigg|\nonumber \\
  &\leq \sum_{t=\khop+1}^\infty  \gamma^t   \bar{r} \text{TV}( \pi_{t,i},\pi_{t,i}')   \leq  \frac{\bar{r}}{1-\gamma}\gamma^{\khop+1}, \label{appendix:exp_decay:eq:tv}
\end{align}
where $ \text{TV}( \pi_{t,i},\pi_{t,i}') $ is the total variation distance between $\pi_{t,i}$ and $\pi_{t,i}'$ which is upper bounded by $1$. The above inequality shows that the $(\frac{\bar{r}}{1-\gamma},\gamma)$-exponential decay property holds and concludes the proof of Lemma~\ref{lem:exp_decaying} (a). 

The proof of part (b) is almost identical to that of part(a). The only change is that in step \eqref{appendix:exp_decay:eq:tv}, we use $\textup{TV}(\pi_{t,i},\pi'_{t,i})\leq 2c'\mu^t$.
\end{proof}

\section{Proof of Lemma~\ref{lem:truncated_pg}}\label{sec:appendix_truncated_pg}
%In this section, we provide the proof of Lemma~\ref{lem:truncated_pg}. 

%\noindent\textbf{Proof of Lemma~\ref{lem:truncated_pg}} 
We first show part (a), the truncated $Q$ function is a good approximation of the true $Q$ function. To see that, we have for any $(s,a)\in\mathcal{S}\times \mathcal{A}$, by \eqref{eq:truncated_q} and \eqref{eq:truncated_q_weights}, 
\begin{align}
&|\hat{Q}_i^\theta (s_{\nik}, a_{\nik}) - Q_i^\theta(s,a)| \nonumber  \\
&= \Big| \sum_{s_{\nminusik}', a_{\nminusik}'} w_i(s_{\nminusik}', a_{\nminusik}';s_{\nik}, a_{\nik} ) Q_i^\theta (s_{\nik},s_{\nminusik}', a_{\nik},a_{\nminusik}') - Q_i^\theta(s_{\nik},s_{\nminusik},a_{\nik},a_{\nminusik}) \Big|\nonumber \\
&\leq \sum_{s_{\nminusik}', a_{\nminusik}'} w_i(s_{\nminusik}', a_{\nminusik}';s_{\nik}, a_{\nik} ) \Big| Q_i^\theta (s_{\nik},s_{\nminusik}', a_{\nik},a_{\nminusik}') - Q_i^\theta(s_{\nik},s_{\nminusik},a_{\nik},a_{\nminusik}) \Big|\nonumber\\
&\leq c\rhok,\label{appendix:truncated:eq:q_err}
\end{align}
where in the last step, we have used the $(c,\rho)$-exponential decay property, cf. Definition~\ref{def:exp_decaying}. 

Next, we show part (b). Recall by the policy gradient theorem (Lemma~\ref{lem:policy_grad}),
\begin{align*}
    \nabla_{\theta_i} J(\theta) &= \frac{1}{1-\gamma} \E_{s\sim \pi^\theta, a\sim\zeta^\theta(\cdot|s)}\left[ Q^{\theta}(s,a) \nabla_{\theta_i} \log \zeta^\theta (a|s)\right]\\
    &= \frac{1}{1-\gamma} \E_{s\sim \pi^\theta, a\sim\zeta^\theta(\cdot|s)} \left[ Q^{\theta}(s,a) \nabla_{\theta_i} \log \zeta_i^{\theta_i} (a_i|s_i) \right],
\end{align*}
where we have used $\nabla_{\theta_i}\log\zeta^{\theta}(a|s) =\nabla_{\theta_i} \sum_{j\in\mathcal{N}} \log\zeta_j^{\theta_j}(a_j|s_j)=\nabla_{\theta_i}\log\zeta_i^{\theta_i}(a_i|s_i)$ by the localized policy structure. With the above equation, we can compute $\hat{h}_i(\theta) - \nabla_{\theta_i} J(\theta)$,
\begin{align*}
   & \hat{h}_i(\theta) - \nabla_{\theta_i} J(\theta) \\
    & = \frac{1}{1-\gamma} \E_{s\sim \pi^\theta, a\sim\zeta^\theta(\cdot|s)} \left[ \left(\frac{1}{n}\sum_{j\in {\nik}} \hat{Q}_j^{\theta} (s_{\njk},a_{\njk}) -  Q^{\theta}(s,a)\right) \nabla_{\theta_i}  \log \zeta_i^{\theta_i}(a_i|s_i) \right]\\
    &= \frac{1}{1-\gamma} \E_{s\sim \pi^\theta, a\sim\zeta^\theta(\cdot|s)} \left[\left(\frac{1}{n}\sum_{j\in \mathcal{N}} \hat{Q}_j^{\theta} (s_{\njk},a_{\njk}) - \frac{1}{n}\sum_{j\in \mathcal{N}} Q^{\theta}_j(s,a)\right) \nabla_{\theta_i}  \log \zeta_i^{\theta_i} (a_i|s_i) \right]\\
    &\qquad - \frac{1}{1-\gamma} \E_{s\sim \pi^\theta, a\sim\zeta^\theta(\cdot|s)} \left[ \frac{1}{n}\sum_{j\in \nminusik} \hat{Q}_j^{\theta} (s_{\njk},a_{\njk}) \nabla_{\theta_i}  \log \zeta_i^{\theta_i} (a_i|s_i)\right]\\
    &:= E_1 - E_2.
\end{align*}
We claim that $E_2 = 0$. To see this, consider for any $j\in \nminusik$, 
    \begin{align*}
        &\mathbb{E}_{s\sim\pi^\theta, a\sim \zeta^{\theta}(\cdot|s)}  \left[\nabla_{\theta_i}  \log \zeta_i^{\theta_i}(a_i|s_i)    \hat{Q}_j^{\theta}(s_{\njk},a_{\njk})\right]\\
        &= \sum_{s,a} \pi^\theta(s) \prod_{\ell=1}^n \zeta^{\theta_\ell}_\ell(a_\ell|s_\ell) \frac{\nabla_{\theta_i}  \zeta_i^{\theta_i}(a_i|s_i)}{\zeta_i^{\theta_i}(a_i|s_i)} \hat{Q}_j^{\theta}(s_{\njk},a_{\njk})\\
        &= \sum_{s,a} \pi^\theta(s) \prod_{\ell\neq i} \zeta^{\theta_\ell}_\ell(a_\ell|s_\ell) \nabla_{\theta_i}  \zeta_i^{\theta_i}(a_i|s_i)  \hat{Q}_j^{\theta}(s_{\njk},a_{\njk})\\
        &= \sum_{s,a_{1},\ldots,a_{i-1},a_{i+1},\ldots,a_n } \pi^\theta(s) \prod_{\ell\neq i} \zeta^{\theta_\ell}_\ell(a_\ell|s_\ell)  \hat{Q}_j^{\theta}(s_{\njk},a_{\njk}) \sum_{a_i} \nabla_{\theta_i}   \zeta_i^{\theta_i}(a_i|s_i) \\
        & = 0,
    \end{align*}
    where in the last equality, we have used $\hat{Q}_j^{\theta}(s_{\njk},a_{\njk})$ does not depend on $a_i$ as $i\not\in \njk$; and $\sum_{a_i} \nabla_{\theta_i}   \zeta_i^{\theta_i}(a_i|s_i) =  \nabla_{\theta_i}   \sum_{a_i} \zeta_i^{\theta_i}(a_i|s_i) =\nabla_{\theta_i} 1 = 0 $. Now that we have shown $E_2=0$, we can bound $E_1$ as follows
    \begin{align*}
       & \Vert\hat{h}_i(\theta) - \nabla_{\theta_i} J(\theta)\Vert = \Vert E_1\Vert \\
       & \leq \frac{1}{1-\gamma} \E_{s\sim \pi^\theta, a\sim\zeta^\theta(\cdot|s)} \left[\frac{1}{n} \sum_{j\in \mathcal{N}} \Big|\hat{Q}_j^{\theta} (s_{\njk},a_{\njk}) - Q^{\theta}_j(s,a)\Big| \Vert \nabla_{\theta_i}  \log \zeta_i^{\theta_i} (a_i|s_i)\Vert\right]\\
       &\leq \frac{1}{1-\gamma} c\rhok L_i,
    \end{align*}
where in the last step, we have used \eqref{appendix:truncated:eq:q_err} and the upper bound $\Vert \nabla_{\theta_i}  \log \zeta_i^{\theta_i} (a_i|s_i)\Vert \leq L_i$. This concludes the proof of Lemma~\ref{lem:truncated_pg}.
\qed
\section{Proof of Auxiliary Results for the Analysis of the Critic}\label{sec:appendix_critic}

\subsection{Proof of Lemma~\ref{lem:critic:barab}}\label{subsec:proof_barab}
The goal of Lemma~\ref{lem:critic:barab} is to understand $\bar{A}_{t-1} \hat{Q}_i^{t-1} + \bar{b}_{t-1}$. Recall that $d_{t-1}$ is the distribution of $z(t-1)$ conditioned on $\mathcal{F}_{t-\tau}$ and $
    \bar{A}_{t-1} = \E [A_{t-1} | \mathcal{F}_{t-\tau} ]= \E [A_{z(t-1),z(t)} | \mathcal{F}_{t-\tau}]$, $ \bar{b}_{t-1} = \E [b_{t-1} | \mathcal{F}_{t-\tau} ]= \E [b_{z(t-1)} | \mathcal{F}_{t-\tau}]$. 

Recall that $P$ is the transition matrix from $z(t-1)$ to $z(t)$. Given any distribution $d$ on the state-action space $\mathcal{Z}$, we define $\tilde{A}_z = \E_{z'\sim P(\cdot|z) }A_{z,z'}$ and 
$\bar{A}^d = \E_{z\sim d} \tilde{A}_z$, $\bar{b}^d    =\E_{z\sim d} b_z$. Then, $\bar{A}_{t-1}$ and $\bar{b}_{t-1}$ can be rewritten as $\bar{A}_{t-1} = \bar{A}^{d_{t-1}}$, $\bar{b}_{t-1} = \bar{b}^{d_{t-1}}$. In what follows, we provide characterizations of $\bar{A}^d$, $\bar{b}^d$ for general distributions $d$. 
%The transition matrix $P$ will depend on the current policy, but in this section we focus entirely on the critic update within the inner loop, so we drop the dependence on policy. 

Firstly, $\tilde{A}_z$ is given by,
\begin{align}
    \tilde{A}_z = \E_{z'\sim P(\cdot|z) }A_{z,z'} = \mathbf{e}_{z_{\nik}} [ \gamma   P(\cdot|z)\Phi - \mathbf{e}_{z_{\nik}}^T ] ,
\end{align}
where $P(\cdot|z)$ is understood as the $z$'th row of $P$ and is treated as a row vector. Also, we have defined $\Phi\in \R^{\mathcal{Z}\times \mathcal{Z}_{\nik}}$ to be a matrix with each row indexed by $z\in \mathcal{Z}$ and each column indexed by $z_{\nik}' \in   \mathcal{Z}_{\nik}$. Further, the $z$'th row of $\Phi$ is the indicator vector $\mathbf{e}_{z_{\nik}}^\top$, in other words $\Phi(z,z_{\nik}')=1$ if $z_{\nik}' =z_{\nik}  $ and  $\Phi(z,z_{\nik}')=0$ elsewhere. 

Then, $\bar{A}^d$, $\bar{b}^d$ are given by,
\begin{align}
    \bar{A}^d &= \E_{z\sim d} \tilde{A}_z   = \sum_{z\in\mathcal{Z}}d(z) \mathbf{e}_{z_{\nik}} [ \gamma   P(\cdot|z)\Phi - \mathbf{e}_{z_{\nik}}^\top ]  = \Phi^\top \diag(d) \big[\gamma P\Phi - \Phi \big], \label{appendix:critic:eq:bar_a} \\
 \bar{b}^d    &=\E_{z\sim d} b_z = \Phi^\top\diag(d) r_i, \label{appendix:critic:eq:bar_b}
\end{align}
where $\diag(d)\in\R^{\mathcal{Z}\times\mathcal{Z}}$ is a diagonal matrix with the $z$'th diagonal entry being $d(z)$; in the last equation, $r_i$ is understood as a vector over the entire state-action space $\mathcal{Z}$, though it only depends on $z_i$.

With the above characterizations, we show the following property of $\bar{A}^d$ and $ \bar{b}^d$ in Lemma~\ref{appendix:critic:lem:ab}. Lemma~\ref{appendix:critic:lem:ab} (with $d$ set as $d_{t-1}$) will directly lead to the results in Lemma \ref{lem:critic:barab}, with $D_{t-1}$ being the $D$ in Lemma~\ref{appendix:critic:lem:ab} and it satisfies $D_{t-1} \succeq \sigma I$ due to Assumption~\ref{assump:local_explor}; $g_{t-1}$ and $\hat{Q}_i^{*,t-1}$ being $g^{d_{t-1}}$ and $\hat{Q}_i^{d_{t-1}}$ in Lemma~\ref{appendix:critic:lem:ab}.  

%The proof is postponed to Section~\ref{appendix:critic:subsec:lemma_ab}. %We next look at a few properties of $\bar{A}^d$ and $\bar{b}^d$ for a given distribution $d$ on the state-action pair $z$.
%Assuming the pair $(s(t-1),a(t-1))$ follows distribution $d$, and denote the $P$ be the transition from $(s(t-1),a(t-1))$ to $(s(t),a(t))$. Then we have
\begin{lemma}\label{appendix:critic:lem:ab}
Given distribution $d$ on state-action pair $z$ whose marginalization onto $z_{\nik}$  is non-zero for every $z_{\nik}$, we have
%    \item [(a)] $\bar{A}^d \hat{Q}_i + \bar{b}^d = \Phi^\top\diag(d) \td(\Phi\hat{Q}_i) - \Phi^\top\diag(d)\Phi\hat{Q}_i $ where $\td:\mathbb{R}^{\mathcal{Z}} \rightarrow \mathbb{R}^{\mathcal{Z}}$, $\td(Q_i)= r_i + \gamma P Q_i$ is the Bellman operator.
%    \item [(b)] Let $\Pi^d = (\Phi^\top\diag(d)\Phi)^{-1}\Phi^\top\diag(d)$. It is well defined and non-expansive in infinity norm.
 $\bar{A}^d \hat{Q}_i + \bar{b}^d $ can be written as
    $$\bar{A}^d \hat{Q}_i + \bar{b}^d = -D \hat{Q}_i + D g^d( \hat{Q}_i),$$
    where $D = \Phi^\top \diag(d)\Phi \in \mathbb{R}^{\mathcal{Z}_{\nik}\times \mathcal{Z}_{\nik}}$ is a diagonal matrix, with the $z_{\nik}$'th entry being the marginalized distribution of $z_{\nik}$ under distribution $d$; $g^d(\cdot)$ is given by $g^d( \hat{Q}_i)=\Pi^d \td\Phi \hat{Q}_i$, where $\Pi^d = (\Phi^\top\diag(d)\Phi)^{-1}\Phi^\top\diag(d)$ and $\td(Q_i)= r_i + \gamma P Q_i$ is the Bellman operator in \eqref{appendix:critic:eq:bellman}.
    
    Further, $g^d(\cdot)$ is $\gamma$ contractive in infinity norm, and has a unique fixed point $\hat{Q}_i^d\in\R^{\mathcal{Z}_{\nik}}$ depending on $d$, and the fixed point satisfies 
    \begin{align}
        \sup_{z\in\mathcal{Z}} |\hat{Q}_i^{d} (z_{\nik}) - Q_i^*(z)| =  \Vert \Phi \hat{Q}_i^{d} - Q_i^*\Vert_\infty \leq \frac{ c \rhok}{1-\gamma}. \label{appendix:critic:eq:fixed_point_error}
    \end{align}
\end{lemma}
\begin{proof}[Proof of Lemma~\ref{appendix:critic:lem:ab}]
It is easy to check that $D = \Phi^\top\diag(d)\Phi \in \mathbb{R}^{\mathcal{Z}_{\nik}\times \mathcal{Z}_{\nik}}$ is a diagonal matrix, and the $z_{\nik}$'th diagonal entry is the marginal probability of $z_{\nik}$ under $d$, which is non-zero by the assumption of the lemma. Therefore, $\Phi^\top\diag(d)\Phi \in \mathbb{R}^{\mathcal{Z}_{\nik}\times \mathcal{Z}_{\nik}}$ is invertable and matrix $\Pi^d = (\Phi^\top\diag(d)\Phi)^{-1}\Phi^\top\diag(d)$ is well defined. Further, the $z_{\nik}$'th row of $\Pi^d$ is in fact the conditional distribution of the full state $z$ given $z_{\nik}$. So, $\Pi^d$ must be a stochastic matrix and is non-expansive in infinity norm.

By the definition of $\bar{A}^d$ and $\bar{b}^d$, we have,
    \begin{align*}
        \bar{A}^d \hat{Q}_i + \bar{b}^d &= \Phi^\top \diag(d) \big[\gamma P\Phi - \Phi \big] \hat{Q}_i+ \Phi^\top\diag(d) r_i \\
        &=\Phi^\top \diag(d) [r_i+ \gamma P\Phi \hat{Q}_i] - \Phi^\top\diag(d)\Phi\hat{Q}_i\\
        &= \Phi^\top \diag(d) \td(\Phi\hat{Q}_i) -  \Phi^\top\diag(d)\Phi\hat{Q}_i\\
&= -D\hat{Q}_i + D \Pi^d\td(\Phi\hat{Q}_i)\\
&=  -D\hat{Q}_i +D  g^d(\hat{Q}_i),
        \end{align*}
where $\td$ is the Bellman operator for reward $r_i$ defined in \eqref{appendix:critic:eq:bellman}, and operator $g^d$ is given by $g^d( \hat{Q}_i)=\Pi^d \td\Phi \hat{Q}_i$.

Notice that $\Phi$ is non-expansive in $\Vert\cdot\Vert_\infty$ norm since each row of $\Phi$ has precisely one entry being $1$ and all others are zero. Also since $\Pi^d$ is non-expansive in $\Vert\cdot\Vert_\infty$ norm and $\td$ is a $\gamma$-contraction in $\Vert\cdot\Vert_\infty$ norm, we have $g^d = \Pi^d \td \Phi$ is a $\gamma$ contraction in $\Vert\cdot\Vert_\infty$ norm. As a result, $g^d$ has a unique fixed point $\hat{Q}_i^d$.

Finally, we show \eqref{appendix:critic:eq:fixed_point_error}, which bounds the distance between $\Phi\hat{Q}_i^d$ and $Q_i^*$, where $Q_i^*$ is the true $Q$-function for reward $r_i$ and it is the unique fixed point of $\td$ operator \eqref{appendix:critic:eq:bellman}. We have,
\begin{align*}
\Vert \Phi \hat{Q}_i^{d} - Q_i^*\Vert_\infty&\leq \Vert \Phi \hat{Q}_i^{d} - \Phi \Pi^d Q_i^*\Vert_\infty + \Vert \Phi \Pi^d Q_i^* - Q_i^*\Vert_\infty\\
&=\Vert\Phi \Pi^d\td(\Phi \hat{Q}_i^d) - \Phi\Pi^d\td(Q_i^*)\Vert_\infty + \Vert \Phi \Pi^d Q_i^* - Q_i^*\Vert_\infty\\
&\leq \gamma \Vert \Phi \hat{Q}_i^d - Q_i^*\Vert_\infty + \Vert \Phi \Pi^d Q_i^* - Q_i^*\Vert_\infty,
\end{align*}
where the equality follows from the fact that $\hat{Q}_i^d$ is the fixed point of $\Pi^d \td\Phi$, $Q_i^*$ is the fixed point of $\td$; the last inequality is due to $\Phi\Pi^d \td$ is a $\gamma$ contration in infinity norm. Therefore,
\begin{align}
    \Vert \Phi \hat{Q}_i^{d} - Q_i^*\Vert_\infty \leq \frac{1}{1-\gamma} \Vert \Phi \Pi^d Q_i^* - Q_i^*\Vert_\infty. \label{eq:lem_d:fixed_point_error_1}
\end{align}
Next, recall that the $z_{\nik}$'s row of $\Pi^d$ is the distribution of the state-action pair $z$ conditioned on its $\nik$ coordinates being fixed to be $z_{\nik}$. We denote this conditional distribution of the states outside of $\nik$, $z_{\nminusik}$, given $z_{\nik}$, as $d(z_{\nminusik}|z_{\nik})$. With this notation,
$$(\Pi^d Q^*_i)(z_{\nik}) = \sum_{z_{\nminusik}}d(z_{\nminusik}|z_{N_{i}^k} ) Q^*_i(z_{\nik},z_{\nminusik}) .$$
And therefore, 
\begin{align*}
    (\Phi\Pi^d Q_i^*)(z_{\nik}, z_{\nminusik}) =\sum_{z_{\nminusik}'}d(z_{\nminusik}'|z_{\nik} ) Q_i^*(z_{\nik},z_{\nminusik}').
\end{align*}
Further, we have
\begin{align*}
  & |  (\Phi\Pi^d Q_i^*)(z_{\nik}, z_{\nminusik}) - Q_i^* (z_{\nik}, z_{\nminusik})|\\
  &= \bigg|  \sum_{z_{\nminusik}'}d(z_{\nminusik}'|z_{\nik} ) Q_i^*(z_{\nik},z_{\nminusik}') - \sum_{z_{\nminusik}'}d(z_{\nminusik}'|z_{\nik} ) Q_i^*(z_{\nik},z_{\nminusik}) \bigg|\\
  &\leq
 \sum_{z_{\nminusik}'}d(z_{\nminusik}'|z_{\nik} ) \big|Q_i^*(z_{\nik},z_{\nminusik}') -Q_i^*(z_{\nik},z_{\nminusik})\big| \\
  &\leq c \rhok,
\end{align*}
where the last inequality is due to the exponential decay property (cf. Definition~\ref{def:exp_decaying} and Assumption~\ref{assump:exp_decaying}). Therefore,
$$\Vert \Phi \Pi^d Q_i^* - Q_i^*\Vert_\infty\leq  c \rhok.$$
Combining the above with \eqref{eq:lem_d:fixed_point_error_1}, we get the desired result
$$  \Vert \Phi \hat{Q}_i^{d} - Q_i^*\Vert_\infty \leq \frac{ c \rhok}{1-\gamma}.  $$ \end{proof}

\subsection{Proof of Lemma~\ref{appendix:critic:lem:error_recursive}}\label{appendix:critic:subsec:error_recursive}
%In this section, we provide the proof of Lemma~\ref{appendix:critic:lem:error_recursive}.
%\begin{proof}
    %We fix $z = (s,a)$, and use $z_{\nik}$ to represent the local state-action pair $z_{\nik} = (s_{N_i}^k,a_{\nik})$. 
    Recall that the $z_{\nik}$'th diagonal entry of $B_{k,t}$ is $b_{k,t}(z_{\nik})$, and we define similarly the $z_{\nik}$'th diagonal entry of $\tilde{B}_{k,t}$ to be $\tilde{b}_{k,t}(z_{\nik})$. Using these notations, equation \eqref{eq:hatq_recursive} can be written as,
    \begin{align}
        \hat{Q}_i^t(z_{\nik}) &= \overbrace{ \tilde{b}_{\tau-1,t}(z_{\nik})\hat{Q}_i^\tau(z_{\nik}) + \sum_{k=\tau}^{t-1}  b_{k,t}(z_{\nik})   [g_k (\hat{Q}_i^k)](z_{\nik}) }^{:= G(z_{\nik})} \nonumber \\
        &\qquad + 
 \sum_{k=\tau}^{t-1}  \alpha_k \tilde{b}_{k,t}(z_{\nik}) (\epsilon_k(z_{\nik}) + \phi_k(z_{\nik})). \label{eq:contraction_hatq_recursive_componentwise}
    \end{align}
    Notice that by definition, $\tilde{b}_{\tau-1,t}(z_{\nik}) +\sum_{k=\tau}^{t-1}  b_{k,t}(z_{\nik}) =1$. Then,
\begin{align}
    |G(z_{\nik}) - Q_i^*(z)|&\leq \tilde{b}_{\tau-1,t}(z_{\nik})|\hat{Q}_i^\tau(z_{\nik}) - Q_i^*(z)| + \sum_{k=\tau}^{t-1}  b_{k,t}(z_{\nik})  | [g_k (\hat{Q}_i^k)](z_{\nik})-Q_i^*(z)|  \nonumber \\
    &\leq \tilde{b}_{\tau-1,t}(z_{\nik})|\hat{Q}_i^\tau(z_{\nik}) - Q_i^*(z)| + \sum_{k=\tau}^{t-1}  b_{k,t}(z_{\nik})  | [g_k (\hat{Q}_i^k)](z_{\nik})- \hat{Q}_i^{*,k}(z_{\nik}) | \nonumber\\
    &\quad + \sum_{k=\tau}^{t-1}  b_{k,t}(z_{\nik})  | {Q}_i^*(z)- \hat{Q}_i^{*,k}(z_{\nik}) |\nonumber\\
    &\leq  \tilde{b}_{\tau-1,t}(z_{\nik})|\hat{Q}_i^\tau(z_{\nik}) - Q_i^*(z)| + \gamma \sum_{k=\tau}^{t-1}  b_{k,t}(z_{\nik})  \Vert \hat{Q}_i^k- \hat{Q}_i^{*,k} \Vert_\infty \nonumber\\
    &\quad + \sum_{k=\tau}^{t-1}  b_{k,t}(z_{\nik})  \| {Q}_i^*-  \Phi \hat{Q}_i^{*,k} \|_\infty \nonumber\\
    &\leq  \tilde{b}_{\tau-1,t}(z_{\nik})|\hat{Q}_i^\tau(z_{\nik}) - Q_i^*(z)| + \gamma \sum_{k=\tau}^{t-1}  b_{k,t}(z_{\nik})  \Vert \Phi  \hat{Q}_i^k- {Q}_i^* \Vert_\infty \nonumber\\
    &\quad + 2 \sum_{k=\tau}^{t-1}  b_{k,t}(z_{\nik})  \| {Q}_i^*-  \Phi \hat{Q}_i^{*,k} \|_\infty\nonumber\\
    &\leq \tilde{\beta}_{\tau-1,t} \xi_\tau + \gamma \sum_{k=\tau}^{t-1}  b_{k,t}(z_{\nik})  \xi_k +    \frac{ 2 c \rhok}{1-\gamma} ,
\end{align}
where in the thrid inequality, we have used that $g_k$ is $\gamma$-contraction in infinity norm with fixed point $\hat{Q}_i^{*,k}$, and in the last inequality, we have used the property of $\hat{Q}_i^{*,k}$ in Lemma~\ref{lem:critic:barab}.
Combining the above with \eqref{eq:contraction_hatq_recursive_componentwise},  we have
\begin{align*}
 & \xi_t=  \Vert \Phi \hat{Q}_i^t - Q_i^*\Vert_\infty \\
  &\leq \tilde{\beta}_{\tau-1,t}\xi_\tau  + \gamma \sup_{z_{\nik}}\sum_{k=\tau}^{t-1}  b_{k,t}(z_{\nik})  \xi_k +  \frac{ 2 c \rhok}{1-\gamma}  +  \Vert \sum_{k=\tau}^{t-1}  \alpha_k \tilde{B}_{k,t} \epsilon_k\Vert_\infty  + \Vert \sum_{k=\tau}^{t-1}  \alpha_k \tilde{B}_{k,t}  \phi_k \Vert_\infty.
\end{align*}
%We also notice that by Lemma~\ref{appendix:critic:lem:bounded} (d), (to add later)
%$$a_\tau= \Vert\Phi\hat{Q}_i^\tau - Q_i^*\Vert_\infty\leq \Vert\Phi\hat{Q}_i^0 - Q_i^*\Vert_\infty + \Vert\hat{Q}_i^0 - \hat{Q}_i^\tau\Vert_\infty \leq a_0 + [2\frac{\bar{r}}{1-\gamma}+\bar{r}]\tau\alpha_0 $$
%\end{proof} 
\qed
\subsection{Proof of Lemma \ref{appendix:critic:lem:bounded} and Lemma \ref{appendix:critic:lem:stepsize}} \label{appendix:critic:subsec:auxilliary}
In this section, we provide proofs of the two auxiliary lemmas, Lemma \ref{appendix:critic:lem:bounded} and Lemma \ref{appendix:critic:lem:stepsize}. We start with the proof of Lemma~\ref{appendix:critic:lem:bounded}. 

\begin{proof}[Proof of Lemma~\ref{appendix:critic:lem:bounded}]
        First, notice that $A_{z,z'} =\mathbf{e}_{z_{\nik}} [ \gamma   \mathbf{e}_{z_{\nik}'}^T - \mathbf{e}_{z_{\nik}}^T ]$ and $b_{z} = \mathbf{e}_{z_{\nik}} r_i(z_i)$. As such, $\Vert A_{z,z'}\Vert_\infty \leq 1 + \gamma<2$, $\Vert b_{z}\Vert_\infty \leq \bar{r}$.
        
        \bigskip
        Part (a) can be proved by induction. Part (a) is true for $t=0$ as $\hat{Q}_i^0 = 0$. Assume $\Vert \hat{Q}_i^{t-1}\Vert_\infty\leq \frac{\bar{r}}{1-\gamma}$. Recall the update equation \eqref{eq:q_update_vec_0},
        \begin{align}
\hat{Q}_i^{t} = \hat{Q}_i^{t-1} + \alpha_{t-1}[r_i(z_i(t-1))+ \gamma\hat{Q}_i^{t-1}(z_{\nik}(t)) - \hat{Q}_i^{t-1}(z_{\nik}(t-1)) ]\mathbf{e}_{z_{\nik}(t-1)}, \nonumber
\end{align}
or in other words, 
\begin{align*}
    \hat{Q}_i^{t}(z_{\nik}(t-1)) &= \hat{Q}_i^{t-1}(z_{\nik}(t-1))   + \alpha_{t-1}[r_i(z_i(t-1))+ \gamma\hat{Q}_i^{t-1}(z_{\nik}(t)) - \hat{Q}_i^{t-1}(z_{\nik}(t-1)) ]\\
    &= (1-\alpha_{t-1}) \hat{Q}_i^{t-1}(z_{\nik}(t-1))  + \alpha_{t-1}[r_i(z_i(t-1))+ \gamma\hat{Q}_i^{t-1}(z_{\nik}(t)) ].
\end{align*}
And for other entries of $\hat{Q}_i^{t}$, it stays the same as $\hat{Q}_i^{t-1}$. For this reason, 
\begin{align*}
    \Vert \hat{Q}_i^t\Vert_\infty & \leq \max(\Vert \hat{Q}_i^{t-1}\Vert_\infty, | \hat{Q}_i^{t}(z_{\nik}(t-1))|) .
\end{align*}
Notice that
\begin{align*}
    &| \hat{Q}_i^{t}(z_{\nik}(t-1))|\leq (1-\alpha_{t-1}) \frac{\bar{r}}{1-\gamma} + \alpha_{t-1}(\bar{r} + \gamma \frac{\bar{r}}{1-\gamma}) = \frac{\bar{r}}{1-\gamma},
\end{align*}
which finishes the induction and the proof of part (a).  

\bigskip
For part (b), notice that $\epsilon_t=(A_t - \bar{A}_t)\hat{Q}_i^{t+1-\tau} + b_{t} - \bar b_t$. Therefore, it is easy to check that by part (a), $\Vert \epsilon_t\Vert_\infty \leq  4 \frac{\bar{r}}{1-\gamma} + 2\bar{r}= \bar{\epsilon}$.

\bigskip
For part (c), notice that, for any $k$
\begin{align*}
    \Vert \hat{Q}_i^{k} - \hat{Q}_i^{k-1}\Vert_\infty = \alpha_{k-1} \Vert A_{k-1}\hat{Q}_i^{k-1} + b_{k-1} \Vert_\infty\leq \alpha_{k-1} [2\frac{\bar{r}}{1-\gamma} + \bar{r}].
\end{align*}
Therefore, by triangle inequality,
\begin{align*}
    \Vert \hat{Q}_i^{t-1} - \hat{Q}_i^{t-\tau}\Vert_\infty \leq [2\frac{\bar{r}}{1-\gamma} + \bar{r}] \sum_{k=t-\tau}^{t-2} \alpha_k.
\end{align*}
As a consequence, $$\Vert \phi_{t}\Vert_\infty \leq \Vert A_t - \bar{A}_t\Vert_\infty \Vert \hat{Q}_i^t - \hat{Q}_i^{t-\tau+1}\Vert_\infty \leq [8\frac{\bar{r}}{1-\gamma} +4 \bar{r}] \sum_{k=t-\tau+1}^{t-1} \alpha_k = 2 \bar{\epsilon} \sum_{k=t-\tau+1}^{t-1} \alpha_k .$$
\end{proof}

%Next, we provide the proof of Lemma~\ref{appendix:critic:lem:stepsize}.

\begin{proof}[Proof of Lemma~\ref{appendix:critic:lem:stepsize}]      Notice that $\log(1-x)\leq -x$ for all $x<1$. Then, 
        $$ (1 - \sigma\alpha_t) = e^{\log( 1 - \frac{\sigma h}{t+t_0})} \leq e^{- \frac{\sigma h}{t+t_0}}. $$
    Therefore,
    \begin{align*}
        \prod_{\ell=k+1}^{t-1}(1-\sigma \alpha_\ell)& \leq e^{-\sum_{\ell=k+1}^{t-1}  \frac{\sigma h}{\ell+t_0}}\\
        &\leq e^{-\int_{\ell=k+1}^{t}  \frac{\sigma h}{\ell+t_0} d\ell} \\
        &= e^{-\sigma h \log(\frac{t+t_0}{k+1+t_0})}\\
        &= \Big( \frac{k+1+t_0}{t+t_0}\Big)^{\sigma h},
    \end{align*}
    which leads to the bound on $\beta_{k,t}$ and $\tilde{\beta}_{k,t}$.
    
    For part (b), 
    $$\beta_{k,t}^2 \leq  \frac{h^2}{(t+t_0)^{2\sigma h}} \frac{(k+1+t_0)^{2\sigma h}}{(k+t_0)^2} \leq \frac{2 h^2}{(t+t_0)^{2\sigma h}} (k+t_0)^{2\sigma h- 2} ,$$
    where we have used $(k+1+t_0)^{2\sigma h} \leq  2 (k+t_0)^{2\sigma h}$, which is true when $t_0 \geq 4\sigma h$. Then,
    \begin{align*}
        \sum_{k=1}^{t-1}\beta_{k,t}^2&\leq \frac{2 h^2}{(t+t_0)^{2\sigma h}}  \sum_{k=1}^{t-1} (k+t_0)^{2\sigma h- 2}\leq  \frac{2 h^2}{(t+t_0)^{2\sigma h}} \int_{1}^{t} (y+t_0)^{2\sigma h- 2} dy\\
    &< \frac{2 h^2}{(t+t_0)^{2\sigma h}} \frac{1}{2\sigma h - 1} (t+t_0)^{2\sigma h - 1} < \frac{2h}{\sigma } \frac{1}{(t+t_0)},
    \end{align*}
    where in the last inequality we have used $2\sigma h-1 > \sigma h$. 
    
    For part (c), notice that for $k-\tau + 1\leq \ell\leq k-1 $ where $k\geq \tau$, we have $\alpha_\ell \leq \frac{h}{k-\tau+ t_0} \leq \frac{2h}{k+t_0}$ (using $t_0\geq \tau$). Then,
    \begin{align*}
        \sum_{k=\tau}^{t-1}\beta_{k,t}\sum_{\ell = k-\tau+1}^{k-1} \alpha_\ell &\leq  \sum_{k=\tau}^{t-1}\beta_{k,t} \frac{2h\tau}{k+t_0}\leq \sum_{k=\tau}^{t-1} \frac{h}{k+t_0} \Big( \frac{k+1+t_0}{t+t_0}\Big)^{\sigma h} \frac{2h\tau}{k+t_0}\\
        &\leq \sum_{k=\tau}^{t-1} \frac{4h^2\tau}{(t+t_0)^{\sigma h}} (k+t_0)^{\sigma h - 2}\\
        &\leq \frac{4h^2\tau}{(t+t_0)^{\sigma h}} \frac{(t+t_0)^{\sigma h -1}}{\sigma h -1}\\
        &\leq \frac{8h\tau}{\sigma} \frac{1}{t+t_0},
    \end{align*}
    where we have used $(k+1+t_0)^{\sigma h}\leq 2(k+t_0)^{\sigma h}$, and $\sigma h-1 > \frac{1}{2} \sigma h$. 
\end{proof}

\subsection{Proof of Lemma~\ref{appendix:critic:lem:martingale_bound} and Lemma~\ref{appendix:critic:lem:phi_bound}} \label{appendix:critic:subsec:epsilon_phi}

Since $\Vert \sum_{k=\tau}^{t-1}  \alpha_k \tilde{B}_{k,t}  \phi_k \Vert_\infty\leq  \sum_{k=\tau}^{t-1}\beta_{k,t} \Vert \phi_k\Vert_\infty$, Lemma~\ref{appendix:critic:lem:bounded} (c) and Lemma~\ref{appendix:critic:lem:stepsize} (c) directly leads to the bound in Lemma~\ref{appendix:critic:lem:phi_bound}. So, in this section, we focus on the proof of Lemma~\ref{appendix:critic:lem:martingale_bound}. We start by stating a variant of the Azuma-Hoeffding bound that handles our ``shifted'' Martingale difference sequence.

\begin{lemma}\label{appendix:critic:lem:azuma}
	Let $X_t$ be a $\mathcal{F}_t$-adapted stochastic process, satisfying
	$ \E X_t | \mathcal{F}_{t-\tau}=0$. Further, $|X_t|\leq \bar X_t$ almost surely. Then with probability $1-\delta$, we have,
	$$ |\sum_{k=0}^t X_{t}|  \leq  \sqrt{ 2\tau  \sum_{k=0}^t \bar{X}_{k}^2 \log(\frac{2\tau}{\delta}) }. $$
\end{lemma}
\begin{proof}[Proof of Lemma~\ref{appendix:critic:lem:azuma}]       	Let $\ell$ be an integer between $0$ and $\tau-1$. For each $\ell$, define process $Y^\ell_k = X_{\tau k + \ell}$, scalar $\bar{Y}^\ell_k = \bar{X}_{k\tau+\ell}$, and define Filtration $\tilde{\mathcal{F}}_k^\ell = \mathcal{F}_{\tau k+\ell}$. Then, $Y_k^\ell$ is $\tilde{\mathcal{F}}_k^\ell$-adapted, and satisfies
       	$$ \E Y_k^\ell | \tilde{\mathcal{F}}_{k-1}^\ell = \E X_{k\tau + \ell}| \mathcal{F}_{k\tau+\ell - \tau} = 0.$$
       	Therefore, applying Azuma-Hoeffding bound on $Y_{k}^\ell$, we have
       	$$P( |\sum_{k: k\tau + \ell\leq t} Y_{k}^\ell| \geq t ) \leq 2  \exp(-\frac{t^2}{2 \sum_{k:k\tau+\ell\leq t} (\bar{Y}^\ell_{k})^2}),$$
       	i.e. with probability at least $1 - \frac{\delta}{\tau}$,
       	$$\Big|\sum_{k: k\tau + \ell\leq t} X_{k\tau+\ell}\Big|=  \Big|\sum_{k: k\tau + \ell\leq t} Y_{k}^\ell \Big| \leq \sqrt{ 2  \sum_{k:k\tau+\ell\leq t} \bar{X}_{k\tau+\ell}^2 \log(\frac{2\tau}{\delta}) }.$$
       	Using the union bound for $\ell = 0,\ldots,\tau-1$, we get that with probability at least $1-\delta$,
       	$$\Big|\sum_{k=0}^t X_{t}\Big| \leq \sum_{\ell=0}^{\tau-1}\Big|\sum_{k: k\tau + \ell\leq t} X_{k\tau+\ell}\Big| \leq \sum_{\ell=0}^{\tau -1} \sqrt{ 2  \sum_{k:k\tau+\ell\leq t} \bar{X}_{k\tau+\ell}^2 \log(\frac{2\tau}{\delta}) }\leq  \sqrt{ 2\tau  \sum_{k=0}^t \bar{X}_{k}^2 \log(\frac{2\tau}{\delta}) },$$
    where the last inequality is due to Cauchy-Schwarz. 
\end{proof}

%We are now ready to prove Lemma~\ref{appendix:critic:lem:martingale_bound}.

%\adam{the proof below is awkward because of lemma 16...we should reorganize}
%\guannan{I have reorganized it}

Recall that Lemma~\ref{appendix:critic:lem:martingale_bound} is an upper bound on $\Vert\sum_{k=\tau}^{t-1} \alpha_k \tilde{B}_{k,t} \epsilon_k \Vert $, where $\sum_{k=\tau}^{t-1} \alpha_k \tilde{B}_{k,t} \epsilon_k$ is a random vector in $\R^{\mathcal{Z}_{\nik}}$, with its $z_{\nik}$'th entry being 
\begin{align}
    \sum_{k=\tau}^{t-1} \alpha_{k} \epsilon_{k}(z_{\nik}) \prod_{\ell=k+1}^{t-1} (1- \alpha_\ell \bar{d}_\ell(z_{\nik})), \label{appendix:critic:eq:martingale_entrywise}\end{align}
with $\bar{d}_\ell(z_{\nik})\geq \sigma$ almost surely, cf. Lemma~\ref{lem:critic:barab}. Fixing $z_{\nik}$, as have been shown in \eqref{appendix:critic:eq:martingale}, $\epsilon_{k}(z_{\nik})$ is a $\mathcal{F}_{k+1}$ adapted stochastic process satisfying $\E \epsilon_{k}(z_{\nik}) | \mathcal{F}_{k+1-\tau}=0$.  However, $\prod_{\ell=k+1}^{t-1} (1- \alpha_\ell \bar{d}_\ell(z_{\nik}))$ is not $\mathcal{F}_{k+1-\tau}$-measurable, and as such we cannot directly apply the Azuma-Hoeffding bound in Lemma~\ref{appendix:critic:lem:azuma} to quantity \eqref{appendix:critic:eq:martingale_entrywise}. In what follows, we first show in Lemma~\ref{appendix:critic:lem:martingale_upperbound_beta} that almost surely, the absolute value of quantity \eqref{appendix:critic:eq:martingale_entrywise} can be upper bounded by the sup of another quantity, to which we can directly apply Lemma~\ref{appendix:critic:lem:azuma}. With the help of Lemma~ \ref{appendix:critic:lem:martingale_upperbound_beta}, we can use the Azuma-Hoeffding bound to control \eqref{appendix:critic:eq:martingale_entrywise} and prove Lemma~\ref{appendix:critic:lem:martingale_bound}.

\begin{lemma}\label{appendix:critic:lem:martingale_upperbound_beta}
    For each $z_{\nik}$, we have almost surely,
    $$\Big|\sum_{k=\tau}^{t-1} \alpha_{k} \epsilon_{k}(z_{\nik}) \prod_{\ell=k+1}^{t-1} (1- \alpha_\ell \bar{d}_\ell(z_{\nik})) \Big| \leq  \sup_{\tau\leq k_0\leq t-1}\bigg(\big|\sum_{k=k_0+1}^{t-1}  \epsilon_k(z_\nik)\beta_{k,t}\big| + 2 \bar{\epsilon} {\beta}_{k_0,t} \bigg).$$
\end{lemma}
\begin{proof}[Proof of Lemma~\ref{appendix:critic:lem:martingale_upperbound_beta}]    Let $p_k$ be a scalar sequence defined as follows. Set $p_\tau=0$, and $$p_{k} = (1 - \alpha_{k-1} \bar{d}_{k-1}(z_{\nik}))p_{k-1} + \alpha_{k-1} \epsilon_{k-1}(z_{\nik}) .$$
    Then 
    $p_t = \sum_{k=\tau}^{t-1} \alpha_{k} \epsilon_{k}(z_{\nik}) \prod_{\ell=k+1}^{t-1} (1- \alpha_\ell \bar{d}_\ell(z_{\nik})) $, and to prove Lemma~\ref{appendix:critic:lem:martingale_upperbound_beta} we need to bound $|p_t|$. 
    Let $$k_0 = \sup \{k\leq t-1: (1 - \alpha_{k} \bar{d}_{k}(z_{\nik}))|p_{k}| \leq \alpha_{k} |\epsilon_{k}(z_{\nik})|\}.$$
     We must have $k_0\geq \tau$ since $|p_\tau| = 0$. 
    With $k_0$ defined, we now define another scalar sequence $\tilde{p}$ s.t. $\tilde{p}_{k_0+1} = p_{k_0+1}$ and 
    $$\tilde{p}_{k} = (1 - \alpha_{k-1}\sigma)\tilde{p}_{k-1} + \alpha_{k-1} \epsilon_{k-1}(z_{\nik}) .$$
    We claim that for all $k\geq k_0+1$, $p_k$ and $\tilde{p}_{k}$ have the same sign, and $|p_k| \leq |\tilde p_{k}|$. This is obviously true for $k=k_0+1$. Suppose it is true for for $k-1$. Without loss of generality, suppose both $p_{k-1}$ and $\tilde{p}_{k-1}$ are non-negative. Since $k-1>k_0$ and by the definition of $k_0$, we must have
    $$ (1 - \alpha_{k-1} \bar{d}_{k-1}(z_{\nik}))p_{k-1} > |\alpha_{k-1} \epsilon_{k-1}(z_{\nik})|. $$
    Therefore, $p_k>0$. Further, since $\bar{d}_{k-1}(z_{\nik})\geq \sigma$, we also have
    $$(1 - \alpha_{k-1} \sigma)\tilde{p}_{k-1}\geq (1 - \alpha_{k-1} \bar{d}_{k-1}(z_{\nik}))p_{k-1} > |\alpha_{k-1} \epsilon_{k-1}(z_{\nik})| . $$
    These imply $ \tilde{p}_k\geq p_k >0$. The case where both $p_{k-1}$ and $\tilde{p}_{k-1}$ are negative are similar. This finishes the induction, and as a result, $|p_t|\leq |\tilde{p}_t|$.
    %So $p_k$ will be of the same sign as $p_{k-1}$; Since , $\tilde{p}_k$ will still be the same sign as $\tilde{p}_{k-1}$ and $|\tilde{p}_k|\geq |p_k|$it should also be true for $k$. So the induction is complete.
    
    Notice, 
    $$\tilde{p}_t = \sum_{k=k_0+1}^{t-1} \alpha_k \epsilon_k(z_{\nik})\prod_{\ell=k+1}^{t-1}(1-\alpha_\ell \sigma) + \tilde{p}_{k_0+1}\prod_{\ell=k_0+1}^{t-1}(1-\alpha_\ell \sigma)= \sum_{k=k_0+1}^{t-1}  \epsilon_k(z_{\nik})\beta_{k,t} + \tilde{p}_{k_0+1}\tilde{\beta}_{k_0,t}.$$
    By the definition of $k_0$, we have $$|p_{k_0+1}| \leq (1-\alpha_{k_0}\bar{d}_{k_0}(z_{\nik}))|p_{k_0}| + \alpha_{k_0} |\epsilon_{k_0}(z_{\nik})| \leq 2\alpha_{k_0} |\epsilon_{k_0}(z_{\nik})| \leq 2\alpha_{k_0} \bar{\epsilon},$$ 
 where in the last step, we have used the upper bound on $\Vert \epsilon_{k_0}\Vert_\infty$ in Lemma~\ref{appendix:critic:lem:bounded} (b). As a result, 
    \begin{align*}
        |p_t| & \leq |\tilde{p}_t| \leq \big|\sum_{k=k_0+1}^{t-1}  \epsilon_k(z_{\nik})\beta_{k,t}\big| + \big|\tilde{p}_{k_0+1}\tilde{\beta}_{k_0,t}\big|\\
        &\leq \big|\sum_{k=k_0+1}^{t-1}  \epsilon_k(z_{\nik})\beta_{k,t}\big| + \big|2 \alpha_{k_0}\bar{\epsilon} \tilde{\beta}_{k_0,t}\big|\\
        &= \big|\sum_{k=k_0+1}^{t-1}  \epsilon_k(z_{\nik})\beta_{k,t}\big| + 2 \bar{\epsilon} {\beta}_{k_0,t}.
    \end{align*}
\end{proof}
With the above preparations, we are now ready to prove Lemma~\ref{appendix:critic:lem:martingale_bound}.

\begin{proof}[Proof of Lemma~\ref{appendix:critic:lem:martingale_bound}]
%Notice that $\epsilon_k$ is $\mathcal{F}_{k+1}$-measurable, and satisfies 
%$$\mathbb{E} \epsilon_k | \mathcal{F}_{k-\tau+1} = 0 $$ 
Fix $z_{\nik}$ and $\tau\leq k_0\leq t-1$. As have been shown in \eqref{appendix:critic:eq:martingale}, $\epsilon_{k}(z_{\nik})\beta_{k,t}$ is a $\mathcal{F}_{k+1}$ adapted stochastic process satisfying $\E \epsilon_{k} (z_{\nik}) \beta_{k,t} | \mathcal{F}_{k+1-\tau}=0$.
Also by Lemma~\ref{appendix:critic:lem:bounded}(b), $|\epsilon_k(z_{\nik})\beta_{k,t}|\leq \bar{\epsilon}\beta_{k,t}$ almost surely. As a result, we can use the Azuma-Hoeffding bound in Lemma~\ref{appendix:critic:lem:azuma} to get with probability $1-\delta$,
 $$ \Big|\sum_{k=k_0+1}^{t-1}  \epsilon_k(z_{\nik})\beta_{k,t}\Big| \leq \bar{\epsilon} \sqrt{ 2\tau  \sum_{k=k_0+1}^{t-1} \beta_{k,t}^2 \log(\frac{2\tau}{\delta}) } .  $$
By a union bound on $\tau\leq k_0\leq t-1$, we get with probability $1-\delta$,
 $$ \sup_{\tau\leq k_0\leq t-1 }\big|\sum_{k=k_0+1}^{t-1}  \epsilon_k(z_{\nik})\beta_{k,t}\big| \leq \sup_{\tau\leq k_0\leq t-1 } \bar{\epsilon} \sqrt{ 2\tau  \sum_{k=k_0+1}^{t-1} \beta_{k,t}^2 \log(\frac{2\tau t}{\delta}) } \leq    \bar{\epsilon} \sqrt{ 2\tau  \sum_{k=\tau+1}^{t-1} \beta_{k,t}^2 \log(\frac{2\tau t}{\delta}) } .$$
Then, by Lemma~\ref{appendix:critic:lem:martingale_upperbound_beta}, we have with probability $1-\delta$, 
\begin{align*}
    \Big|\sum_{k=\tau}^{t-1} \alpha_{k} \epsilon_{k}(z_{\nik}) \prod_{\ell=k+1}^{t-1} (1- \alpha_\ell \bar{d}_\ell(z_{\nik}))\Big| &\leq  \sup_{\tau\leq k_0\leq t-1}\bigg(\big|\sum_{k=k_0+1}^{t-1}  \epsilon_k(z_{\nik})\beta_{k,t}\big| + 2 \bar{\epsilon} {\beta}_{k_0,t} \bigg) \\
    &\leq \bar{\epsilon} \sqrt{ 2\tau  \sum_{k=\tau+1}^{t-1} \beta_{k,t}^2 \log(\frac{2\tau t}{\delta}) } +  \sup_{\tau\leq k_0\leq t-1}2 \bar{\epsilon} {\beta}_{k_0,t} \\
    &\leq  2 \bar{\epsilon} \sqrt{  \frac{\tau h}{\sigma(t+t_0)}   \log(\frac{2\tau t}{\delta}) } + \sup_{\tau\leq k_0\leq t-1}2 \bar{\epsilon} \frac{h}{k_0+t_0} \Big( \frac{k_0+1+t_0}{t+t_0}\Big)^{\sigma h}\\
    &\leq  2 \bar{\epsilon} \sqrt{  \frac{\tau h}{\sigma(t+t_0)}   \log(\frac{2\tau t}{\delta}) } + 2 \bar{\epsilon} \frac{h}{t-1+t_0} \\
    &\leq 6 \bar{\epsilon} \sqrt{  \frac{\tau h}{\sigma(t+t_0)}   \log(\frac{2\tau t}{\delta}) }  ,
\end{align*}
where in the third inequality, we have used the bounds on $\beta_{k,t}$ in Lemma~\ref{appendix:critic:lem:stepsize}. Finally, apply the union bound over $z_{\nik}\in \mathcal{Z}_{\nik}$, and noticing that $|\nik|\leq \fk$ and $|\mathcal{Z}_{\nik}|\leq (SA)^{\fk}$ by Assumption~\ref{assump:bounded_reward}, we have with probability $1-\delta$,
$$\Big\Vert \sum_{k=\tau}^{t-1} \alpha_{k}\tilde{B}_{k,t} \epsilon_{k} \Big\Vert_\infty \leq 6 \bar{\epsilon} \sqrt{  \frac{\tau h}{\sigma(t+t_0)}   \log(\frac{2\tau t (SA)^{\fk}}{\delta}) } =  6 \bar{\epsilon} \sqrt{  \frac{\tau h}{\sigma(t+t_0)}  [ \log(\frac{2\tau t }{\delta}) + \fk\log SA]}  . $$
\end{proof}

\subsection{Proof of Lemma~\ref{lem:b_kt_bound}}\label{appendix:critic:subsec:b_kt_bound}
%In this section, we provide the proof of Lemma~\ref{lem:b_kt_bound}.

%\noindent\textbf{Proof of Lemma~\ref{lem:b_kt_bound}}
Throughout the proof, we fix $z_{\nik}$ and prove the desired upper bound. For notational simplicity, we drop the dependence on $z_{\nik}$ and write $b_{k,t}$ and $\bar{d}_k$ instead, and we will use the property $\bar{d}_k\geq \sigma$. Define the sequence 
$$e_t =\sum_{k=\tau}^{t-1}  b_{k,t} \frac{1}{{(k+t_0)^\omega}}. $$
We use induction to show that $e_t \leq \frac{1}{ \sqrt{\gamma}{(t+t_0)^\omega}}$. The statement is clearly true for $t=\tau+1$, as $e_{\tau+1} =  b_{\tau,\tau+1} \frac{1}{{(\tau+t_0)^\omega}}= \alpha_\tau \bar{d}_\tau \frac{1}{{(\tau+t_0)^\omega}} \leq  \frac{1}{\sqrt{\gamma}(\tau+1+t_0)^\omega}$ (last step needs $\alpha_\tau \leq \frac{1}{2}, (1+\frac{1}{t_0})^\omega\leq \frac{2}{\sqrt{\gamma}} $, implied by $t_0\geq 1$, $\omega\leq 1$). Let the statement be true for $t-1$. Then, notice that,
\begin{align*}
e_t &=\sum_{k=\tau}^{t-2}  b_{k,t} \frac{1}{(k+t_0)^\omega} + b_{t-1,t} \frac{1}{{(t-1+t_0)^\omega}} \\
&= (1-\alpha_{t-1}\bar{d}_{t-1}) \sum_{k=\tau}^{t-2}  b_{k,t-1} \frac{1}{(k+t_0)^\omega} + \alpha_{t-1} \bar{d}_{t-1} \frac{1}{{(t-1+t_0)^\omega}}\\
&= (1-\alpha_{t-1}\bar{d}_{t-1})  e_{t-1}+ \alpha_{t-1} \bar{d}_{t-1} \frac{1}{{(t-1+t_0)^\omega}}\\
&\leq (1-\alpha_{t-1}\bar{d}_{t-1})  \frac{1}{\sqrt{\gamma}{(t-1+t_0)}^\omega} + \alpha_{t-1}\bar{d}_{t-1}  \frac{1}{{(t-1+t_0)^\omega}}\\
&= \Big[ 1 - \alpha_{t-1} \bar{d}_{t-1} (1-\sqrt{\gamma})\Big]\frac{1}{\sqrt{\gamma}{(t-1+t_0)}^\omega} ,
\end{align*}
where the inequality is based on induction assumption. Then, plug in $\alpha_{t-1} = \frac{h}{t-1+t_0}$ and use $\bar{d}_{t-1}\geq \sigma$, we have,
\begin{align*}
e_t&\leq \Big[ 1 - \frac{\sigma h}{t-1+t_0} (1-\sqrt{\gamma})\Big]\frac{1}{\sqrt{\gamma}{(t-1+t_0)}^\omega} \\
&= \Big[ 1 - \frac{\sigma h}{t-1+t_0} (1-\sqrt{\gamma})\Big] \Big(\frac{t+t_0}{t-1+t_0}\Big)^\omega \frac{1}{\sqrt{\gamma}{(t+t_0)}^\omega}\\
&= \Big[ 1 - \frac{\sigma h}{t-1+t_0} (1-\sqrt{\gamma})\Big] \Big( 1+ \frac{1}{t-1+t_0}\Big)^\omega \frac{1}{\sqrt{\gamma}{(t+t_0)}^\omega}.
%&\leq \Big[ 1 - \frac{1}{t-1+t_0}\Big]\Big[ 1+ \frac{1}{t-1+t_0}\Big] \frac{1}{\sqrt{\gamma}{(t+t_0)}^\omega}< \frac{1}{\sqrt{\gamma}{(t+t_0)}^\omega}
\end{align*}
Now using the inequality that for any $x>-1$, $(1+x)\leq e^x$, we have,
\begin{align*}
\Big[ 1 - \frac{\sigma h}{t-1+t_0} (1-\sqrt{\gamma})\Big] \Big( 1+ \frac{1}{t-1+t_0}\Big)^\omega \leq  e^{- \frac{\sigma h}{t-1+t_0} (1-\sqrt{\gamma}) + \omega\frac{1}{t-1+t_0} } \leq 1,
\end{align*}
where in the last inequality, we have used $\omega\leq 1$ and the condition on $h$ s.t. $\sigma h (1-\sqrt{\gamma}) \geq 1$. This shows $e_t \leq  \frac{1}{\sqrt{\gamma}{(t+t_0)}^\omega}$ and finishes the induction. \qed

\section{Proof of Auxiliary Results for Analysis of Actor} \label{sec:appendix_sec}

\subsection{Proof of Lemma~\ref{appendix:actor:lem:upperbound}}\label{appendix:actor:subsec:upperbound}
%In this section, we provide the proof of Lemma~\ref{appendix:actor:lem:upperbound}.

%\noindent\textbf{Proof of Lemma~\ref{appendix:actor:lem:upperbound}} 
Recall that 
$$ \hat{g}_i(m) =  \sum_{t=0}^{T} \gamma^t  \frac{1}{n} \sum_{j\in \nik} \hat{Q}_j^{m,T} (s_{\njk}(t),a_{\njk}(t)) \nabla_{\theta_i} \log \zeta_i^{\theta_i(m)} (a_i(t)|s_i(t)). $$
Therefore,
\begin{align*}
    \Vert \hat{g}_i(m)\Vert &\leq  \sum_{t=0}^{T} \gamma^t  \frac{1}{n} \sum_{j\in \nik} |\hat{Q}_j^{m,T} (s_{\njk}(t),a_{\njk}(t))| \Vert \nabla_{\theta_i}  \log \zeta_i^{\theta_i(m)} (a_i(t)|s_i(t))\Vert\\
  &  \leq \sum_{t=0}^{T} \gamma^t \frac{\bar{r}}{1-\gamma} L_i< \frac{\bar{r}}{(1-\gamma)^2} L_i,
\end{align*}
where we have used that $\Vert \hat{Q}_j^{m,T}\Vert_\infty \leq \frac{\bar{r}}{1-\gamma}$ almost surely (cf. Lemma~\ref{appendix:critic:lem:bounded} (a)). As a result,
$$\Vert \hat{g}(m)\Vert = \sqrt{\sum_{i=1}^n \Vert \hat{g}_i(m)\Vert^2 } < \frac{\bar{r}}{(1-\gamma)^2}L.$$
The upper bounds for $\Vert g(m)\Vert$, $\Vert h(m)\Vert$ and $\Vert \nabla J(\theta(m))\Vert$ can be obtained in an almost identical way and their proof is therefore omitted.\qed
\subsection{Proof of Lemma~\ref{appendix:actor:lem:e1_bound}}\label{appendix:actor:subsec:e1_bound}
 Let $\mathcal{G}_m$ be the $\sigma$-algebra generated by the trajectories in the first $m$ outer-loop iterations. Then, Theorem~\ref{appendix:critic:thm} implies that, fixing each $m\leq M$ and $i\in \mathcal{N}$, conditioned on $\mathcal{G}_{m-1}$, the following event happens with probability at least $1-\delta$:
$$ \sup_{(s,a)\in\mathcal{S}\times \mathcal{A}} \big| Q_i^{\theta(m)}(s,a) - \hat{Q}_i^{m,T}(s_{\nik},a_{\nik}) \big| \leq \frac{C_a(\delta,T)}{\sqrt{T+t_0}} + \frac{C_a'}{T+t_0} + \frac{2c\rhok}{(1-\gamma)^2} ,$$
where $$C_a(\delta,T)=   \frac{6\bar{\epsilon}}{1-\sqrt{\gamma}}  \sqrt{  \frac{\tau h}{\sigma}  [ \log(\frac{2\tau T^2 }{\delta}) + \fk\log SA]} , C_a'= \frac{2}{1-\sqrt{\gamma}} \max( \frac{16\bar\epsilon  h\tau}{\sigma}, \frac{2 \bar{r}}{1-\gamma}{(\tau+t_0)}),$$
with $\bar{\epsilon} =4 \frac{\bar{r}}{1-\gamma} + 2\bar{r} $.

We can take expectation and average out $\mathcal{G}_{m-1}$, and apply union bound over $0\leq m\leq M-1$ and $i\in \mathcal{N}$, getting with probability at least $1-\frac{\delta}{2}$,
\begin{align}
  \sup_{m\leq M-1}\sup_{i\in\mathcal{N}}  \sup_{(s,a)\in\mathcal{S}\times \mathcal{A}} \big| Q_i^{\theta(m)}(s,a) - \hat{Q}_i^{m,T}(s_{\nik},a_{\nik}) \big| &\leq \frac{C_a(\frac{\delta}{2 nM},T)}{\sqrt{T+t_0}} + \frac{C_a'}{T+t_0} + \frac{2c\rhok}{(1-\gamma)^2} \nonumber \\
  &\leq \frac{4c\rhok}{(1-\gamma)^2}, \label{appendix:actor:eq:critic_error}
\end{align}
where in the last step, we have used that our lower bound on $T$ implies  $\frac{C_a(\frac{\delta}{2 nM},T)}{\sqrt{T+t_0}} + \frac{C_a'}{T+t_0}  \leq \frac{2c\rhok}{(1-\gamma)^2}  $. Therefore, conditioned on \eqref{appendix:actor:eq:critic_error} being true, we have for any $m\leq M-1$ and any $i\in\mathcal{N}$,
\begin{align*}
   & \Vert \hat{g}_i(m) - g_i(m)\Vert \\
    &\leq \big\Vert \sum_{t=0}^{T} \gamma^t \frac{1}{n} \sum_{j\in \nik}\big[ Q_j^{\theta(m)} (s(t),a(t)) - \hat{Q}_j^{m,T} (s_{\njk}(t),a_{\njk}(t))\big]\nabla_{\theta_i} \log \zeta_i^{\theta_i(m)} (a_i(t)|s_i(t))  \big \Vert\\
    &\leq \sum_{t=0}^{T} \gamma^t \frac{1}{n} \sum_{j\in \nik}\Big| Q_j^{\theta(m)} (s(t),a(t)) - \hat{Q}_j^{m,T} (s_{\njk}(t),a_{\njk}(t))\Big|  \big\Vert\nabla_{\theta_i} \log \zeta_i^{\theta_i(m)} (a_i(t)|s_i(t))  \big \Vert\\
    &\leq \sum_{t=0}^{T} \gamma^t \frac{4c\rhok}{(1-\gamma)^2}  L_i< \frac{4cL_i\rhok}{(1-\gamma)^3}.
\end{align*}
As a result, 
$$\sup_{0\leq m\leq M-1}\Vert \hat{g}(m) - g(m)\Vert \leq \frac{4c L\rhok}{(1-\gamma)^3}, $$
which is true conditioned on event \eqref{appendix:actor:eq:critic_error} is true that happens with probability at least $1-\frac{\delta}{2}$.
\qed

\subsection{Proof of Lemma~\ref{appendix:actor:lem:martingale_bound}}\label{appendix:actor:subsec:martingale_bound}
    By Lemma~\ref{appendix:actor:lem:upperbound}, we have almost surely,
    $$|  \eta_m \langle \nabla J(\theta(m)), e^2(m) \rangle  | \leq \eta_m \Vert \nabla J(\theta(m))\Vert \Vert h(m) - g(m)\Vert \leq \eta_m \frac{2\bar{r}^2 L^2}{(1-\gamma)^4}. $$
    As $\eta_m \langle \nabla J(\theta(m)), e^2(m) \rangle$ is a martingale difference sequence w.r.t. $\mathcal{G}_m$, we have by Azuma Hoeffding bound, with probability at least $1-\frac{1}{2}\delta$, 
    $$\Big| \sum_{m=0}^{M-1} \eta_m \langle \nabla J(\theta(m)), e^2(m) \rangle \Big| \leq \frac{2\bar{r}^2 L^2}{(1-\gamma)^4} \sqrt{2\sum_{m=0}^{M-1} \eta_m^2 \log\frac{4}{\delta} } . $$
\qed

\subsection{Proof of Lemma~\ref{appendix:actor:lem:error_h_grad}}\label{appendix:actor:subsec:error_h_grad}
By \eqref{appendix:actor:eq:grad_i}, we have
\begin{align*}
    \nabla_{\theta_i} J(\theta(m))& = \sum_{t=0}^\infty \E_{s\sim\pi_t^{\theta(m)}, a\sim \zeta^{\theta(m)}(\cdot|s)} \left[ \gamma^t Q^{\theta(m)}(s,a) \nabla_{\theta_i}\log\zeta^{\theta(m)}(a|s)\right]\\
    &= \sum_{t=0}^\infty \E_{s\sim\pi_t^{\theta(m)}, a\sim \zeta^{\theta(m)}(\cdot|s)} \left[ \gamma^t Q^{\theta(m)}(s,a) \nabla_{\theta_i}\log\zeta_i^{\theta_i(m)}(a_i|s_i)\right]
\end{align*}
where we have used $\nabla_{\theta_i}\log\zeta^{\theta(m)}(a|s) =\nabla_{\theta_i} \sum_{j\in\mathcal{N}} \log\zeta_j^{\theta_j(m)}(a_j|s_j)=\nabla_{\theta_i}\log\zeta_i^{\theta_i(m)}(a_i|s_i)$. Also recall the definition of $h_i(\theta)$ in \eqref{appendix:actor:eq:h_i}, 
$$ h_i(m) =  \sum_{t=0}^T \E_{s\sim\pi^{\theta(m)}_t,a \sim \zeta^{\theta(m)}(\cdot|s)}\left[ \gamma^t \frac{1}{n} \sum_{j\in {\nik}} {Q}_j^{\theta(m)} (s,a)\nabla_{\theta_i} \log\zeta_i^{\theta_i(m)} (a_i|s_i)\right].$$
The rest of the proof is essentially the same as Lemma~\ref{lem:truncated_pg}. For completeness we provide a proof below. Combining the above two equations, we have,
\begin{align*}
        &\nabla_{\theta_i} J(\theta(m)) - h_i(m) \\
        &= \sum_{t=0}^T \mathbb{E}_{s \sim\pi^{\theta(m)}_t, a\sim \zeta^{\theta(m)}(\cdot|s)}\left[ \gamma^t\nabla_{\theta_i}  \log \zeta_i^{\theta_i(m)}(a_i|s_i)  \left(  Q^{\theta(m)}(s,a)  - \frac{1}{n}\sum_{j\in \nik} Q_j^{\theta(m)}(s,a)\right)\right]\\
        &\quad + \sum_{t=T+1}^\infty  \mathbb{E}_{s \sim\pi^{\theta(m)}_t, a\sim \zeta^{\theta(m)}(\cdot|s)}\left[\gamma^t  \nabla_{\theta_i}  \log \zeta_i^{\theta_i(m)}(a_i|s_i)  Q^{\theta(m)}(s,a)\right] \\
        &:= E_1 + E_2.
        \end{align*}
   Clearly, the second term satisfies $\Vert E_2\Vert \leq \frac{ L_i \bar{r}}{(1-\gamma)^2}\gamma^{T+1} $. For $E_1$, we have
        \begin{align*}
        E_1&=\sum_{t=0}^T \mathbb{E}_{s \sim\pi^{\theta(m)}_t, a\sim \zeta^{\theta(m)}(\cdot|s)}\left[ \gamma^t\nabla_{\theta_i}  \log \zeta_i^{\theta_i(m)}(a_i|s_i) \left(\frac{1}{n} \sum_{j\in \nminusik} {Q}_j^{\theta(m)}(s,a) \right)\right]\\
        &= \sum_{t=0}^T \mathbb{E}_{s \sim\pi^{\theta(m)}_t, a\sim \zeta^{\theta(m)}(\cdot|s)}\left[ \gamma^t\nabla_{\theta_i}  \log \zeta_i^{\theta_i(m)}(a_i|s_i)   \frac{1}{n}\sum_{j\in \nminusik} \left( Q_j^{\theta(m)}(s,a)- \hat{Q}_j^{\theta(m)}(s_{\njk},a_{\njk}) \right)\right] \\
        &\qquad + \sum_{t=0}^T \mathbb{E}_{s \sim\pi^{\theta(m)}_t, a\sim \zeta^{\theta(m)}(\cdot|s)}\left[ \gamma^t\nabla_{\theta_i}  \log \zeta_i^{\theta_i(m)}(a_i|s_i) \frac{1}{n} \sum_{j\in \nminusik}  \hat{Q}_j^{\theta(m)}(s_{\njk},a_{\njk})\right]\\
        &:= E_3 +E_4,
    \end{align*}
 where $\hat{Q}_j^{\theta(m)}$ is any truncated $Q$ function for $Q_j^{\theta(m)}$ as defined in \eqref{eq:truncated_q}. We claim $E_4$ is zero. To see this, consider for any $j\in \nminusik$ and any $t$, 
    \begin{align*}
        &\mathbb{E}_{s\sim\pi^{\theta(m)}_t, a\sim \zeta^{\theta(m)}(\cdot|s)} \left[ \nabla_{\theta_i}  \log \zeta_i^{\theta_i(m)}(a_i|s_i)    \hat{Q}_j^{\theta(m)}(s_{\njk},a_{\njk})\right] \\
        &= \sum_{s,a} \pi^{\theta(m)}_t(s) \prod_{\ell=1}^n \zeta^{\theta_\ell(m)}_\ell(a_\ell|s_\ell) \frac{\nabla_{\theta_i}  \zeta_i^{\theta_i(m)}(a_i|s_i)}{\zeta_i^{\theta_i(m)}(a_i|s_i)} \hat{Q}_j^{\theta(m)}(s_{\njk},a_{\njk})\\
        &= \sum_{s,a} \pi^{\theta(m)}_t(s) \prod_{\ell\neq i} \zeta^{\theta_\ell(m)}_\ell(a_\ell|s_\ell) \nabla_{\theta_i}  \zeta_i^{\theta_i(m)}(a_i|s_i)  \hat{Q}_j^{\theta(m)}(s_{\njk},a_{\njk})\\
        &= \sum_{s,a_{1:i-1},a_{i+1:n}} \pi^{\theta(m)}_t(s) \prod_{\ell\neq i} \zeta^{\theta_\ell(m)}_\ell(a_\ell|s_\ell)  \hat{Q}_j^{\theta(m)}(s_{\njk},a_{\njk}) \sum_{a_i} \nabla_{\theta_i}   \zeta_i^{\theta_i(m)}(a_i|s_i) \\
        & = 0,
    \end{align*}
    where in the last equality, we have used $\hat{Q}_j^{\theta(m)}(s_{\njk},a_{\njk})$ does not depend on $a_i$ as $i\not\in \njk$; and $\sum_{a_i} \nabla_{\theta_i}   \zeta_i^{\theta_i(m)}(a_i|s_i) =  \nabla_{\theta_i}   \sum_{a_i} \zeta_i^{\theta_i(m)}(a_i|s_i) =\nabla_{\theta_i} 1 = 0 $. 
    
    For $E_3$, by the exponential decay property, the truncated $Q$ function has a small error, cf. \eqref{appendix:truncated:eq:q_err},  
    $$\sup_{s,a} | Q_j^{\theta(m)}(s,a)- \hat{Q}_j^{\theta(m)}(s_{\njk},a_{\njk})| \leq c\rhok,$$
    and as a result,
    $$\Vert E_3\Vert \leq \frac{1-\gamma^{T+1}}{1-\gamma} L_i c\rhok < \frac{L_i c }{(1-\gamma)}\rhok.$$
Therefore,
    
    \begin{align*}
       \Vert \nabla_{\theta_i} J(\theta(m)) - h_i(m)  \Vert& = \Vert E_2 + E_3\Vert \leq   \frac{ L_i \bar{r}}{(1-\gamma)^2}\gamma^{T+1}    +  \frac{L_i c }{(1-\gamma)}\rhok,\\
       &\leq 2 \frac{L_i c }{(1-\gamma)}\rhok,
 %      &= \frac{ nL \bar{r}}{(1-\gamma)^2} \big[2 \frac{1-\gamma^{T+1}}{1-\gamma} \gamma^{k+1}  + \gamma^{T+1} \big]\\
 %      &<  \frac{3 nL \bar{r}}{(1-\gamma)^3}  \gamma^{k+1}   \\
    \end{align*}
where in the last step, we have used $$T+1\geq \frac{\log \frac{c(1-\gamma)}{\bar{r}}+ (\khop+1)\log\rho}{\log \gamma} ,$$
and as a result, $\Vert \nabla J(\theta(m)) - h(m)  \Vert\leq  2 \frac{L c }{(1-\gamma)}\rhok$.
\qed